 \documentclass[11 pt]{article}
\usepackage{amssymb, amsmath, amsfonts,amsthm}
\usepackage[a4paper,vmargin={3.5cm,3.5cm},hmargin={2.5cm,2.5cm}]{geometry}
\usepackage{mathrsfs} 
\usepackage[usenames,dvipsnames,svgnames,table]{xcolor}
\usepackage{url} 
\usepackage{bbm} 
\usepackage{enumerate}
\usepackage{graphicx}
\usepackage{comment}

\usepackage[colorlinks,
            linkcolor=Blue,
            anchorcolor=red,
            citecolor=RoyalPurple,
            ]{hyperref} 

\usepackage[capitalize, nameinlink]{cleveref}
\usepackage{titlesec} 
\titleformat{\section}[block]
{\normalfont \large\bfseries}
{\thesection}{1.2em}{\bfseries}
\titleformat{\subsection}[block]
{\normalfont \normalsize \bfseries}
{\thesubsection}{0.9em}{\bfseries}
\titlespacing{\paragraph}{%
  0pt}{
  0.5\baselineskip}{
  1em}

\usepackage{setspace}
\makeatletter
\def\th@plain{%
  \thm@notefont{}
  \itshape 
}
\def\th@definition{%
  \thm@notefont{}
  \normalfont 
}
\makeatother

\newtheorem{theorem}{Theorem}[section]
\newtheorem{corollary}[theorem]{Corollary}
\newtheorem{lemma}[theorem]{Lemma}
\newtheorem{proposition}[theorem]{Proposition}
\newtheorem{remark}[theorem]{Remark}

\newtheorem{definition}[theorem]{Definition}
\numberwithin{equation}{section}

\newcommand{\bN}{\mathbb{N}}

\newcommand{\bR}{\mathbb{R}}
\newcommand{\ind}[1]{\mathbbm{1}_{\left\{ #1 \right\} } } 
\newcommand{\dd}{\mathrm{d}} 

\newcommand{\bE}{\mathbb{E}} 
\newcommand{\bP}{\mathbb{P}} 

\newcommand{\cA}{\mathcal{A}} 
\newcommand{\cB}{\mathcal{B}} 
\newcommand{\cR}{\mathcal{R}} 
\newcommand{\cG}{\mathcal{G}} 
\newcommand{\cE}{\mathcal{E}} %
\newcommand{\cD}{\mathcal{D}} 
\newcommand{\cF}{\mathcal{F}}
\newcommand{\cL}{\mathcal{L}} 
\newcommand{\cS}{\mathcal{S}} 
\newcommand{\cW}{\mathcal{W}} %
\newcommand{\cY}{\mathcal{Y}} 

\newcommand{\rF}{\mathscr{F}} 
\newcommand{\rG}{\mathscr{G}} 
\newcommand{\rK}{\mathscr{K}} 

\newcommand{\eqdis}{\overset{d}{=}}

\newcommand{\fN}{\mathbf{N}} 
\newcommand{\fe}{\mathbf{e}} 

\newcommand{\kL}{\mathfrak{L}} 
\newcommand{\ki}{\mathfrak i}

\newcommand{\cT}{\mathcal{T}} 
\newcommand{\besq}{{\rm BESQ}}

\title{Marked stable L\'evy processes in Brownian local time}
\author{
Elie A\"id\'ekon \footnote{School of Mathematical Sciences, Fudan University; {aidekon@fudan.edu.cn}}, Quan Shi \footnote{SKLMS, Academy of Mathematics and Systems Science, Chinese Academy of Sciences; {quan.shi@amss.ac.cn}}, Chengshi Wang\footnote{School of Mathematical Sciences, Fudan University; {cswang17@fudan.edu.cn}}
}
\date{\today}

\begin{document}

\maketitle
\vspace{-1em}
\begin{abstract}
    We construct a random partition of the (local) time-space  half-plane  $\bR_+\times \bR$, measurable with respect to a single Brownian motion. The construction is driven by two objects canonically associated with this Brownian path: a white noise driving a pair of coupled squared Bessel flows whose parameters differ by $\delta\in (0,2)$, and a continuous space-filling exploration of the half-plane. Exceptional bifurcation points of the coupled flows generate a random partition of the half-plane, up to a zero-measure gasket, whose widths are squared Bessel excursions with negative parameter $(-\delta)$. A suitable time-change of the Brownian exploration arranges these excursions, in exploration order, as marks attached to the jumps of a spectrally positive stable L\'evy process of index $(1+\delta/2)$.
    
As applications, we recover the interval partition evolutions of \cite{Paper1-1} as horizontal slices of our construction, and use this representation to prove the independence between their total mass evolution and normalized partition evolution, as conjectured in \cite{Paper1-2}. We also prove that the gasket of our partition is homeomorphic to the half-plane analogue of the stable shredded sphere introduced in \cite{BjCuSi22}.
\end{abstract}

 

\section{Introduction}

Given a white noise $\cW$ on $\bR_+\times \bR$, a squared Bessel (BESQ) flow with parameter $c\in \bR$ is a collection of processes  $(\cS_{r,x}(a),\, x\ge r)$ indexed by $(a,r)\in \bR_+\times \bR$, where each $(\cS_{r,x}(a),\, x\ge r)$ starts at $\cS_{r,r}(a)=a$ and is a solution of the stochastic differential equation (SDE) 
\begin{equation}
    \dd_x \cS_{r,x}(a) =  2 \cW([0,\cS_{r,x}(a)],\dd x) + c \, \dd x,\qquad x\ge r,
\end{equation}
\noindent with absorption at $0$ when $c\le 0$. 
It is a classical result that this SDE admits a unique strong solution, which is a so-called squared Bessel process of dimension $c$. 
 We refer to Section \ref{s:BESQ} for further background on BESQ flows.  BESQ flows first appeared in \cite{PitmYor82} and are connected among others to continuous-state branching processes \cite{bertoin-legall00, lambert02, dawson2012stochastic}, Ray--Knight theorems \cite{legall-yor, aidekon2024infinite,aidekon2023stochastic}, skew Brownian flows \cite{burdzy2001local,burdzy2004lenses,AWYskew}. 

In a recent work, \cite{aidekon2024infinite} uses  BESQ flows to give an alternative proof of the well-known Ray--Knight theorems \cite{Knight63,Ray63,legall-yor,CarmPetiYor94}. The results in that paper apply to the so-called perturbed reflecting Brownian motion (PRBM) \vspace{-0.1cm}
$$ X_t := |B_t| - \mu \kL_t,\quad t\ge 0,\vspace{-0.1cm}$$

\noindent where $\mu> 0$, $B=(B_t)_{t\ge 0}$ is a standard one-dimensional Brownian motion  and $(\kL_t)_{t\ge 0}$ is the continuous local time process of $B$ at position 0. A BESQ flow $\cR$ is related to the local times $(L(t,x))_{t\ge 0,x\in \bR}$ of $X$ via the formula
$$
\cR_{r,x}(a)=L(\tau_a^r,x), \qquad a\ge 0,\  x\ge r, 
$$

\noindent where
\[
    \tau_a^r := \inf\{t\ge0\,:\, L(t,r) > a\} 
\]

\noindent denote the right-continuous inverse local times of $X$. 
Set for every squared integrable $g: \bR_+\times \bR \rightarrow \bR$, 
\begin{align}\label{WN}
    W(g):= \int_0^{+\infty} g(L(t,X_t),X_t)\ {\rm sgn}(B_t)\dd B_t.\vspace{-0.1cm}
\end{align}
\noindent Then $W$ defines a white noise on $\bR_+\times\bR$ and $\cR$ is a BESQ flow driven by $W$ with the parameter $c$ varying in space: $\delta:= \frac{2}{\mu}$ for $x < 0$, and $0$ for $x \ge 0$. 

In this work, we employ this framework to explore the finer structure of  Brownian local time. Our analysis reveals a hidden  Poisson point process of $\besq$-excursions with negative parameters and leads to the appearance of a stable Lévy process. This was partially motivated from recent developments in interval partition evolutions \cite{Paper1-2,FPRW-Aldous} and certain random planar maps \cite{BjCuSi22}, where  analogous structures were observed. We thereby establish a perhaps unexpected   connection to these diverse contexts.

\paragraph{Coupled BESQ flows and a marked stable process}

Specifically, we consider two coupled BESQ flows driven by the same white noise. 
More precisely, let $\mu>1$ such that $\delta =\frac2\mu \in (0,2)$. 
 We consider both the BESQ flow $\cR$ (for red) given by the local time defined above, as well as a BESQ flow $\cB$ (for blue), also driven by $W$, with parameter $0$ for $x < 0$, and $-\delta$ for $x \ge 0$. 
For a fixed point $(a_0,r_0)\in\bR_+\times\bR$, almost surely, a single blue/red flow line starts at $(a_0, r_0)$ with $\cB\le \cR$. However, in relation with semi-flat bifurcation times in the skew Brownian flow \cite{burdzy2004lenses}, it is shown in \cite{AWYskew} that there exist exceptional points $(a,r)$ where flow lines of $\cR$ and $\cB$ simultaneously bifurcate and interlace, as in Figure~\ref{fig:flow}. At such a common bifurcation point $(a,r)$, the rightmost blue flow line temporarily lies between two red lines. The region delimited by this blue line and the red line at its left will be denoted by $\cT_{(a,r)}$ in \eqref{eq:spindle}. Any two such regions $\cT_{(a,r)}$ are either disjoint or one is contained in the other. We then consider maximal regions for the inclusion order, and refer to them as \textbf{spindles}, see Definition \ref{def:spindle}.

\begin{figure}[htbp]
\centering
       \scalebox{0.6}{
        \def\svgwidth{\columnwidth}
\begingroup%
  \makeatletter%
  \providecommand\color[2][]{%
    \errmessage{(Inkscape) Color is used for the text in Inkscape, but the package 'color.sty' is not loaded}%
    \renewcommand\color[2][]{}%
  }%
  \providecommand\transparent[1]{%
    \errmessage{(Inkscape) Transparency is used (non-zero) for the text in Inkscape, but the package 'transparent.sty' is not loaded}%
    \renewcommand\transparent[1]{}%
  }%
  \providecommand\rotatebox[2]{#2}%
  \newcommand*\fsize{\dimexpr\f@size pt\relax}%
  \newcommand*\lineheight[1]{\fontsize{\fsize}{#1\fsize}\selectfont}%
  \ifx\svgwidth\undefined%
    \setlength{\unitlength}{312.41851566bp}%
    \ifx\svgscale\undefined%
      \relax%
    \else%
      \setlength{\unitlength}{\unitlength * \real{\svgscale}}%
    \fi%
  \else%
    \setlength{\unitlength}{\svgwidth}%
  \fi%
  \global\let\svgwidth\undefined%
  \global\let\svgscale\undefined%
  \makeatother%
  \begin{picture}(1,0.39813776)%
    \lineheight{1}%
    \setlength\tabcolsep{0pt}%
    \put(0,0){\includegraphics[width=\unitlength,page=1]{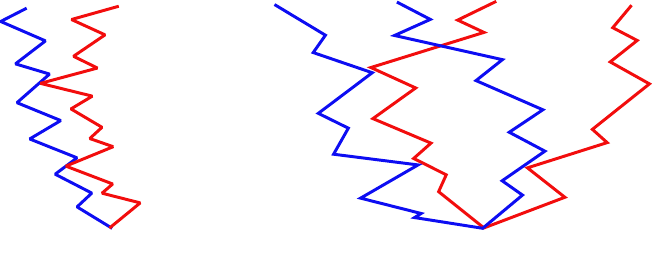}}%
    \put(0.11137114,0.00951943){\makebox(0,0)[lt]{\lineheight{1.25}\smash{\begin{tabular}[t]{l}\huge $(a_0,r_0)$\end{tabular}}}}%
    \put(0.69372474,0.00667868){\makebox(0,0)[lt]{\lineheight{1.25}\smash{\begin{tabular}[t]{l}\huge $(a,r)$\end{tabular}}}}%
    \put(0,0){\includegraphics[width=\unitlength,page=2]{flow.pdf}}%
    \put(0.66874434,0.21053342){\makebox(0,0)[lt]{\lineheight{1.25}\smash{\begin{tabular}[t]{l}\huge $\cT_{(a,r)}$\end{tabular}}}}%
  \end{picture}%
\endgroup%

    }
    \caption{
    Left: For a fixed point $(a_0,r_0)\in\bR_+\times\bR$, almost surely, a single blue/red flow line starts at $(a_0, r_0)$.
Right: 
There exist exceptional points $(a,r)$, where the blue and red flows bifurcate and the rightmost blue line temporarily lies between the two red lines. The shaded region $\cT_{(a,r)}$ is bounded by the rightmost blue and leftmost red lines. }
    \label{fig:flow}
\end{figure}

Spindles are disjoint, bounded subsets of the half-plane that (roughly) form a partition, in the sense that the \textbf{gasket $\mathscr{K}$}—the complement of all spindle interiors—has Lebesgue measure zero (see Corollary~\ref{c:spindle zero measure}).
The main results of this paper (Theorems~\ref{thm:PPP} and \ref{thm:Levy}) establish two key representations:
\begin{enumerate}
    \item The widths of spindles, viewed as a function of the vertical level, arise as a Poisson point process of $\besq$ excursions of parameter $-\delta\in (-2,0)$;
    \item spindles  are naturally embedded in the jumps of a spectrally positive stable Lévy process $\xi$ of index $1+\frac{\delta}{2}\in (1,2)$.  
\end{enumerate}

In other words, we can describe our main result as follows. Note that each spindle $\cT_{(a,r)}$ has a bottom point $(a,r)$ and a top point where its blue and red boundaries meet, call it $(c,z)$. We then call width process the process $(\cB_{r,x}(a)-\cR_{r,x}(a-),x\in [r,z])$ where $\cB_{r,\cdot}(a)$ denotes the blue boundary and $\cR_{r,\cdot}(a-)$ the red boundary. Our main result shows that there is a bijection between the spindles generated by $(\cB,\cR)$ and the jumps of a stable process $\xi$, such that 
the collection 
\[
\{(\xi(t-),\xi(t))\colon t>0,\ \Delta \xi(t)>0\}
\]

\noindent is the collection of pairs $(r,z)$ over all spindles. This bijection preserves the left-right order of the spindles, in the sense that, for any fixed level $y$, one spindle is to the left of another if and only if its corresponding jump in $\xi$ occurs at an earlier time. Moreover, conditionally on $\xi$, the width processes of the spindles are independent $\besq^{-\delta}$ excursions of duration  $\Delta \xi(t)$. See Figure~\ref{fig:skewer} for an illustration.

\begin{figure}[htbp]
\centering
       \scalebox{0.9}{
        \def\svgwidth{\columnwidth}
\begingroup%
  \makeatletter%
  \providecommand\color[2][]{%
    \errmessage{(Inkscape) Color is used for the text in Inkscape, but the package 'color.sty' is not loaded}%
    \renewcommand\color[2][]{}%
  }%
  \providecommand\transparent[1]{%
    \errmessage{(Inkscape) Transparency is used (non-zero) for the text in Inkscape, but the package 'transparent.sty' is not loaded}%
    \renewcommand\transparent[1]{}%
  }%
  \providecommand\rotatebox[2]{#2}%
  \newcommand*\fsize{\dimexpr\f@size pt\relax}%
  \newcommand*\lineheight[1]{\fontsize{\fsize}{#1\fsize}\selectfont}%
  \ifx\svgwidth\undefined%
    \setlength{\unitlength}{502.88701576bp}%
    \ifx\svgscale\undefined%
      \relax%
    \else%
      \setlength{\unitlength}{\unitlength * \real{\svgscale}}%
    \fi%
  \else%
    \setlength{\unitlength}{\svgwidth}%
  \fi%
  \global\let\svgwidth\undefined%
  \global\let\svgscale\undefined%
  \makeatother%
  \begin{picture}(1,0.40888825)%
    \lineheight{1}%
    \setlength\tabcolsep{0pt}%
    \put(0,0){\includegraphics[width=\unitlength,page=1]{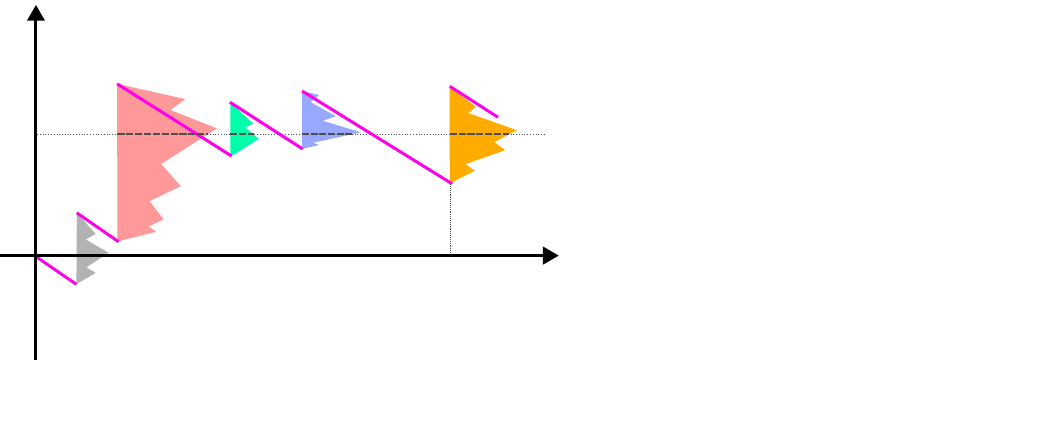}}%
    \put(0.42310871,0.13423515){\makebox(0,0)[lt]{\lineheight{1.25}\smash{\begin{tabular}[t]{l}\large $t_i$\end{tabular}}}}%
    \put(0.00831804,0.27677876){\makebox(0,0)[lt]{\lineheight{1.25}\smash{\begin{tabular}[t]{l}\large $y$\end{tabular}}}}%
    \put(0,0){\includegraphics[width=\unitlength,page=2]{skewer.pdf}}%
    \put(0.00735042,0.39709652){\makebox(0,0)[lt]{\lineheight{1.25}\smash{\begin{tabular}[t]{l}\large $\xi$\end{tabular}}}}%
    \put(0.87668949,0.20963774){\makebox(0,0)[lt]{\lineheight{1.25}\smash{\begin{tabular}[t]{l}\large $(a,r)$\end{tabular}}}}%
    \put(0.63420367,0.27489849){\makebox(0,0)[lt]{\lineheight{1.25}\smash{\begin{tabular}[t]{l}\large $y$\end{tabular}}}}%
    \put(0.80430547,0.0487253){\makebox(0,0)[lt]{\lineheight{1.25}\smash{\begin{tabular}[t]{l}\large $\cB_{r,y}(a)-\cR_{r,y}(a-)$\end{tabular}}}}%
    \put(0,0){\includegraphics[width=\unitlength,page=3]{skewer.pdf}}%
    \put(0.27377904,0.05107341){\makebox(0,0)[lt]{\lineheight{1.25}\smash{\begin{tabular}[t]{l}\large $f_i(y-\xi(t_i-))$\end{tabular}}}}%
  \end{picture}%
\endgroup%

    }
    \caption{On the top left picture, we draw a schematic representation of the process $\xi$ colored in purple, with its jumps decorated by independent $\besq^{-\delta}$ excursions $f_i$ filled with different colors. On the top right picture, we draw some spindles generated by the red and blue flows $\cR$  and $\cB$. The widths of the spindles at level $y$, i.e. $\cB_{r,y}(a)-\cR_{r,y}(a-)$, are given by the values of the $\besq^{-\delta}$ excursions at the corresponding level, i.e. $f_i(y-\xi(t_i-))$.}
    \label{fig:skewer}
\end{figure}

\noindent With this interpretation, we can perform more explicit calculations for the random partition. For instance,  we deduce that (Corollary~\ref{cor:stb-sub}), at every level $y\in \mathbb{R}$, the intersection of the gasket with the horizontal line,  $(\bR_+\times \{y\}) \cap \mathscr{K}$, is the range of a $\frac\delta2$-stable subordinator with Laplace exponent $\lambda \to \lambda^{\frac\delta2}$.  Following the approach in \cite{Lambert10,LambertBravo16}, the appearance of the Lévy process further reveals a branching structure within the spindles. Roughly speaking, we can interpret the family of spindles as a Crump-Mode-Jagers branching process,  where each spindle generates offspring according to a Poisson point process, and the resulting genealogy is described by a $(1+\frac\delta2)$-stable L\'evy tree \cite{DuLG05}.  

Our random partition of the space-time (in local time) is also reminiscent of the study in \cite{PitmYor82}, where Pitman and Yor construct a $\besq^c$ process with $c\ge 0$ as a sum of a Poisson point process of $\besq^0$-excursions. In view of the Ray--Knight theorem, this corresponds to decomposing the Brownian path into excursions. 
In \cite{PW18}, a $\besq^{-\delta}$ process is embedded within the Brownian local time by decomposing the Brownian path via a related skew Brownian motion.

\paragraph{Two limiting objects of discrete random structures}

Marked Lévy processes were employed in \cite{Paper1-1,Paper1-2} to construct a family of \emph{self-similar interval-partition-valued diffusion processes}, which were further generalized to a two-parameter family in \cite{FRSW2}.
They arise as scaling limits of up-down Markov chains on integer compositions \cite{RogersWinkel,RivRiz23,ShiWinkel-2}. Their projection onto the Kingman simplex is related to Petrov’s diffusion \cite{Paper1-3,Petrov09}. Moreover, in the monograph \cite{FPRW-Aldous}, interval partition evolutions serve as fundamental tools for constructing a continuum-tree-valued diffusion with stationary distribution given by the Brownian continuum random tree, solving a conjecture originally posed by \cite{AldousDiffusionProblem}. 
Such marked Lévy processes are also used in \cite{Paper0} to give a uniform Hölder-continuous approximation of the local time of a stable process. These structures were later used in \cite{PaperFV,Paper1-3} to build a two-parameter family of measure-valued processes with Pitman–Yor stationary distributions, whose existence was conjectured by \cite{FengSun10}, thereby extending Ethier and Kurtz’s \emph{infinitely-many-neutral-alleles models} \cite{EthKurtzBook}. 

Our representation of spindles reveals a fundamental connection between BESQ flows and interval partition evolutions. In  Theorems~\ref{thm:IPE-0} and~\ref{thm:IPE-alpha} we recover the self-similar interval-partition evolutions in \cite{Paper1-2}.   Our new approach allows applying the powerful machinery of stochastic flows to study problems concerning interval partition evolutions. 
As a first application, by applying Perkins' disintegration theorem developed for flows in 
\cite{aidekon2023stochastic}, we prove that, informally speaking, the total length of an interval partition evolution and the relative proportion of the intervals, evolve independently; see Theorem~\ref{thm:IPE-independence} for the rigorous statement. This confirms Conjecture~2 in \cite{Paper1-2}.

Our second application concerns  random planar maps which are the dual maps of a hard multimer model defined by Di
Francesco and Guitter \cite{difrancesco_guitter}. These planar maps are a generalized version of causal dynamical triangulations in which lozenges (made up of two triangles) are replaced with faces of arbitrarily large degrees. While the Gromov--Hausdorff limit of causal dynamical triangulations is proved to be a line segment \cite{CuHuNa},  its generalized version converges to some non-degenerate object called  \emph{stable shredded sphere}  \cite{BjCuSi22}, which is encoded by an excursion of the stable process $\xi$. The result of our paper shows that the gasket $\rK$ is homeomorphic to the half-plane analogue of the stable shredded sphere, see Theorem \ref{t:homeomorphism}.

\paragraph{Strategy of the proof}

Let us briefly describe the appearance of the stable process $\xi$. The process
$$
\varrho_t:=(L(t,X_t),X_t),\quad t\ge 0
$$

\noindent can be thought of as an explorer of $\bR_+\times \bR$ (Forman \cite{Formanbrick} names $\varrho$ a \emph{Brownian bricklayer} in the Brownian motion case). When $X$ is a Brownian motion, it is shown in \cite{Barlow} that the process $L(t,X_t)$ is not a semimartingale, while in \cite{beffara} the process $\varrho$ is seen as a scaling limit of an ${\rm SLE}_\kappa$ process as $\kappa\to \infty$. 
The process $\varrho$ is continuous, space-filling and does not cross itself. Although the time spent by $\varrho$ in the gasket $\rK$ has Lebesgue measure zero, we perform a time-change for $\varrho$, which effectively acts as an inverse local time for the gasket, to track $\varrho$ only when it lies on the gasket. This construction is carried out in two steps in Section \ref{s:time-change} and Section \ref{sec:markedLevy}.  Via this time-change we  transform the process $X$ into the L\'evy process $\xi$; see Theorem \ref{thm:Levy}. Each jump of $\xi$ is associated with the exploration of a spindle by $\varrho$, when it enters the spindle through the bottom and leaves it through the top. 
The size of the jump of $\xi$ is given by the height of this corresponding spindle. 
Note that, at the moment of exit through the top, the spindle has not yet been fully explored. We refer to Section \ref{s:explore} for a more detailed description of the exploration of spindles. Finally, we associate to each spindle its width process, which is the distance between its blue boundary and its red boundary as one goes from the bottom to the top of a spindle.  Additive properties of the BESQ processes hint that this process should be a  ${\rm BESQ}^{-\delta}$ excursion. We show that it is indeed the case in Theorem \ref{thm:PPP}.

\paragraph{Analogy with imaginary geometry}
The structural properties of our model bear a striking resemblance to the imaginary geometry  developed in a series of works by Duplantier, Miller and Sheffield in \cite{DuMiSh,MillerSheffieldI,MillerSheffieldIV}. To illustrate this analogy, we informally summarize   some features of that setting. Note that in the Liouville quantum gravity setting, one needs to measure the space through the Liouville measure, which, in our setting, is simply the Lebesgue measure. Therefore we omit the discussion on the Liouville measure below.
\begin{enumerate}[(i)]
\item \cite{DuMiSh,MillerSheffieldI,MillerSheffieldIV} A space filling SLE explores the plane. The boundary of the domain visited  up to a given time is formed by imaginary flow lines, which are measurable with respect to an underlying Gaussian free field (GFF). Both the SLE explorer and the  GFF are measurable with respect to a pair of correlated Brownian motions $(B^1,B^2)$.
\item \cite{borga,gwynne_holden_sun} The interaction between the SLE explorer and imaginary flow lines of different angles gives rise to  flows of skew Brownian motions driven by $(B^1,B^2)$.
\item \cite{MiSh19} To a real parameter $\theta \in (0,\pi)$, one can define a {\it light cone} of opening angle $\theta$, whose complementary components are called {\it pockets}. 
The SLE explorer visits a pocket by entering its opening point, leaving by its closing point,  then making excursions into and outside the pocket. 
\end{enumerate}

The papers \cite{aidekon2024infinite,aidekon2023stochastic} give an analog of (i) in which the role of the triplet (SLE, GFF, $(B_1,B_2)$) is played by ($\varrho$, $\cW$, $B$). The counterpart of statement (ii) for the BESQ flow corresponds to \cite{AWYskew}. In this case, the related flow of skew Brownian motions previously appeared in \cite{burdzy2001local,burdzy2004lenses}, though unconnected to the BESQ flow at the time. The present paper is the analog of (iii) where spindles play the role of the SLE pockets and the gasket $\mathscr{K}$ that of the light cone.
  We mention that the imaginary geometry  was instrumental in the recent result  \cite{LiuZhuang} to prove that the adjacency graph of the non-filling ${\rm SLE}_\kappa$ process is disconnected for $\kappa$ close to $8$, addressing a question from \cite{DuMiSh}. Connectedness for $\kappa$ near $4$ was established  in \cite{GwynnePfeffer}. A similar question for the adjacency graph of the stable shredded sphere was raised in \cite{BjCuSi22}, where it is noted that the arguments of \cite{GwynnePfeffer} can be adapted to show connectedness for $\delta$ close to $0$. It would be interesting to see whether our representation could provide insights into the regime $\delta$ close to $2$.

\paragraph{Outlook on continuum‑tree‑valued Markov processes}

The Brownian continuum random tree (BCRT) \cite{AldousCRT1} is a random rooted real tree with a non‑atomic probability measure on its leaves. Aldous conjectured \cite{AldousDiffusionProblem} that, for certain Markov chains on discrete trees, a scaling limit exists to yield a diffusion on  continuum trees, stationary with respect to the BCRT. In a recent work \cite{FPRW-Aldous}, the conjectured limiting process is constructed using interval partition evolutions. The connection with interval partitions is built via the spinal decomposition: taking the unique path from the root to a random leaf, the masses of subtrees along the spine form an interval partition; as the tree evolves, their dynamics is the interval partition evolution. A similar operation can be performed for \(k\) distinct leaves (a multi‑spinal projection). Ultimately, one describes the full tree evolution by obtaining a consistent family: the \((k+1)\)-system is obtained by decomposing the evolution of a single interval in the \(k\)-system (corresponding to the mass evolution of a certain subtree) into a finer interval partition evolution (corresponding to a further spinal decomposition of that subtree). However, the construction in \cite{FPRW-Aldous} only gives the continuum-tree-valued process starting from the stationary state. It remains an open problem to identify a state space for this continuum‑tree‑valued Markov process and to characterize the transition kernels (Problem 8.20 in \cite{FPRW-Aldous}).

Using our construction, it becomes possible to encode the above consistent systems by a single Brownian motion path, through its corresponding white noise and the associated BESQ flows. Specifically, in our framework the interval partition evolutions are obtained as the horizontal slices of spindles (see Section~\ref{sec:constr}). This representation naturally admits a self‑similar refinement: within each spindle the same white noise structure reappears due to the decomposition property. This leads to an infinitely nested hierarchical system. Such a perspective may facilitate a more direct analytic study of the Aldous diffusion and could provide a useful foundation for addressing Problem 8.20 in \cite{FPRW-Aldous}. We will further explore this aspect in greater depth in a forthcoming paper.

\paragraph{Embedding of stable trees}

As indicated in the proof sketch, we will show that a time-change turns the process $X$ into a spectrally positive stable process $\xi_\alpha$ of index $\alpha=1+\frac{\delta}{2}$. Since the positive part of $X$ behaves as a reflecting Brownian motion, our result demonstrates that there exists a family of time-changes indexed by $\alpha \in (1,2]$ that transforms a Brownian excursion into an excursion of the process $\xi_\alpha$. Note that these time-changes are measurable with respect to the Brownian excursion, but are not stopping times, so are not Skorokhod embeddings. Our paper therefore establishes the following coupling between $\alpha$-stable trees. Consider the BCRT  encoded by the Brownian excursion. Our spindle partition shows how to measurably select a subset $\mathcal G_\alpha$ of the Brownian tree. The tree $\mathscr T_\alpha$ obtained by contracting each connected component of $\mathcal G_\alpha$ to a point is precisely the stable tree encoded by $\xi_\alpha$. Observe that by the monotonicity in $\delta$ of the ${\rm BESQ}^\delta$ flow lines, the sets $\mathcal G_\alpha$ are decreasing in $\alpha$. It would be interesting to study further the properties of the family $(\mathscr T_\alpha)_{1<\alpha\le 2}$.  
 Other couplings of stable trees exist in the literature.  Curien and Haas \cite{CurienHaas} show that stable trees  form a nested family of trees decreasing in $\alpha$, while  Rembart and Winkel \cite{RembartWinkel23} embed stable trees into a binary (non-Brownian) continuum random tree. Closer to the spirit of our coupling, Le Gall \cite{LeGall18} constructs a stable tree of parameter $3/2$ as the subordinate tree of a BCRT with respect to the past maximum process of the Brownian motion on the BCRT, which also amounts to contracting subtrees in the middle of the Brownian tree.

\paragraph{Organization of the paper}
The paper is organized as follows. Section \ref{s:review} reviews results on BESQ flows. Section~\ref{s:spindle} introduces spindles in the coupled flows $\cR$ and $\cB$ and describe the exploration of spindles. Section~\ref{sec:main} proves the main results Theorems~\ref{thm:PPP} and \ref{thm:Levy}, giving the embedding of the BESQ excursions in the stable L\'evy process $\xi$. Finally, Section \ref{sec:constr} shows the connections with interval partition evolutions of  \cite{Paper1-1,Paper1-2,FRSW2} and stable shredded spheres of \cite{BjCuSi22}.  

\bigskip

We suppose that all the $\sigma$-fields appearing in the paper are complete and filtrations are right-continuous.


\section{BESQ Flows}
\label{s:review}
We first recall some properties of BESQ flows, closely following the exposition in \cite{AWYskew}. We will sometimes apply results of \cite{AWYskew} to ${\rm BESQ}(\overline{\delta})$ flows instead of ${\rm BESQ}^\delta$ flows, using \cite[Section~2.5]{AWYskew}.

\subsection{Preliminaries on BESQ Flows}
\label{s:BESQ}

Let $\delta \ge 0$.  The squared Bessel process of dimension $\delta$ started at $a\ge 0$, denoted by ${\rm BESQ}^\delta_a$,  is the  unique solution of
\begin{equation}\label{eq:besselprocess}
S_x = a+2\int_0^x \sqrt{|S_r|} \, \dd B_r + \delta x, \qquad x\ge 0,
\end{equation}

\noindent where
$B$ is a standard Brownian motion. The ${\rm BESQ}^\delta$ process hits zero at a positive time if and only if $\delta< 2$. It is absorbed at $0$ when $\delta= 0$ and is  reflecting
at $0$ when $\delta\in (0,2)$. When $\delta<0$, the ${\rm BESQ}^\delta_a$ process is solution of \eqref{eq:besselprocess} but is absorbed at $0$ when it touches it, which happens in a finite time a.s. We refer to \cite{GoinYor03} for details and more properties of squared Bessel processes.

\bigskip
\begin{definition}[BESQ flow; {\cite[Definition 3.3]{aidekon2023stochastic}, \cite[Definition 2.2]{AWYskew}}]
\label{def:BESQflow}
Let $\delta \in \bR$ and ${\mathcal W}$ be a white noise on $\bR_+\times \bR$.  We call ${\rm BESQ}^\delta$ flow  driven by ${\mathcal W}$ a collection ${\mathcal S}$ of continuous processes $({\mathcal S}_{r,x}(a),\, x\ge r)_{r\in \bR,a\ge 0}$ such that: 
\begin{enumerate}[(1)]
\item
for each $(a,r)\in \bR_+\times \bR$, the process $({\mathcal S}_{r,x}(a),\, x\ge r)$ is almost surely the unique strong solution of the following SDE  
\begin{equation}\label{eq:BESQflow}
{\mathcal S}_{r,x}(a) = 
a + 2 \int_r^x  {\mathcal W}([0,{\mathcal S}_{r,s}(a)], \dd s) + 
\delta (x-r),\qquad x\ge r,
\end{equation}
where, if $\delta\le 0$, the process is absorbed when hitting $0$.
 \item Almost surely,
 \begin{enumerate}[(i)]
 \item for all $r\in \bR$ and $a\ge 0$, $ {\mathcal S}_{r,r}(a)=a$,
 \item for all $r\le x$, $a\mapsto  {\mathcal S}_{r,x}(a)$ is c\`adl\`ag,
 \item (Perfect flow property) for any $r\le x \le y$ and $a\ge 0$, $\cS_{r,y}(a)=\cS_{x,y}\circ \cS_{r,x}(a)$. 
\end{enumerate} 
\end{enumerate}

We call killed ${\rm BESQ}^\delta$ flow driven by $\cW$ the flow obtained from the ${\rm BESQ}^\delta$ flow by absorbing at $0$ the flow line $\cS_{r,\cdot}(a)$ at time $\inf\{x>r\,:\,\cS_{r,x}(a)=0\}$. 

We call general ${\rm BESQ}^\delta$ flow  a  ${\rm BESQ}^\delta$ flow or a  killed ${\rm BESQ}^\delta$ flow.   
\end{definition}

When $\delta\notin(0,2)$ there is no difference between the killed ${\rm BESQ}^\delta$ flow and the (non-killed) ${\rm BESQ}^\delta$ flow lines. Indeed when $\delta\le 0$, flow lines touching $0$ stay at $0$, while when $\delta\ge 2$, flow lines cannot touch $0$, except at their starting point. When $\delta\in (0,2)$, flow lines are reflected at $0$ in the ${\rm BESQ}^\delta$ flow, whereas they are absorbed at $0$ in the killed version. Note that when $\delta\in(0,2)$ there are exceptional times $r$ when $\cS_{r,x}(0)>0$ for all $x>r$ close enough to $r$. At such time $r$, the flow line of a killed ${\rm BESQ}^\delta$ flow with $\delta\in (0,2)$ starting from $(0,r)$ leaves $0$ before eventually coming back and being absorbed.

The difference between flow lines of BESQ flows behave as squared Bessel processes, see Section~\ref{s:decomp}. We can also consider BESQ flows with varying dimensions.
For $\cF = (\cF_t, t\in \bR)$  a filtration, a $\cF$-predictable process $\overline{\delta}:\bR\to\bR$ is called a drift function if there exists a deterministic vector $(\delta_i)_{1\le i\le n} \in \bR^n$ and (possibly random) numbers $-\infty=\tilde{t}_1<\tilde{t}_2<\ldots<\tilde{t}_{n+1}=\infty$ such that $\overline{\delta}=\delta_i$ on $A_i:=(\tilde{t}_{i},\tilde{t}_{i+1})$.
The drift functions in this paper will be of the form  $\overline{\delta}= \delta_1 \mathbbm{1}_{(-\infty,r)}+\delta_2 \mathbbm{1}_{(r,\infty)}$ or $\overline{\delta}:= \delta_1 \mathbbm{1}_{(-\infty,r_1)}+\delta_2 \mathbbm{1}_{(r_1,r_2)} + \delta_3 \mathbbm{1}_{(r_2,\infty)}$. We will write them as $\delta_1 \,|_r\, \delta_2$ and $\delta_1 \,|_{r_1}\, \delta_2\,|_{r_2}\, \delta_3$ respectively. (We may have $r_1=r_2$, in which case $\delta_1 \,|_{r_1}\, \delta_2\,|_{r_2}\, \delta_3 := \delta_1 \,|_{r_1}\,  \delta_3$. If $r=\infty$, $\delta_1\,|_r\, \delta_2=\delta_1$.) 
The following generalization of Definition \ref{def:BESQflow} is taken from \cite{AWYskew}. 

\begin{definition}\cite[Definition 2.15]{AWYskew}\label{def:vary}
   Let $\cW$ be a white noise with respect to some filtration $\cF$.  Let $\overline{\delta}$ be a drift function with respect to ${\mathcal F}$.  
   We call a collection of non-negative processes $\cS=(\cS_{r,x}(a),\, x\ge r)_{r\in \bR,a\ge 0}$ a ${\rm BESQ}(\overline{\delta})$ flow driven by $\cW$ if 
   
    (1) The restriction of $\cS$ to each $A_i$ is a ${\rm BESQ}^{\delta_i}$ flow driven by $\cW$.

    (2) The regularity conditions of Definition \ref{def:BESQflow} hold.

If $E$ is a (possibly random) Borel set, we call $\rm BESQ(\overline{\delta})$ flow killed on $E$ the flow obtained from the $\rm BESQ(\overline{\delta})$ flow by absorbing the flow line $\cS_{r,\cdot}(a)$ at $0$ at time $\inf\{x \in E\cap (r,\infty) \,:\, \cS_{r,x}(a)=0\}$.

 If $E=\bR$, we will simply call it killed  $\rm BESQ(\overline{\delta})$ flow. We call general  $\rm BESQ(\overline{\delta})$ flow a $\rm BESQ(\overline{\delta})$ flow or a  killed $\rm BESQ(\overline{\delta})$ flow.
\end{definition}

The following proposition collects results of \cite[Section 3 \& Appendix A]{aidekon2023stochastic} and \cite[Section~2.3]{AWYskew}.
\begin{proposition}\label{p:perfect}(\cite{aidekon2023stochastic, AWYskew})
Let $\cS$ be a general ${\rm BESQ}(\overline{\delta})$ flow  driven by ${\mathcal W}$. 
\begin{enumerate}[(i)]

 \item ((Almost) \footnote{The term \emph{almost} indicates that we restrict to the case where $\cS_{r,x}(a)>0$, thereby excluding the exceptional points mentioned earlier, which may occur for certain $x>r$ with $\cS_{r,x}(a)=0$  when $\cS$ a killed ${\rm BESQ}^{\delta}$ flow with parameter $\delta\in (0,2)$. This restriction is superfluous when $\cS$ is non-killed by definition.} 
 perfect flow property) Almost surely, for every $r\le x\le y$, $a\ge 0$ with $\cS_{r,x}(a)>0$, $\cS_{r,y}(a)=\cS_{x,y}\circ \cS_{r,x}(a)$.

 \item (Coalescence)\footnote{If $\cS$ is non-killed, it is also true when $\max(r,r')= x$ by the perfect flow property.} Almost surely, for $r,r'< x$ and $0\le a,a'$,  if $\cS_{r,x}(a)=\cS_{r',x}(a')$, then $\cS_{r,y}(a)=\cS_{r',y}(a')$ for all $y\ge x$. 

 \item (Instantaneous coalescence) Let $\cD$ be a countable dense subset of $\bR_+\times\bR$.  Almost surely, for every $(a,r)\in \bR_+\times\bR$ and $\varepsilon>0$, there exists some $(a_n,r_n)\in \cD$ such that $\cS_{r,x}(a)=\cS_{r_n,x}(a_n)$ for all $x\geq r+\varepsilon$.

 \item (Comparison principle) Let $\cS'$ be a ${\rm BESQ}(\overline{\delta'})$ flow driven by the same white noise ${\mathcal W}$ with $\overline{\delta'}\ge \overline{\delta}$.  Almost surely, for all $(a,r)\in \bR_+\times\bR$ and $x\ge r$, $\cS_{r,x}(a)\le \cS'_{r,x}(a)$. We will write it for short $\cS\le \cS'$. 
\end{enumerate}   
\end{proposition}

By \cite[Proposition 2.4 (i)]{AWYskew}, almost surely, for any $r<x$ and $a\ge 0$
\begin{equation}\label{eq:rightcont}
\cS_{r,x}(a')=\cS_{r,x}(a)  \textrm{ for all } a'>a \textrm{ close enough to }a. 
\end{equation}

For a white noise $\cW$ and a flow $\cS=(\cS_{r,x}(a))$ and a real number $y$, we let $\theta_y\cW$ be $\cW$ translated by $y$ in space, i.e. the white noise
\begin{align}\label{eq:WN translation}
   (\theta_y\cW)(A\times[r,x]) := \cW(A\times[r-y,x-y]),\quad A \text{ a Borel subset of } \bR_+,\, x\ge r
\end{align}
and $\theta_y\cS$, called $\cS$ translated by $y$ in space, denote the flow 
\begin{align}\label{eq:S translation}
   \theta_y \cS_{r,x}(a) := \cS_{r-y,x-y}(a),\quad a\ge 0,\, x\ge r.
\end{align}

\noindent For a function $g:\mathbb{R}\to \mathbb{R}_+$ and a flow $\cS$, we denote by $\cS-g$ the flow
\begin{equation}\label{eq:S-g}
    \max(\cS_{r,x}(a+g(r))-g(x),0),\qquad r\le x,\, a\ge 0.
\end{equation}

\paragraph{Left-continuous version of a BESQ flow.}

Let $\cS$ be a ${\rm BESQ}(\overline{\delta})$ flow. We define the left-continuous version of  $\cS$ as, for any $a>0$ and $x\ge r$,
\begin{equation}\label{def:left S}
\cS_{r,x}(a-):= \lim_{a'\uparrow a} \cS_{r,x}(a'). 
\end{equation}

\noindent The coalescence property of $\cS$ together with \cite[Proposition 2.4 (ii)]{AWYskew} shows that almost surely, for every $(a,r)\in (0,\infty)\times\bR$ and $x>r$, one can find $a'<a$ such that  
\begin{equation}\label{eq:leftcont}
\cS_{r,y}(a-)=\cS_{r,y}(a'),\ \forall\, y\ge x.
\end{equation}
This identity implies that a flow line of the left-continuous version of $\cS$ coalesces with flow lines of $\cS$  outside its starting point in the sense that almost surely for any $r<x$, $a>0$, $r'\le x$ and $a'\ge 0$, 
\begin{equation}\label{eq:coal_leftcont}
\cS_{r,x}(a-)=\cS_{r',x}(a') \Rightarrow \cS_{r,y}(a-)=\cS_{r',y}(a'),\; \forall\,y\ge x.
\end{equation}
We next justify that almost surely for every $a>0$, $b>0$ and $r\le x\le y$,
\begin{equation}
\label{eq:perfectflow-}
\cS_{r,x}(a-)\le b \Rightarrow \cS_{r,y}(a-)\le \cS_{x,y}(b-).
\end{equation}

\noindent One can suppose without loss of generality that $r<x<y$.  One can find $a'<a$  such that   $\cS_{r,s}(a')=\cS_{r,s}(a-)$ for any $s\ge x$. Suppose first that $\cS_{r,x}(a-)< b$.  Let $b'\in (\cS_{r,x}(a'),b)$ such that $\cS_{x,y}(b')=\cS_{x,y}(b-)$. The coalescence property of $\cS$ shows that $\cS_{r,y}(a')\le \cS_{x,y}(b')$ which completes the proof. In the case $\cS_{r,x}(a-)=b$, the perfect flow property of $\cS$ shows that $\cS_{r,y}(a')=\cS_{x,y}(b)$ hence $\cS_{r,y}(a-)=\cS_{x,y}(b)$ which is also $\cS_{x,y}(b-)$ by  Proposition \ref{c: no middle bifurcation} below. 
The last argument is a rephrasing of the (almost) perfect flow property for the left-continuous version of $\cS$.  This flow also satisfies the other properties of Proposition \ref{p:perfect}. Properties (ii) and (iii) follow from the corresponding properties of $\cS$ while  the comparison principle holds by taking the limit in \eqref{def:left S}.

\medskip

In Proposition \ref{c:measurability}, we will need to examine the measurability of a flow line starting from a random point w.r.t. some filtration. To this end, we will use Propositions \ref{p:decomp_coalescence} and \ref{p:decompdual_coalescence}, both of which are derived from the next result.

\begin{proposition}
    Let $\cS$ be a ${\rm BESQ}(\overline{\delta})$ flow and $\cD$ be a countable dense subset of $\bR_+\times\bR$. Almost surely for every $(a,r)\in \bR_+\times \bR$ and $x> r$,
    \begin{equation}\label{eq:coalescence}
\cS_{r,x}(a) =\lim_{\substack{(a_n,r_n)\in \cD \to (a,r)\\ r_n\in (r,x),a_n\ge \cS_{r,r_n}(a)}} \cS_{r_n,x}(a_n)=\limsup_{\substack{(a_n,r_n)\in \cD \to (a,r) \\ r_n\in (r,x)}} \cS_{r_n,x}(a_n).
\end{equation}
If moreover  $a>0$,
\begin{equation}
\cS_{r,x}(a-) =\lim_{\substack{(a_n,r_n)\in \cD \to (a,r)\\ r_n\in (r,x), a_n\le \cS_{r,r_n}(a-)}} \cS_{r_n,x}(a_n)=\liminf_{\substack{(a_n,r_n)\in \cD \to (a,r) \\ r_n\in (r,x)}} \cS_{r_n,x}(a_n). \label{eq:coalescence-}
\end{equation}
\end{proposition}

\begin{proof}
    Let $(a,r)\in \bR_+\times \bR$ and $x>r$. Let $a'>a$ such that $\cS_{r,x}(a')=\cS_{r,x}(a)$. Take a sequence $(a_n,r_n)\in \cD$ which converges to $(a,r)$, such that $r_n\in (r,x)$ and $a_n\ge \cS_{r,r_n}(a)$. For $n$ large enough,  $a_n\in [\cS_{r,r_n}(a),\cS_{r,r_n}(a'))$, hence $\cS_{r_n,x}(a_n)=\cS_{r,x}(a)$ by coalescence. It shows the first identity of \eqref{eq:coalescence}. If $r_n\in (r,x)$ and $a_n\le \cS_{r,r_n}(a)$, then $\cS_{r_n,x}(a_n)\le \cS_{r,x}(a)$, hence we get the second equality.
    
    We turn to the proof of \eqref{eq:coalescence-}. By \eqref{eq:leftcont} one can find $a'<a$  such that $\cS_{r,x}(a')=\cS_{r,x}(a-)$. By \eqref{eq:coal_leftcont}, $\cS_{r_n,x}(a_n)=\cS_{r,x}(a-)$ if $r_n\in (r,x)$ and $a_n \in (\cS_{r,r_n}(a'),\cS_{r,r_n}(a-)]$. 
     If $a_n\ge \cS_{r,r_n}(a-)$, then $\cS_{r_n,x}(a_n)\ge \cS_{r,x}(a-)$. It completes the proof.
\end{proof}

\paragraph{Duality and bifurcation} 
  
  For a white noise ${\mathcal W}$ on $\bR_+\times \bR$, let ${\mathcal W}^*$ be its image by the map $(a,r)\mapsto (a,-r)$.
 
\begin{proposition}(\cite[Proposition 3.8]{aidekon2023stochastic}, \cite[Proposition 2.17]{AWYskew})\label{p:BESQdual vary}
    Let $\overline{\delta}\colon \bR \to \bR$ be a deterministic drift function and $\overline{\delta}^*(x):=\overline{\delta}(-x)$, $x\in \bR$. If ${\mathcal S}$ is a $\rm BESQ(\overline{\delta})$ flow (resp.\ killed $\rm BESQ(\overline{\delta})$ flow), driven by ${\mathcal W}$, then its dual
    \begin{equation}\label{def:dual}
    \big( \cS^*_{r,x}(a) := \inf\{ b>0: \cS_{-x,-r}(b)>a \}, \; x\ge r\big), \quad r\in \bR, a\ge 0
    \end{equation}
    is a killed $\rm BESQ(2-\overline{\delta}^*)$ flow (resp.\ a   $\rm BESQ(2-\overline{\delta}^*)$ flow), driven by $-{\mathcal W}^*$.
\end{proposition}

We refer to \cite{Yaolin} for a general statement on the duals of stochastic flows. Forward flow lines (i.e. flow lines of $\cS$) and dual flow lines (i.e. flow lines of $\cS^*$) do not cross. More precisely, we have the following result which is \cite[Proposition 2.10]{AWYskew} together with the discussion in \cite[Section 2.5]{AWYskew}.

\begin{proposition}(\cite[Proposition 2.10]{AWYskew})\label{p:properties dual}
 Let $\overline{\delta}\colon \bR \to \bR$ be a drift function, $\cS$ be a general $\rm BESQ(\overline{\delta})$ flow
 and $\cS^*$ be its dual as in \eqref{def:dual}. The following statements hold almost surely.
\begin{enumerate}[(i)]
\item For any $r<x$, $b\ge 0$ and $a\ge \cS^*_{-x,-r}(b)$,  we have $\cS_{r,x}(a)>b$ and 
\begin{equation}\label{eq:properties dual+}
    \cS_{r,y}(a)\ge \cS^*_{-x,-y}(b),\qquad  y\in [r,x].
\end{equation}
Moreover  $\cS_{r,y}(a)> \cS^*_{-x,-y}(b)$ if $y\in (r,x)$ and $\cS_{r,y}(a)>0$.
\item  For any $r<x$, $b\ge 0$ such that $a:=\cS^*_{-x,-r}(b)>0$ and $0\le a'< a$, 
\[
\cS_{r,y}(a')\le \cS^*_{-x,-y}(b),\qquad  y\in [r,x].
\]
 Moreover $\cS_{r,y}(a')< \cS^*_{-x,-y}(b)$ if  $y\in (r,x)$  and $\cS^*_{-x,-y}(b)>0$.
\end{enumerate}
\end{proposition}

\noindent The dual flow is related to the following bifurcation events. 
\begin{definition}\label{def:bifur}
We call a point $(a, r) \in (0,\infty)\times\bR$ a \textbf{bifurcation point} if $\cS_{r,x}(a) > \cS_{r,x}(a-)$ for some  $x > r$. 
We say that a bifurcation point $(a,r)$ is \textbf{an ancestor of $(b,x)$} if $\cS_{r,x}(a-) \le b <\cS_{r,x}(a)$. 
\end{definition}

\noindent Note that since $\cS_{r,x}(a-)$ coalesces with flow lines of $\cS$ outside its starting point, the  bifurcation  may happen only at the beginning.

 \begin{lemma}[{\cite[Proposition 2.9]{AWYskew}}]
      Almost surely, for any $a>0$, $b\ge 0$ and $r<x$,
\begin{equation}\label{equivalence bifurcation}
   (a,r) \textrm{ is an ancestor of } (b,x)  \quad \iff \quad  a=\cS^*_{-x,-r}(b). 
\end{equation}
 \end{lemma}

\noindent The following proposition says that flow lines of a BESQ flow cannot meet a bifurcation point in $(0,\infty)\times \bR$ outside their starting point. 

\begin{proposition}[{\cite[Corollary 2.11]{AWYskew}}]\label{c: no middle bifurcation}
    Let  $\cS$ be a general ${\rm BESQ}(\overline{\delta})$ flow.  Almost surely, for any $a\ge 0$ and $r<x$ such that $\cS_{r,x}(a)>0$, the point $(\cS_{r,x}(a),x)$ is not a bifurcation point for $\cS$. 
\end{proposition}

\subsection{Decomposition of BESQ flows}
\label{s:decomp}

In this section, we divide the space-time plane $\bR_+\times \bR$ into left and right components with respect to a flow line $Y$ as in \cite[Section 3]{AWYskew}. We give properties of the restrictions of a flow $\cS$  to each component, denoted respectively by $\cS^-$ and $\cS^+$. We then state a similar result when $Y$ is a dual flow line.

\medskip

We use the setting of \cite[Chapter 2]{Walsh1986AnIT}. Let $\cW$ be a white noise on $\bR_+\times \bR$ with respect to some right-continuous filtration  $\cF:=(\cF_x,\, x\in \bR)$.   Following \cite{AWYskew}, if $Y=(Y_r,\,r\ge r_0)$ is a non-negative predictable process,  
we introduce the martingale measures  $\cW^-_Y$ and $\cW^+_Y$ defined by
\begin{align}
 \label{eq:W-}   \cW^-_Y(\dd \ell, \dd x) &:= \cW(\dd \ell, \dd x),\, && \ell\le Y_x,\ x\ge r_0, \\
\label{eq:W+}    \cW^+_Y(\dd \ell, \dd x)  &:= \cW(Y_x+\dd \ell, \dd x), && \ell\ge 0,\ x\ge r_0.
\end{align}

\noindent In other words,  for any Borel set $A\subset \bR_+$ with finite Lebesgue measure, and any $r_0\le x\le y$,
\begin{align*}
    (\cW_Y^-)_{x,y}(A) &:= \int_{A\times [x,y]} \mathbbm{1}_{ [0,Y_r]}(\ell)\cW(\dd \ell,\dd r),
    \\
    (\cW_Y^+)_{x,y}(A) &:= \int_{\bR_+\times [x,y]} \mathbbm{1}_{A}(\ell-Y_r) \cW(\dd \ell, \dd r).
\end{align*}

\noindent  See Figure~\ref{fig:decomp} for an illustration. Observe that $\cW_Y^+$ is a white noise on $\bR_+\times[r_0,\infty)$. The case $b=0$ of the following result follows from \cite[Proposition 3.3]{AWYskew}. The general case $b\ge 0$ is similar. The outline of the proof is given in Section \ref{s:proof decomp} for completeness.

\begin{proposition}\label{p:decomp}
    Let  $\delta_1\ge 0$, $\delta_2,\delta_3\in \bR$, $b\ge 0$ and $r_0,z\in \bR$. Let $Y$ be the ${\rm BESQ}(\delta_1\,|_z\, \delta_2)$ flow line starting from $(b,r_0)$ driven by $\cW$ and $\cS$ be the ${\rm BESQ}(\delta_1+\delta_3 \,|_z\, \delta_2+\delta_3)$ flow driven by $\cW$. Write $\zeta$ for the absorption time of $Y$, i.e. $\zeta:=\inf\{r\ge \max(r_0,z)\,:\, Y_r =0\}$ if $\delta_2<0$ and $\zeta=\infty$ otherwise.
    
    (i) The flow $\cS^+$ defined by
\begin{equation}\label{eq:decomp+}
\cS^+_{r,x}(a):= \max(\cS_{r,x}(a+Y_r)-Y_x,0),\qquad r_0\le r\le x, a\ge 0
\end{equation}

\noindent is on $\bR_+\times [r_0,+\infty)$ a ${\rm BESQ}(\delta_3\,|_{\zeta} \,\delta_2+\delta_3 )$ flow driven by $\cW_Y^+$. 

(ii) Both $Y$ and the collection of processes $\cS^-$ defined by  
\begin{equation}\label{eq:decomp-}
\cS^-_{r,x}(a):= 
\min(\cS_{r,x}(a),Y_x), \quad  r_0\le r\le x , 0\le a< Y_r 
\end{equation}

\noindent are driven by the martingale measure $\cW_Y^-$. In particular, they are measurable w.r.t. $\cW_Y^-$. 

(iii)  $\cW_Y^+$ is a white noise independent of $\cW_Y^-$.
\end{proposition}

We give the analogue of this proposition when $Y$ is a dual line, i.e.\  a non-negative process predictable with respect to the natural filtration of $\cW^*$. The proof is given in Section \ref{s:proof decomp}.
Let $\cW^{*,-}_Y$ and $\cW^{*,+}_Y$ be defined as in \eqref{eq:W-} and \eqref{eq:W+} respectively, with $\cW^*$ replacing $\cW$. Then define $\cW^-_Y$ and $\cW^+_Y$ as the images of $\cW^{*,-}_Y$ and $\cW^{*,+}_Y$, respectively, under the map $(a,x)\mapsto (a,-x)$.

\begin{proposition}\label{p:dual S+}
Let $\delta_1^* \ge 0$, $\delta_2^*\in \bR$, $z\ge r_0$, $b\ge 0$ and $Y$ be the ${\rm BESQ}(\delta^*_1\,|_z\, \delta_2^*)$ flow line starting at $(b,r_0)$ driven by $-\cW^*$ and killed on $(z,\infty)$.    Let $\delta_3\ge 0$ such that $\delta_3\ge 2-\delta_1^*$ and $\cS$ be the ${\rm BESQ}(2-\delta_2^*-\delta_3\,|_{-z}\,2-\delta_1^*-\delta_3)$ flow driven by $\cW$.  Consider the collection of processes defined by
 \[
\cS^{+}_{r,x}(a) :=\max(\cS_{r,x}(a+ Y_{-r})- Y_{-x},0), \qquad a\ge 0,\, r\le x\le -r_0
\]
and 
 \[
\cS^{-}_{r,x}(a) :=
\cS_{r,x}(a), \quad  -\zeta^*<r\le x\le -r_0,\,\quad 0\le a < Y_{-r}, 
\]

\noindent where $\zeta^*$ is the absorption time of $Y$, i.e. $\zeta^*:=\inf\{r\ge z\,:\, Y_r=0\}$ if $\delta_2^*<2$ and $\zeta^*=\infty$ otherwise.

    (i) The flow  $\cS^{+}$ is on $(-\infty,-r_0)$ a  ${\rm BESQ}(2-\delta_2^*-\delta_3\,|_{-\zeta^*}\, {2-\delta_3})$ flow driven by $\cW^{+}_Y$, killed on $(-\zeta^*,-r_0)$. 

    (ii) $\cS_{r,x}^-(a)\le Y_{-x}$ for any $ -\zeta^*<r\le x\le -r_0,\, 0\le a<Y_{-r}$.
    
    (iii)  Both $Y$ and $\cS^{-}$  are measurable w.r.t. $\cW_Y^{-}$.

    (iv) $\cW_Y^+$ is a white noise independent of $\cW_Y^-$.
\end{proposition}

See Figure~\ref{fig:decomp} for an illustration of Propositions~\ref{p:decomp} and \ref{p:dual S+}.

 \begin{figure}[htbp]
\centering
       \scalebox{0.7}{ 
        \def\svgwidth{\columnwidth}
        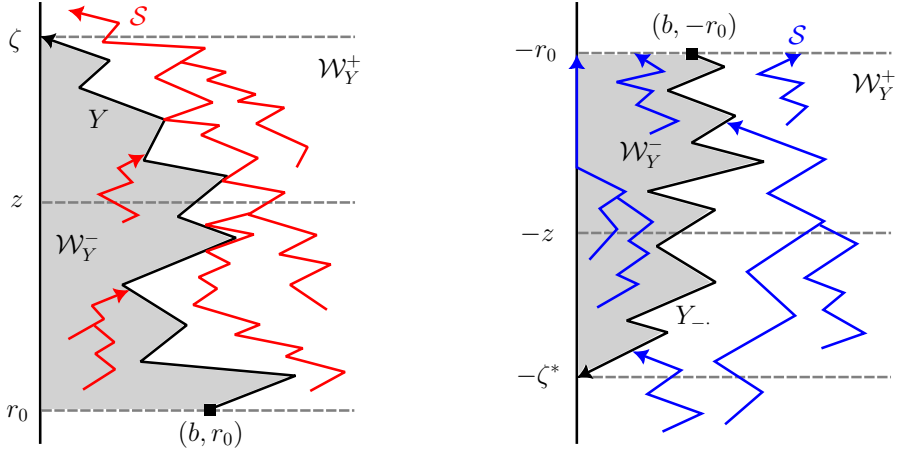
    }

\caption{The left picture describes the decomposition of the white noise $\cW$ and the flow $\cS$ by a (forward) line $Y$ as discussed in Proposition~\ref{p:decomp} in the case $\delta_2\le 0,\, \delta_3>0$. In this case flow lines in  $\cS^-$ will be equal to $Y$ once they hit $Y$, while flow lines in  $\cS^+$ are reflected at 0 in $[r_0,\zeta]$. The shaded area represents the noise $\cW_Y^-$.\\
The right picture illustrates the decomposition from a line $Y$ driven by $-\cW^*$ as discussed in Proposition~\ref{p:dual S+}. As shown in the proposition, flow lines in $\cS^+$ are killed at 0 between $-\zeta^*$ and $-r_0$, while flow lines of $\cS^-$ stay in the shaded area.}
\label{fig:decomp}
\end{figure}

\section{Spindles in coupled BESQ flows}\label{s:spindle}
  The purpose of this section is to partition $\bR_+ \times \bR$ into disjoint regions delimited by suitably chosen coupled BESQ flows, which are referred to as \emph{the blue and red lines}. Each region is called a \emph{spindle} and we will show that the spindles form a Poisson point process of Pitman--Yor $\besq^{-\delta}$-excursions with $\delta \in (0,2)$. 

\subsection{Coupled BESQ flows constructed from a PRBM}\label{s:PRBM}

Let $B=(B_t)_{t\ge 0}$ be a standard one-dimensional Brownian motion  and $(\kL_t)_{t\ge 0}$ the continuous local time process of $B$ at position 0. The perturbed reflecting Brownian motion (PRBM) with parameter $\mu$, also called $\mu$-process, is the process $X$ defined as
$$ X_t := |B_t| - \mu \kL_t,\quad t\ge 0.$$
 When $\mu=1$, $X$ is distributed as a standard Brownian motion by L\'evy's identity. Let $\rF$ be the natural filtration of $X$.
 
 Suppose $\mu>0$ so that $X$ is indeed recurrent. We will repeatedly use the fact that $X$ does not have increasing times since it behaves as a Brownian motion away from its infimum. Almost surely, for any $0\le t<t'$ such that $X_t<X_{t'}$, 
\begin{equation}\label{no_increase}
x\in (X_t,X_{t'}) \Rightarrow L(t,x)<L(t',x).
\end{equation}

\noindent However, when $\mu>1$, $X$ almost surely has decreasing times \cite[Proposition 6]{perman1997perturbed}, which are necessarily times when $X$ equals its running infimum. The semimartingale $X$ admits a bicontinuous version of its local times $L(t,x)$ at time $t$ and position $x$, defined via
$$ L(t,x) = \lim_{\varepsilon\rightarrow 0}\frac{1}{\varepsilon} \int_0^t \ind{x\le X_s\le x+\varepsilon}\dd s.$$
Let then $\tau_a^r:=\inf\{t\geq 0: L(t,r)>a\}$ be the right-continuous inverse local time of $X$.  The generalized Ray--Knight theorem \cite{CarmPetiYor94} gives a connection between the local times of $X$ and BESQ processes. Let \vspace{-0.1cm}
\[
\delta:=\frac2\mu.
\]
In the notation of Section \ref{s:BESQ}, for $r\le 0$, $(L(\tau_a^r,\,r-h),h\ge 0)$ is a ${\rm BESQ}^{2-\delta}_a$ process absorbed at $0$, and  $(L(\tau_a^r,\,r+h),h\in [0,|r|])$ is a ${\rm BESQ}^{\delta}_a$ process on $[0,|r|]$. The local times $L(\tau_a^r,h)$ when $r,h\ge 0$ are the local times of a standard Brownian motion hence satisfy the standard Ray--Knight theorems, i.e. for $r\ge 0$, $(L(\tau_a^r,\,r+h),h\ge 0)$ is a ${\rm BESQ}^{0}_a$ process, and  $(L(\tau_a^r,\,r-h),h\in [0,r])$ is a ${\rm BESQ}^{2}_a$ process on $[0,r]$. It is known that such Ray--Knight theorems are connected to stochastic flows appearing in the context of continuous-state branching processes \cite{bertoin-legall00,lambert02}. Using the framework of Dawson and Li \cite{dawson2012stochastic}, the paper \cite{aidekon2024infinite} shows that this connection can be expressed in the following way. Recall the white noise $W$ defined via \eqref{WN}. Set 
\begin{equation}\label{def:S}
    \cR_{r,x}(a):=L(\tau_a^r,x),\qquad \cR^*_{r,x}(a):=L(\tau_a^{-r},-x)
\end{equation}

\noindent for all $a\ge 0$ and $r\le x$. The flow $\cR^*$ is the dual of the flow $\cR$ as in \eqref{def:dual} by \cite[Proposition 2.5]{aidekon2023stochastic}. The following result is a consequence of \cite[Theorem 5.1]{aidekon2024infinite}. Note that the part when $X$ is positive behaves as a Brownian motion, so that we can apply the results of \cite{aidekon2024infinite} with $\mu=1$ on the positive half-line. We let $W^*$ be the image of $W$ by the map $(a,r)\mapsto (a,-r)$.
\begin{theorem}(\cite[Theorem 5.1]{aidekon2024infinite})\label{t:rayknight}
Recall that $\delta=\frac2\mu$. 

(i) The collection $(\cR_{r,x}(a),\,x\ge r)_{(a,r)\in \bR_+\times\bR}$ is a ${\rm BESQ}( \delta \,|_0\, 0)$ flow,  driven by $W$.

(ii) The collection  $(\cR^*_{r,x}(a),\,x\ge r)_{(a,r)\in \bR_+\times\bR}$ is a  killed ${\rm BESQ}( 2 \,|_0\, 2-\delta)$ flow driven by $-W^*$. 
\end{theorem}
\vspace{-0.1cm}

 \noindent The fact that the flow in (i) satisfies the regularity conditions of Definition \ref{def:vary} comes for example from \cite[Propositions 2.1 \& 2.6]{aidekon2023stochastic}. Statement (ii) follows from (i) by duality, using Proposition \ref{p:BESQdual vary}.

If $\sigma_a^r$ denotes the left-continuous inverse local times of $X$, i.e.
\[
\sigma_a^r := \sup\{t\ge0\,:\, L(t,r)< a\}, \, a>0\vspace{-0.1cm}
\]

\noindent then for all $r\le x$, and $a> 0$, in the notation \eqref{def:left S},
\[
L(\sigma_a^r,x)=\cR_{r,x}(a-).\vspace{-0.1cm}
\]

\noindent Excursions of $X$ at level $r$ correspond to $a$ such that $\sigma_a^r<\tau_a^r$. If $a>0$ and the excursion on $(\sigma_a^r,\tau_a^r)$ is above $r$, then  $(a,r)$ is a bifurcation point for $\cR$ as in Definition \ref{def:bifur}. Note that one can have excursions with $a=0$, but we omit their discussion since such points won't be used in this paper.    
 
\medskip

 Fix $(b,s)\in\bR_+\times \bR$ and let $T:=\tau_b^s$ for brevity. We aim to establish a decomposition result similar to those of Section~\ref{s:decomp}, this time dividing the space-time plane using the line $x\in \bR\to L(T,x)$. In the notation of Section~\ref{s:decomp},
    we define two martingale measures $(W^T, W^{T,-})$ as follows:
\begin{itemize}\vspace{-0.1cm}
    \item On $(s, \infty)$, $(W^T, W^{T,-}) = (W_Y^+, W_Y^-)$ with $Y = \cR_{s,\cdot}(b) = L(T, \cdot)$. 
    \item On $(-\infty, s)$, $(W^T, W^{T,-}) = (W_Y^+, W_Y^-)$ with $Y = \cR^*_{-s,\cdot}(b) = L(T, -\cdot)$.
\end{itemize}
   Therefore, we can rewrite them as  \vspace{-0.1cm}
\begin{align}
 \label{eq:WT-}    W^{T,-}(\dd \ell, \dd x) &= W(\dd \ell, \dd x),\, & \ell\le L(T,x),\; x\in \bR, \\
\label{eq:WT}    W^T(\dd \ell, \dd x)  &= W(L(T,x)+\dd \ell, \dd x), & \ell\ge 0,\; x\in \bR.
\end{align}
Recall that $\rF= (\rF_t)_{t\ge 0}$ denotes the natural filtration of the Brownian motion $B$. We let 
\begin{equation}\label{def:I}
I_t:=\inf_{[0,t]} X=-\mu \kL_t,\quad M_t:=\sup_{[0,t]} X.
\end{equation}

\begin{proposition}\label{p:markov S}
     Fix $(b,s)\in\bR_+\times \bR$. Let $T:=\tau_b^s$ and define $(W^{T,-},W^{T})$ as above. Let $\delta'\le 0$ and $\cS'$ be a ${\rm BESQ}( \delta+\delta'\,|_0\,\delta')$ flow driven by $W$.
   
    (i) $W^T$ is a white noise  independent of $W^{T,-}$. In particular, $I_T$ is independent of $W^T$.

    (ii) We have equality of the $\sigma$-fields $\sigma(W^{T,-})=\rF_T$.

    (iii) The      flow $\cS' - L(T,\cdot)$ in the notation \eqref{eq:S-g} 
\noindent is a ${\rm BESQ}(\delta+\delta' \,|_{I_T}\,2+\delta'\,|_s\, \delta')$ flow driven by $W^T$ killed on $(I_T,s)$. 

    (iv) Let $L^*(T,r):=L(T,-r)$ for all $r\in \bR$. The    flow $\cR^* - L^*(T,\cdot)$ in the notation \eqref{eq:S-g}  is a killed ${\rm BESQ}(2 \,|_{-s}\, 0\,|_{-I_T}\, 2-\delta)$ flow driven by $-W^{T,*}$, where $W^{T,*}$ is the image of $W^T$ by the map $(a,r)\mapsto (a,-r)$.
\end{proposition}

\begin{proof}
(i)  The first statement is a consequence of Proposition \ref{p:decomp} (iii) and Proposition \ref{p:dual S+} (iv). Since $-I_T$ is the hitting time of $0$ by the process $x\mapsto \cR^*_{-s,x}(b)$, it is measurable with respect to $W^{T,-}$ by Proposition \ref{p:dual S+} (iii).

(ii) By definition of $W^{T,-}$, for any Borel set $A\subset \bR_+$ with finite Lebesgue measure, and any $s\le  x\le y$,
$$
    (W^{T,-})_{x,y}(A) = \int_{A\times [x,y]} \mathbbm{1}_{ [0,L(T,r)]}(\ell) W(\dd \ell,\dd r). 
$$ 

\noindent Using \eqref{WN} with $g(\ell,r):=\mathbbm{1}_{ A\times[x,y]}(\ell,r)\mathbbm{1}_{[0, L(T,r)]}(\ell)$ there, noting that, for Lebesgue-a.e.\ $t>0$, 
\[
g(L(t,X_t),X_t)=\mathbbm{1}_{ A\times[x,y]}(L(t,X_t),X_t)\mathbbm{1}_{\{t\in [0,T]\}},
\]
\noindent we deduce that  $(W^{T,-})_{x,y}(A)$ is also equal to $\int_0^T  \mathbbm{1}_{ A\times[x,y]}(L(t,X_t),X_t) {\rm sgn}(B_t) \dd B_t$; see \cite[equation (3.8)]{aidekon2024infinite} with $b<\min(s,0)$ there. This equality is also true when $x\le y\le s$ following the arguments of  \cite[Proposition 3.2]{aidekon2024infinite} (which readily contains the case $s\le 0$). We deduce the inclusion  $\sigma(W^{T,-})\subset \sigma(X_t,\,0\le t\le T)$. The reverse inclusion can be seen from the fact that, by \cite[Remark 2.8]{aidekon2023stochastic}, $(X_t,\, 0\le t\le T)$ is measurable w.r.t. $L(T,\cdot)$ and $(\cR_{r,x}(a),\,a<L(T,r),\, r\le x)$.

(iii) The ${\rm BESQ}(\delta+\delta' \,|_{I_T}\,2+\delta'\,|_s\, \delta')$ flow driven by $W^T$ is well-defined since $I_T$ is independent of $W^T$. 
 We consider $(s,\infty)$.
    We apply Proposition \ref{p:decomp} to $Y_r=\cR_{s,r}(b)=L(T,r)$ and $\cS'$ in place of $\cS$, hence $z=0$, $\delta_1=\delta$, $\delta_2=0$, $\delta_3=\delta'$.     We find  by Proposition \ref{p:decomp} (i) that on $(s,\infty)$, $\cS'-L(T,\cdot)$ is a ${\rm BESQ}^{\delta'}$ flow driven by $W_Y^+=W^T$.
We reason similarly on $(-\infty,s)$. 
We apply Proposition \ref{p:dual S+} (i) to $Y_r=\cR^*_{-s,r}(b)$ and  $\cS'$ in place of $\cS$, hence $z=0$, $\delta_1^*=2$, $\delta_2^*=2-\delta$, $\delta_3=-\delta'$. Then $\cS' - L(T,\cdot)$ on $(-\infty,s)$ is a ${\rm BESQ}(\delta+\delta'\,|_{I_T}\,2+\delta')$ driven by $W^T$ killed on $(I_T,s)$. We conclude by the perfect flow property of  $\cS'-L(T,\cdot)$ at level $s$.

(iv) The proof is similar to (iii) so we feel free to omit it.
\end{proof}

We note that $I_T=s$ almost surely in the case $s<0$ and $b=0$. Recall the definition of a drift function at the beginning of Section~\ref{s:BESQ}. Thus in this case $\delta_1 \,|_{I_T}\,\delta_2\,|_s\, \delta_3 = \delta_1 \,|_s\, \delta_3$ for every $\delta_1,\delta_2,\delta_3\in\bR$. We point out that Proposition \ref{p:markov S} still holds in this case since it is a degenerate case in which there is no need to introduce the dual flow line.

\bigskip

\begin{definition}[Red and blue lines]
\label{def:red-blue}
 Recall that the white noise  $W$ is given by \eqref{WN}. The flow $\cR$  was defined in \eqref{def:S}. It is a ${\rm BESQ}(\delta\,|_0 \, 0)$ flow driven by $W$ by Theorem~\ref{t:rayknight}. We let $\cB=(\cB_{r,x}(a),\, x\ge r)_{(a,r)\in \bR_+\times \bR}$ be the ${\rm BESQ}(0\,|_0\, -\delta)$ flow driven by $W$. 
     We refer to flow lines of $\cR$ as \emph{red lines} and of $\cB$ as \emph{blue lines}.
\end{definition}

We will often consider the left-continuous version of $\cR$ in the notation \eqref{def:left S}, which is $\cR_{r,x}(a-)=L(\sigma_a^r,x)$. 
 By the comparison principle stated in Proposition \ref{p:perfect}, $\cB \le \cR$. Recall the definition of dual flows in \eqref{def:dual}. The flow $\cR^*$  of equation \eqref{def:S} is the dual of $\cR$, and we let $\cB^*$  denote the dual of $\cB$. We know from Proposition \ref{p:BESQdual vary} that $\cB^*$ is a ${\rm BESQ}(2+\delta\, |_0 \, 2)$ flow and $\cR^*$ is a killed ${\rm BESQ}(2\,|_0\, 2-\delta)$ flow. From the comparison principle, $\cR^*\le \cB^*$. We finish this section by the following simple observation. 
\begin{lemma}\label{l:basic blue red}
   Almost surely, for any $r<x$ and $b\ge 0$, we have $\cB^*_{-x,-y}(b)<\cB_{r,y}(a)$ for all $y\in (r,x)$ where $a=\cB^*_{-x,-r}(b)$.
\end{lemma}
 \begin{proof}
It is a consequence of  Proposition~\ref{p:properties dual} (i), noting that $\cB^*_{-x,-y}(b)>0$ for any $y<x$ and $b\ge 0$. 
 \end{proof}

\subsection{Ancestors and spindles}

For $(a,r)\in (0,\infty)\times \bR$, let $\cT_{(a,r)}$ be the region 
\begin{equation}\label{eq:spindle}
    \cT_{(a,r)}:=\{(b,x)\in \bR_+\times (r,\infty)\,:\, \cR_{r,x}(a-) \le b< \cB_{r,x}(a)\}.
\end{equation}
 For any fixed point $(a,r)\in (0,\infty)\times \bR$, we have $\cT_{(a,r)}=\emptyset$ a.s. However, we are interested in the exceptional points where $\cT_{(a,r)}\ne \emptyset$.  
 The existence of  such exceptional points is characterized by \cite[Theorem 2.13]{AWYskew}, with dependence on the parameter difference $\delta$ between the ${\rm BESQ}(\delta\,|_0 \, 0)$ flow $\cR$ and the ${\rm BESQ}(0\,|_0\, -\delta)$ flow $\cB$.  Note that we have restricted our discussion to $a > 0$: when $a = 0$, as all flow lines of $\cB$ starting at zero remain identically  zero,  we always have $\cT_{(a,r)} = \emptyset$. 

\begin{lemma}[{\cite[Theorem 2.13]{AWYskew}}]
 The set of exceptional points $\{(a,r)\in (0,\infty)\times \bR\colon \cT_{(a,r)}\neq \emptyset\}$ is non-empty if and only if $\delta \in (0,2)$. Moreover, when this condition holds, it has Hausdorff dimension  $\min(2-\delta,\frac{3-\delta}{2})$ a.s.
\end{lemma}

In the sequel, we always assume that 
\begin{equation}\label{eq:mu}
    \mu> 1 \quad \text{and thus} \quad  \delta:=\frac2\mu\in (0,2). 
\end{equation}

\begin{lemma}\label{lem:cT}
The following statements hold almost surely, for all $(a,r),(a',r')\in (0,\infty)\times \bR$ and $(b,x)\in \bR_+\times (r,\infty)$: 
\begin{enumerate}[(i)]
\item  We have $(b,x)\in \cT_{(a,r)}$ if and only if $(a,r)$ is an ancestor of $(b,x)$ for both $\cB$ and $\cR$. In particular, there is the identity
\begin{equation}\label{eq:spindle equiv}
    \cT_{(a,r)}=\{(b,x)\in \bR_+\times (r,\infty)\,:\, a=\cB^*_{-x,-r}(b)=\cR^*_{-x,-r}(b)\}.
\end{equation} 

\item   If $(b,x)\in \cT_{(a,r)}$, then the points  $(\cR^*_{-x,-y}(b),y)$ and $(\cB^*_{-x,-y}(b),y)$ for $y\in (r,x)$ also belong to $\cT_{(a,r)}$.
\item  $\cT_{(a,r)}$ is bounded.
    \item Either $\cT_{(a,r)}$ and $\cT_{(a',r')}$ are disjoint or one is included in the other.
   
\end{enumerate}
\end{lemma}
\begin{proof}
     By the comparison principle, we have $\cB_{r,x}(a-)\le \cR_{r,x}(a-)$ and $\cB_{r,x}(a)\le \cR_{r,x}(a)$. Therefore  (i) is a consequence of  the definition of $\cT_{(a,r)}$ in \eqref{eq:spindle} and of ancestors in Definition \ref{def:bifur}, while  equation \eqref{eq:spindle equiv} follows from \eqref{equivalence bifurcation}. 

    Recall that $\cR^*\le\cB^*$ and the definition of $\cT_{(a,r)}$ in \eqref{eq:spindle}. Then (ii) is a consequence of Lemma \ref{l:basic blue red} and  
     Proposition \ref{p:properties dual} (ii) applied to $\cS=\cR$, which shows that $\cR_{r,y}(a-)\le \cR^*_{-x,-y}(b)$.
     
     To see (iii), we use again that  $\cB_{r,x}(a) \le \cR_{r,x}(a)$ and notice that $\cR_{r,x}(a)= \cR_{r,x}(a-)$ for all $x$ big enough.
     
    To prove (iv), suppose without loss of generality that $r'\ge r$ and $(a,r)\neq (a',r')$. If $\cT_{(a,r)}$ and $\cT_{(a',r')}$ are not disjoint, then one can find a common point $(b,x)$ with $x>\max(r,r')$. By (i), $a=\cB^*_{-x,-r}(b)=\cR^*_{-x,-r}(b)$ and $a'=\cB^*_{-x,-r'}(b)=\cR^*_{-x,-r'}(b)$. We necessarily have $r'>r$ otherwise $r'=r$ and $a=a'$.   By (ii), $(a',r')\in \cT_{(a,r)}$, i.e. $\cR_{r,r'}(a-)\le a'<\cB_{r,r'}(a)$. Equation \eqref{eq:perfectflow-} implies that $\cR_{r,y}(a-)\le \cR_{r',y}(a'-)$ for all $y\ge r'$. Similarly, the coalescence property of $\cB$ shows that $\cB_{r,y}(a)\ge \cB_{r',y}(a')$ for all $y\ge r'$. It follows that $\cT_{(a',r')}\subseteq \cT_{(a,r)}$.
\end{proof}

\begin{definition}[Spindle]\label{def:spindle}
For $(a,r)\in (0,\infty)\times \bR$, the region $\cT_{(a,r)}$ is called a \emph{spindle}, if it is non-empty and maximal under set inclusion (meaning that there is no $(a',r')\in (0,\infty)\times \bR$ with $\cT_{(a,r)}\subsetneq \cT_{(a',r')}$). 

The point $(a,r)$ is called the \emph{bottom point} of the spindle $\cT_{(a,r)}$. 
    Letting $z = \inf \{s>r\colon \cB_{r,s}(a)=\cR_{r,s}(a-)\}$ and $c = \cB_{r,z}(a)$, the point $(c,z)$ is called the \emph{top point} of $\cT_{(a,r)}$.
    The \emph{left/red boundary} (resp.\ \emph{right/blue boundary}) of $\cT_{(a,r)}$ is defined as the set $\{(\cR_{r,s}(a-),s),\, s\in (r,z)\}$ (resp.\ $\{(\cB_{r,s}(a),s),\, s\in (r,z)\}$).
\end{definition}

We stress that, from this definition, each spindle contains its left boundary but not its right boundary, nor its top and bottom points. 
Due to Lemma~\ref{lem:cT}~(iv), if $\cT_{(a,r)}$ is  a spindle, then  for any $(a',r')\in (0,\infty)\times \bR$ with  $\cT_{(a',r')}\cap \cT_{(a,r)} \ne \emptyset$, there is $\cT_{(a',r')}\subseteq \cT_{(a,r)}$. 
In particular, different spindles are disjoint. Observe that, since the flow lines of $\cR$ are absorbed at $0$ on $\mathbb{R_+}$, the left-boundary of a spindle stays at $0$ if it hits $0$ at a positive level. We will repeatedly use that, since bottom points of spindles are bifurcation points for both flows $\cB$ and $\cR$, and in view of Proposition \ref{c: no middle bifurcation}:
\begin{lemma}\label{l:bottom_boundary}
Bottom points of spindles do not lie on the closure of any other spindle. More generally, bottom points do no belong to any flow line of $\cR$ or $\cB$ outside its starting point.
\end{lemma}

The concept of \emph{spindle} arises in the literature on interval partition evolutions \cite{Paper1-1,Paper1-2}, particularly in the context of L\'evy processes decorated by BESQ excursions. We adopt this terminology following the aforementioned references; indeed, we will later establish the exact correspondence between the two notions of spindles discussed here. Such a region is named for its spindle-like shape, as illustrated in Figure~\ref{fig:spindle}.

\begin{figure}[htbp]
\centering
       \scalebox{0.45}{ 
        \def\svgwidth{\columnwidth}
\begingroup%
  \makeatletter%
  \providecommand\color[2][]{%
    \errmessage{(Inkscape) Color is used for the text in Inkscape, but the package 'color.sty' is not loaded}%
    \renewcommand\color[2][]{}%
  }%
  \providecommand\transparent[1]{%
    \errmessage{(Inkscape) Transparency is used (non-zero) for the text in Inkscape, but the package 'transparent.sty' is not loaded}%
    \renewcommand\transparent[1]{}%
  }%
  \providecommand\rotatebox[2]{#2}%
  \newcommand*\fsize{\dimexpr\f@size pt\relax}%
  \newcommand*\lineheight[1]{\fontsize{\fsize}{#1\fsize}\selectfont}%
  \ifx\svgwidth\undefined%
    \setlength{\unitlength}{389.78207566bp}%
    \ifx\svgscale\undefined%
      \relax%
    \else%
      \setlength{\unitlength}{\unitlength * \real{\svgscale}}%
    \fi%
  \else%
    \setlength{\unitlength}{\svgwidth}%
  \fi%
  \global\let\svgwidth\undefined%
  \global\let\svgscale\undefined%
  \makeatother%
  \begin{picture}(1,0.97665404)%
    \lineheight{1}%
    \setlength\tabcolsep{0pt}%
    \put(0,0){\includegraphics[width=\unitlength,page=1]{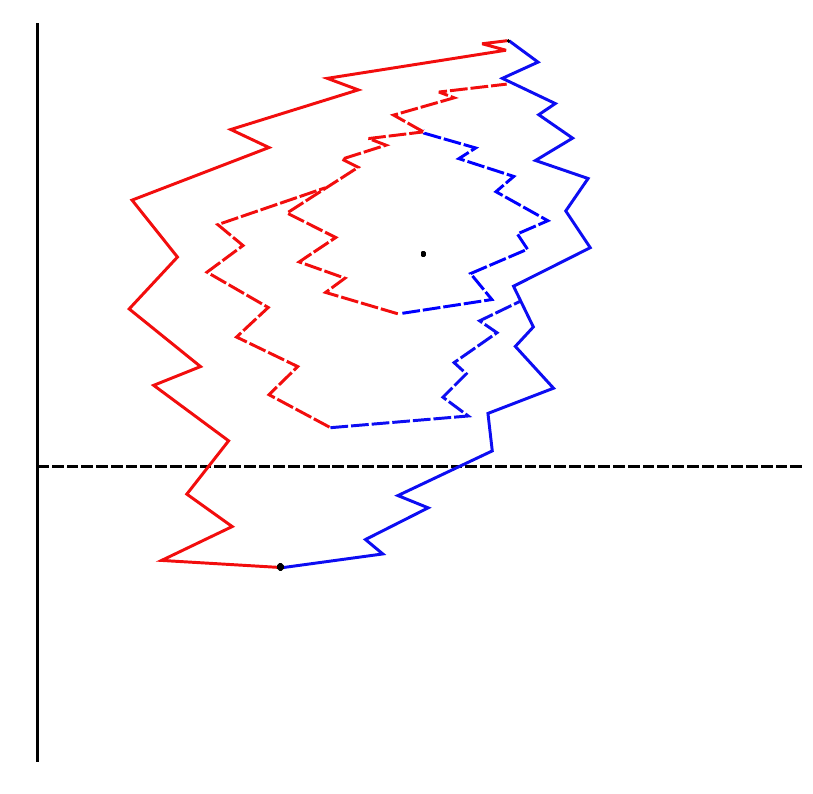}}%
    \put(0.30456335,0.22174434){\color[rgb]{0,0,0}\makebox(0,0)[lt]{\lineheight{1.25}\smash{\begin{tabular}[t]{l}\huge $(a,r)$\end{tabular}}}}%
    \put(0.48368272,0.68635862){\color[rgb]{0,0,0}\makebox(0,0)[lt]{\lineheight{1.25}\smash{\begin{tabular}[t]{l}\huge $(b,x)$\end{tabular}}}}%
    \put(0.56409679,0.94303094){\color[rgb]{0,0,0}\makebox(0,0)[lt]{\lineheight{1.25}\smash{\begin{tabular}[t]{l}\huge $(c,z)$\end{tabular}}}}%
    \put(0.00006432,0.38870287){\color[rgb]{0,0,0}\makebox(0,0)[lt]{\lineheight{1.25}\smash{\begin{tabular}[t]{l}\huge 0\end{tabular}}}}%
    \put(0.20166778,0.85876721){\color[rgb]{0,0,0}\makebox(0,0)[lt]{\lineheight{1.25}\smash{\begin{tabular}[t]{l}\huge \textcolor{red}{$\cR_{r,\cdot}(a-)$}\end{tabular}}}}%
    \put(0.6972097,0.83265179){\color[rgb]{0,0,0}\makebox(0,0)[lt]{\lineheight{1.25}\smash{\begin{tabular}[t]{l}\huge \textcolor{blue}{$\cB_{r,\cdot}(a)$}\end{tabular}}}}%
  \end{picture}%
\endgroup%

    }
    \hfill
    \scalebox{0.45}{ 
        \def\svgwidth{\columnwidth}
        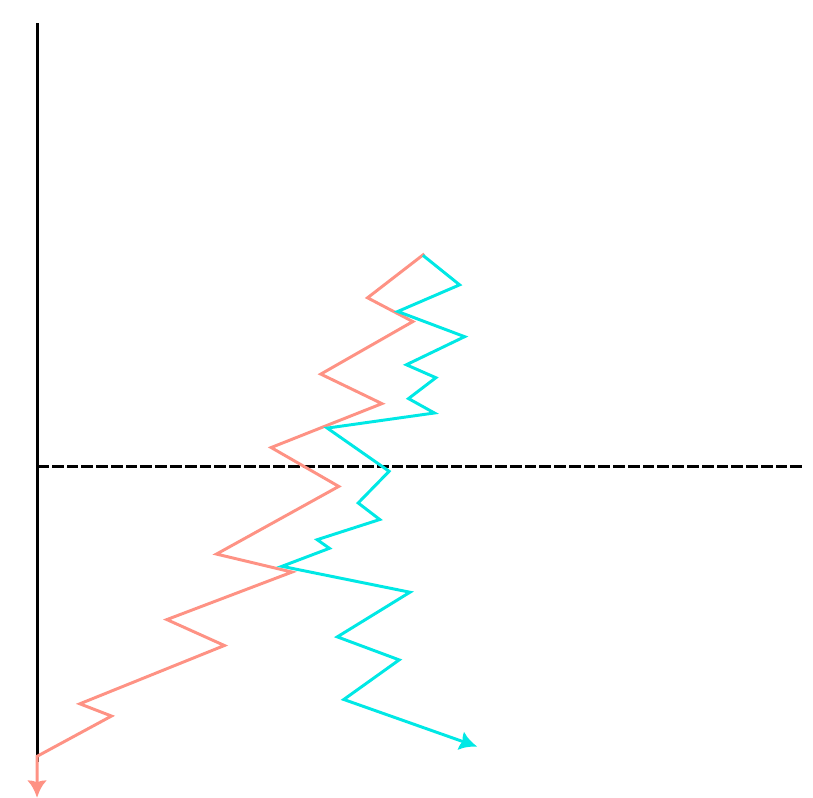
    }
    \caption{In the left figure, we illustrate several regions defined in \eqref{eq:spindle} that contain the point $(b,x)$. Among those, the spindle $\cT_{(a,r)}$ with bottom point $(a,r)$ and top point $(c,z)$ is the maximal one. 
    In the right figure, we depict the two dual flow lines $\cR^*_{-x,-\cdot}(b)$ and $\cB^*_{-x,-\cdot}(b)$ starting from $(b,x)$.  These two paths satisfy $\cR^*\le \cB^*$ and $\cR^*$ is absorbed at 0.
    By Proposition \ref{p:ancestor equiv}, the bottom point $(a,r)$ of the spindle $\cT_{(a,r)}$ is the point where these two dual flow lines intersect for the last time. It is a point where both $\cR$ and $\cB$ bifurcate.}
\label{fig:spindle}
\end{figure}

\begin{proposition}[Bottom points of spindles]\label{p:ancestor equiv}
     Almost surely, for every  $(b,x)\in \bR_+\times \bR$, we have
     \[
     r := \inf\{y\in (-\infty,x] \colon \cB^*_{-x,-y}(b) = \cR^*_{-x,-y}(b)\}>-\infty, 
     \]
     and $\cR^*_{-x,-y}(b) < \cB^*_{-x,-y}(b)$ for all $y<r$. If $r<x$, then $\cT_{(a,r)}$ is the spindle which contains $(b,x)$, where 
     \[
     a :=  \cB^*_{-x,-r}(b) = \cR^*_{-x,-r}(b).
     \] 
     In the other case, $r=x$ and $(b,x)$ does not belong to any spindle. 
 \end{proposition}
\begin{proof}
Recall that $\cB^*$ is a ${\rm BESQ}(2+\delta\, |_0 \, 2)$ flow and $\cR^*$ is a killed ${\rm BESQ}(2\,|_0\, 2-\delta)$ flow, and that  $\cR^*\le \cB^*$ because of the comparison principle.
Since the flow lines behave as squared Bessel processes, we have a.s.\  $\cB^*_{-x, -y}(b)>0$ and $\cR^*_{-x, -y}(b) =0$ for all $y<x$ small enough. This implies that $r$ defined as above is finite and that $\cR^*_{-x,-y}(b) < \cB^*_{-x,-y}(b)$ for all $y<r$. By continuity we also have the identity $\cB^*_{-x,-r}(b) = \cR^*_{-x,-r}(b)$.

If $r<x$, then $\cB^*_{-x,-r}(b)>0$ and it follows from  \eqref{eq:spindle equiv} that $\cT_{(a,r)}$ contains $(b,x)$. 
Moreover, this $\cT_{(a,r)}$ is maximal for  inclusion, so it is a spindle by definition. Indeed, if $(a',r')\in (0,\infty)\times \bR$ satisfies $(b,x)\in \cT_{(a',r')}$, then we deduce from  \eqref{eq:spindle equiv} that 
$a' = \cB^*_{-x,-r'}(b) = \cR^*_{-x,-r'}(b)$. Then $r'\ge r$ by the definition of $r$, and hence $(a',r')\in \cT_{(a,r)}$ by \eqref{eq:spindle equiv} and the perfect flow property. By Lemma~\ref{lem:cT} (iv) we have $\cT_{(a',r')}\subseteq \cT_{(a,r)}$. 

In the case $r=x$, $\cR^*_{-x,-y}(b) < \cB^*_{-x,-y}(b)$ for any $y<x$ and $(b,x)$ does not belong to any spindle by another use of \eqref{eq:spindle equiv}. This completes the proof. 
\end{proof}

Consider the union of interiors in $\bR_+\times\bR$ of all spindles  
\[
\mathscr{P}
    :=
    \bigcup_{\text{spindle } \mathcal{T}_{(a,r)}}
    \overset{\circ}{\mathcal{T}}_{(a,r)}.
\]
We stress that the interior is taken relative to $\bR_+\times\bR$, in which a set is open if it is the intersection of $\bR_+\times\bR$ with some open set in $\bR^2$. Therefore $(b,x)\in \bR_+\times\bR$ is in $\mathscr{P}$ if and only if $(b,x)$ is contained in some spindle $\cT_{(a,r)}$ and one of the following two conditions holds:
    \begin{enumerate}[(i)]
        \item $b>\cR_{r,x}(a-)$;
    \item $b=\cR_{r,x}(a-)=0$ and $x> y:=\inf\{r'>r\vee 0\colon\, \cR_{r,r'}(a-)=0\}$.  In other words, when the left boundary coalesces with the vertical axis such that the top point is some $(0,z)$, then the interior also contains $(0,x)$ with $x\in (y,z)$.
    \end{enumerate}

 \noindent See Figure~\ref{fig:interior} for an illustration.

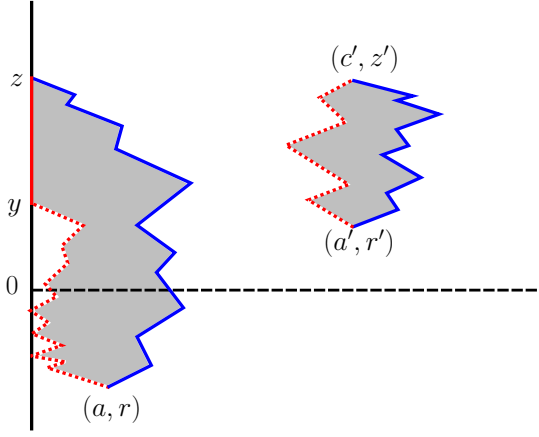
\begin{figure}[htbp]
\centering
       \scalebox{0.5}{
        \def\svgwidth{0.8\columnwidth}
\begingroup%
  \makeatletter%
  \providecommand\color[2][]{%
    \errmessage{(Inkscape) Color is used for the text in Inkscape, but the package 'color.sty' is not loaded}%
    \renewcommand\color[2][]{}%
  }%
  \providecommand\transparent[1]{%
    \errmessage{(Inkscape) Transparency is used (non-zero) for the text in Inkscape, but the package 'transparent.sty' is not loaded}%
    \renewcommand\transparent[1]{}%
  }%
  \providecommand\rotatebox[2]{#2}%
  \newcommand*\fsize{\dimexpr\f@size pt\relax}%
  \newcommand*\lineheight[1]{\fontsize{\fsize}{#1\fsize}\selectfont}%
  \ifx\svgwidth\undefined%
    \setlength{\unitlength}{160.27472909bp}%
    \ifx\svgscale\undefined%
      \relax%
    \else%
      \setlength{\unitlength}{\unitlength * \real{\svgscale}}%
    \fi%
  \else%
    \setlength{\unitlength}{\svgwidth}%
  \fi%
  \global\let\svgwidth\undefined%
  \global\let\svgscale\undefined%
  \makeatother%
  \begin{picture}(1,0.81211409)%
    \lineheight{1}%
    \setlength\tabcolsep{0pt}%
    \put(0,0){\includegraphics[width=\unitlength,page=1]{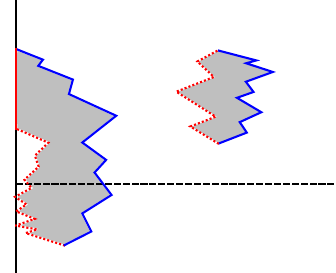}}%
    \put(0.14244647,0.02830817){\makebox(0,0)[lt]{\lineheight{1.25}\smash{\begin{tabular}[t]{l}\huge $(a,r)$\end{tabular}}}}%
    \put(0.00217451,0.41653918){\makebox(0,0)[lt]{\lineheight{1.25}\smash{\begin{tabular}[t]{l}\huge $y$\end{tabular}}}}%
    \put(0.00641124,0.65218334){\makebox(0,0)[lt]{\lineheight{1.25}\smash{\begin{tabular}[t]{l}\huge $z$\end{tabular}}}}%
    \put(0.59616615,0.34134258){\makebox(0,0)[lt]{\lineheight{1.25}\smash{\begin{tabular}[t]{l}\huge $(a',r')$\end{tabular}}}}%
    \put(0.60851761,0.69074007){\makebox(0,0)[lt]{\lineheight{1.25}\smash{\begin{tabular}[t]{l}\huge $(c',z')$\end{tabular}}}}%
    \put(-0.00127191,0.25830653){\makebox(0,0)[lt]{\lineheight{1.25}\smash{\begin{tabular}[t]{l}\huge 0\end{tabular}}}}%
  \end{picture}%
\endgroup%

    }
    \caption{Recall that in the definition of spindles, we include the left/red boundary, and exclude the right/blue boundary, as well as the top and bottom points. The interior of spindles is such that we also exclude the dashed red boundary in the picture. For the spindle $\cT_{(a',r')}$, its interior is the domain filled in gray, while for the spindle $\cT_{(a,r)}$, its interior also contains the solid red line, i.e. $\{0\}\times(y,z)$. }
    \label{fig:interior}
\end{figure}

\medskip
 
 Then we define the \emph{gasket} to be the complement of $\mathscr{P}$ in $\bR_+\times\bR$, denoted by
\begin{equation}\label{eq:gasket}
    \mathscr{K}:= \mathscr{P}^{c} = (\bR_+\times\bR)\setminus \mathscr{P}. 
\end{equation} 
We observe that $\mathscr{K}$ consists of points which do not belong to any spindle as well as the points $(\cR_{r,x}(a-),x)$,  $x\in (r, z)$ if $y\ge z$, and $(\cR_{r,x}(a-),x)$,  $x\in (r, y]$ if $y< z$
in the notation above. The next statement shows that the spindles form a partition of $\bR_+\times \bR$, up to a Lebesgue-negligible set.
\begin{corollary}\label{c:spindle zero measure}
    For any fixed $(b,x)\in (0,\infty)\times \mathbb{R}$, it is  almost surely in the interior of some spindle, i.e.\ $(b,x)\in \mathscr{P}$. Therefore the gasket $\mathscr{K}$ almost surely has Lebesgue measure 0.
\end{corollary}
 \begin{proof}
Let $Z_s:=\cB^*_{-x,-x+s}(b)-\cR^*_{-x,-x+s}(b)$, $s\ge 0$.
 By Proposition \ref{p:decomp} (i) with $\cS=\cB^*$, $Y=\cR^*_{-x,\cdot}(b)$ and $\cW=-W^*$, taking into account that $\cR^*$ is killed at $0$,  $Z$ is a  ${\rm BESQ}_0^\delta$ process until $\cR^*_{-x,-x+s}(b)$ hits $0$, after which time $Z$ becomes a ${\rm BESQ}^2$ process.  Therefore there exists a last time $s'>0$ such that $Z_{s'}=0$. The point $(a,r):=(\cR^*_{-x,-x+s'}(b),x-s')$ is the bottom point of the spindle containing $(b,x)$ by Proposition \ref{p:ancestor equiv}. Being on the (left-)boundary of the spindle $\cT_{(a,r)}$ would mean that $\cR_{r,x}(a-)=b$,  which has probability $0$ by the instantaneously coalescing property. So $(b,x)$ is in the interior of $\cT_{(a,r)}$. The last statement  of the corollary is a consequence of Fubini's theorem. 
\end{proof}

\subsection{Exploration with a PRBM}

 We consider the continuous space-filling  process associated with the PRBM $X$: 
 \[
 \varrho_t:=(L(t,X_t),X_t), \qquad t\ge 0. 
 \]
At any time $t\ge 0$, the range $\{\varrho_u,\,u\le t\}$ is  $\bigcup_{x\in [I_t,M_t]} [0,L(t,x)]\times\{x\}$ where $I_t=\inf_{[0,t]} X$ and $M_t=\sup_{[0,t]} X$ as defined in \eqref{def:I}. In words, it is the domain to the left of the line $x\in [I_t,M_t]\mapsto L(t,x)$. Observe that almost surely, for any $(b,x) \in (0,\infty)\times \bR$, and $t\ge 0$,\vspace{-0.1cm}
\begin{equation}\label{eq:range_explorer}
L(t,x)\ge b \Leftrightarrow (b,x) \in \{\varrho_u,\,u\le t\}.\vspace{-0.2cm}
\end{equation}

 We first deduce a domain Markov property intrinsic to  the exploration procedure, which is a consequence of Proposition \ref{p:markov S}. 
Recall that, for $(b,s)\in \bR_+\times \bR$, 
\[
\tau_b^s:=\inf\{t\geq 0: L(t,s)>b\}\vspace{-0.1cm}
\]
 is the right-continuous inverse local time of $X$ at level $s$.

\begin{proposition}\label{WN^t}
Fix $(b,s)\in \bR_+\times \bR$ and let $T=\tau_b^s$. Define $W^T$ as in Proposition \ref{p:markov S}. Recall the notation \eqref{eq:S-g}.
\begin{enumerate}[(i)]
    \item The flow $\cR - L(T,\cdot)$ is a ${\rm BESQ}(\delta\, |_{I_T}\, 2\, |_s\, 0)$ flow driven by $W^T$.
  
    \item The flow $\cB - L(T,\cdot)$ is a killed ${\rm BESQ}(0\, |_{I_T}\, 2-\delta\, |_s\, -\delta)$ flow driven by $W^T$.
\end{enumerate}
\end{proposition}

\noindent For the case $b=0,\, s<0$ and thus $I_T=s$, the process  $\cR - L(T,\cdot)$ is a  ${\rm BESQ}(\delta\, |_s\, 0)$ flow  while  $\cB - L(T,\cdot)$ is a killed ${\rm BESQ}(0\, |_s\, -\delta)$ flow.
\begin{proof}
 We apply Proposition \ref{p:markov S} (iii) to $\cS'=\cR$ (i.e $\delta'=0$) and $\cS'=\cB$ (i.e. $\delta'=-\delta$). 
 \end{proof}

\vspace{-0.1cm}
The next proposition shows that the amount of time spent in the gasket is Lebesgue-negligible.
\begin{proposition}\label{c:ancestor-time_fixed}
Fix $t> 0$. Then almost surely $\varrho_t=(L(t,X_t),X_t)$ lies in the interior of some spindle, i.e.\ $\varrho_t \in \mathscr{P}$.
\end{proposition}
 
The proof relies on the scaling invariance of our model. 
For $c>0$, consider the usual Brownian scaling with  $X^{(c)}_t=c^{-1} X_{c^2 t}$.  
Then the explorer process satisfies the identity  
\[
\varrho^{(c)}_t:= (L^{(c)}(t,X^{(c)}_t),  X^{(c)}_t)= \Big( c^{-1} L(c^2t, X_{c^2 t}), c^{-1} X_{c^2 t}\Big) =c^{-1} \varrho_{c^2 t}, \qquad t\ge 0. 
\]
This corresponds to scaling the space-time plane $\bR_+\times \bR$ by a factor $c$ which gives   
\begin{equation}\label{eq:scaling-inv}
    \cB^{(c)}_{r,x}(a)=c^{-1} \cB_{cr,cx}(ca), \quad \cR^{(c)}_{r,x}(a)=c^{-1} \cR_{cr,cx}(ca), \qquad \forall\, x\ge r,\, a\ge 0,  
\end{equation}
 and the law of the flows is invariant under this scaling.

\begin{proof}[Proof of Proposition \ref{c:ancestor-time_fixed}]
 The occupation field of $\varrho$ is given by the Lebesgue measure on $\bR_+\times \bR$, as implied by the occupation times formula of $X$. 
Hence, Corollary \ref{c:spindle zero measure} implies that 
\[ \int_{\bR_+} \ind{\varrho_t \in \mathscr{P}^c} \dd t=0 \quad \text{a.s.} \]
Taking expectation and observing that, by scaling, $\bP(\varrho_t \in \mathscr{P}^c)$ does not depend on $t>0$, we deduce that $\bP(\varrho_t \in \mathscr{P}^c)=0$ for every $t> 0$. 
\end{proof}

 The next result will be instrumental in our understanding of the exploration of spindles. Recall the definition of the right boundary of a spindle in Definition \ref{def:spindle}.  We call $t\ge 0$ an exit time of a spindle $\cT$ if 
 \begin{itemize}\vspace{-0.1cm}
     \item $\varrho_t$ is the top point or on the right boundary of  $\cT$
     \item and $\varrho_u\in \cT$ for all $u<t$ close enough to $t$.
 \end{itemize}

\begin{lemma}\label{l: no max right}
     Almost surely, for every $t\ge 0$: 
     
     (i) if  $\varrho_t$ is the top point or on the right boundary of a spindle, then there exist $u$ arbitrarily close to $t$ such that $X_u>X_t$;
     
      if $t$ is some exit time of a spindle, then:
     
     (ii) $X_u\le X_t$ for all $u<t$ close enough to $t$;
     
     (iii) there exist $u>t$ arbitrarily close to $t$ such that $X_u<X_t$;

     (iv)  there exist $u>t$ arbitrarily close to $t$ such that $X_u>X_t$.
\end{lemma}

\begin{proof}
     Consider any time $t\ge 0$ such that $\varrho_t$ is the top point or on the right boundary of a certain spindle $\cT_{(a,r)}$ with bottom point $(a,r)\in (0,\infty)\times \bR$. In particular
    \begin{equation}\label{proof:exit_blue}
    L(t,X_t)=\cB_{r,X_t}(a).
    \end{equation}

Suppose by contradiction that $X_u\le X_t$ for all $u$ in a neighbourhood $(t_1,t_2)$ of $t$. Without loss of generality, one can suppose that  $t_1$ and $t_2$ are of the form $t_1=\tau_{b_1}^{x}$, $t_2=\tau_{b_2}^{x}$ for some $x<X_t$ and $b_2>b_1$. By \eqref{no_increase},  the flow lines $\cR_{x,\cdot}(b_1)=L(\tau_{b_1}^x,\cdot)$ and $\cR_{x,\cdot}(b_2)=L(\tau_{b_2}^x,\cdot)$ meet for the first instance at level $X_t$ since $X_t$ is the maximum of $X$ on $(\tau_{b_1}^x,\tau_{b_2}^x)$, i.e.
\begin{equation}\label{eq:X_t-R}
    X_t = \inf\{y\ge x\colon \cR_{x,y}(b_1) = \cR_{x,y}(b_2) \}, 
\end{equation}

\noindent with common value $\cR_{x,X_t}(b_1) = \cR_{x,X_t}(b_2) = L(t,X_t)$. In view of \eqref{proof:exit_blue}, it implies that
\begin{equation}\label{proof:equal_R_B}
\cR_{x,X_t}(b_1) = \cR_{x,X_t}(b_2)=\cB_{r,X_t}(a).
\end{equation}
We will show that \eqref{eq:X_t-R} and \eqref{proof:equal_R_B} are in contradiction.
  Since the dimension of $\cB$ is strictly less than that of $\cR$, by the comparison principle, a blue flow line  remains to the left of a red flow line after they meet. If \eqref{eq:X_t-R} and \eqref{proof:equal_R_B} both hold, then it implies that, for all $z<X_t$ sufficiently close to $X_t$, one of the following three cases holds :
\begin{equation*}
    \text{(a) } \cB_{r,z}(a)>\cR_{x,z}(b_2),
    \quad  \text{(b) } \cR_{x,z}(b_1) < \cB_{r,z}(a) \le \cR_{x,z}(b_2),
    \quad \text{(c) } \cB_{r,z}(a)\le \cR_{x,z}(b_1).
\end{equation*}
By the perfect flow property of $\cR$, we suppose without loss of generality that $x$ is a rational number strictly greater than $r$ and  that one of the cases holds for all $z\in [x,X_t)$. By \eqref{eq:rightcont} and the coalescence property of $\cR$, one can also suppose that $b_1$ and $b_2$ are rational numbers and $\cB_{r,x}(a) \notin \{b_1,b_2\}$. 
Finally,  one can choose a rational $b_3>\cB_{r,x}(a)$ close enough so that the inequalities in (a), (b), (c) hold with $\cB_{x,z}(b_3)$ in place of $\cB_{r,z}(a)$, and $\cB_{r,X_t}(a)=\cB_{x,X_t}(b_3)$. 

Suppose that (a) holds. By \eqref{proof:equal_R_B} and our choice of $b_3$, $\cB_{x,X_t}(b_3)=\cR_{x,X_t}(b_2)$.
Applying Proposition~\ref{p:decomp} to make a decomposition  with respect to $Y = \cR_{x,\cdot}(b_2)$, we obtain independent $W^+$ and $W^-$ as in Proposition~\ref{p:decomp} (iii). 
Then $\max(\cB_{x,\cdot}(b_3)-\cR_{x,\cdot}(b_2),0)$ is measurable with respect to $W^+$ while  $\cR_{x,\cdot}(b_2)-\cR_{x,\cdot}(b_1)$ is measurable with respect to $W^-$.  
Hence $\max(\cB_{x,\cdot}(b_3)-\cR_{x,\cdot}(b_2),0)$ is independent of $X_t$ in view of \eqref{eq:X_t-R}.
It contradicts the fact that $X_t$ is the hitting time of $0$ by $\max(\cB_{x,\cdot}(b_3)-\cR_{x,\cdot}(b_2),0)$.
In case (b),   we use similarly that $\cR_{x,\cdot}(b_2)-\cB_{x,\cdot}(b_3)$ cannot equal $0$ at the hitting time  of $0$ by $\cB_{x,\cdot}(b_3)-\cR_{x,\cdot}(b_1)$.  In case (c), we use that $\cR_{x,\cdot}(b_1)-\cB_{x,\cdot}(b_3)$ cannot equal $0$ at the hitting time  of $0$ by $\cR_{x,\cdot}(b_2)-\cR_{x,\cdot}(b_1)$. This completes the proof of (i).

    We turn to the proof of (ii), hence we suppose that $t$ is an exit time of a spindle $\cT$ and we let $t_1<t$ such that $\varrho\in \cT$ on $(t_1,t)$. Note that if $\varrho_t$ is the top point of $\cT$, then $X_u\in \cT$ implies $X_u\le X_t$ so the statement is clear in this case. So we can suppose that $\varrho_t$ is on the right boundary of $\cT$. 
    Suppose by contradiction that $X_u>X_t$ for some $u\in (t_1,t)$. Write $\varrho_t=(b,x)$. Since $\varrho_t$ in on the right boundary of $\cT$, we have $b>0$. In particular $t\ge \sigma_b^x$. By Proposition \ref{c: no middle bifurcation}, $\varrho_t$ is not a bifurcation point for $\cB$. By the comparison principle applied to the left-continuous versions of $\cR$ and  $\cB$, we deduce that $\cB_{x,X_u}(b)\le \cR_{x,X_u}(b-)=L(\sigma_b^x,X_u)\le L(t,X_u)$. Choosing $v\in [u,t)$ as the last visit time of $X_u$ before time $t$, we get that $X_v=X_u$ and $L(v,X_u)=L(t,X_u) \ge \cB_{x,X_u}(b)$ which contradicts $\varrho_v \in \cT$. This proves (ii). The statement (iii) follows from (ii)  since $X$ does not have points of increase. Statement (iv) is a consequence of (i) and (ii).
\end{proof}

\subsection{Exploring a spindle}\label{s:explore}
Knowing that $\varrho$ is a space-filling curve that will eventually cover the full half plane, we now study how it explores each spindle. 
Informally speaking, the exploration has two stages described as follows. 
The process $\varrho$ enters each spindle through the bottom point, staying in the spindle for a while until it exits the spindle by the top point. 
Now it has explored the left part and it will return to the spindle afterwards. 
For the return, $\varrho$ enters the spindle from the right boundary, makes an excursion inside the spindle and exits again from the entrance point. 
The exploration process $\varrho$ repeats such excursions iteratively until the right part of the spindle is fully filled. This procedure is illustrated in Figure~\ref{f:exploration}. 
We will formalize this intuitive description with precise mathematical statements in this subsection.

 For each spindle $\cT$, define 
\begin{equation}
    s_{\cT} := \inf\{ t\ge 0 \colon   \varrho_t\in \overset{\circ}{\cT}\} 
    \quad \text{and} \quad 
    t_{\cT} := \inf\{ t> s_{\cT} \colon \varrho_t \not\in \cT\}. 
\end{equation}
So $s_{\cT}$ is the hitting time of the interior of the spindle in $\bR_+\times\bR$  and $t_{\cT}$ is the first exit time of the spindle after entering. 
At time $t\ge 0$, we say that a spindle $\cT$ has been
\begin{itemize}
    \item  \emph{discovered}  if $t>s_{\cT}$; 
    \item \emph{partially explored} if  $t\in (s_{\cT}, \sup\{t\ge 0 \colon \varrho_{t} \in \cT\})$; 
\item \emph{ fully explored} if $t\ge \sup\{t\ge 0 \colon \varrho_{t} \in \cT\}$. 
 \end{itemize} 
Recall that, for each $r\in \bR$, $(\sigma_a^r, a\ge 0)$ and $(\tau_a^r, a\ge 0)$ denote the left- and right-continuous inverse local times at level $r$ of $X$ respectively.

\begin{proposition}[Exploration of the left part]\label{prop:explore}
With probability 1, the following statement holds for every spindle $\cT$: if $(a,r)$ denotes its bottom point and $(c,z)$ its top point, then

(i)   
  $s_{\cT}=\sigma_a^r$; 

(ii) $t_{\cT}$ is the smallest time  $t>s_{\cT}$ when $\varrho$ reaches the top of the spindle, thus $\varrho_{t_{\cT}}= (c,z)$;

(iii) the end of the exploration of the spindle happens at time $\tau_a^r$, when $X$ completes an excursion above $r$, hence \ $\sup\{t\ge 0 \colon \varrho_{t} \in \cT\} = \tau_a^r$. 
\end{proposition}

By Proposition \ref{prop:explore}, we can intuitively divide a spindle $\cT$ into a ``left'' part and a ``right'' part; to be precise, the former is the range of $\varrho$ in the time interval $(s_{\cT},t_{\cT})$, while the latter is its complement in $\cT$. 

\begin{proof}
Recall the characterization of $\overset{\circ}{\cT}$ before \eqref{eq:gasket}. Suppose by contradiction that $\varrho_t\in \overset{\circ}{\cT}$ for some $t< \sigma_a^r$. Then one can find $t'<\sigma_a^r$ such that $\varrho_{t'}\in \overset{\circ}{\cT}$ and $L(t',X_{t'})>0$. For
 $x>r$, $L(t',x)\le L(\sigma_a^r,x)=\cR_{r,x}(a-)$. Applying it to $x=X_{t'}$ yields that $\varrho_{t'}$ is necessarily on the left boundary of $\cT$, which contradicts $\varrho_{t'}\in \overset{\circ}{\cT}$ and $L(t',X_{t'})>0$. It proves that  $s_{\cT}\ge \sigma_a^r$.  Since $L(\sigma_a^r,x)=\cR_{r,x}(a-)<\cR_{r,x}(a)=L(\tau_a^r,x)$ for $x\in (r,z)$, at time $\sigma_a^r$, $X$ starts an excursion above $r$. Since one can find $t>\sigma_a^r$ arbitrarily close to $\sigma_a^r$ such that $L(t,X_t)>L(\sigma_a^r,X_t)$, we get (i). 

Let $x\in (r,z)\cap \mathbb{Q}$ and $b=\cR_{r,x}(a-)=L(\sigma_a^r,x)$. With probability $1$, no rational is a local maximum of $X$. Applying it to $x$, we see that  $\tau_{b}^{x}=\inf\{t>\sigma_a^r \colon X_t=x\}$ since $X$ will accumulate some local time as soon as it crosses $x$, see \eqref{no_increase}. Write $t'=\tau_b^x$ for brevity, in particular $t'>\sigma_a^r$. 
We know that $a=\cR^*_{-x,-r}(b)=\cB^*_{-x,-r}(b)$ by equation \eqref{eq:spindle equiv}. 
Since 
\begin{equation}\label{eq proof explore}
\cR^*_{-x,-y}(b)=L(\tau_{b}^{x},y)=L(t',y) \quad \text{for }y\le x, 
\end{equation}
we have $a=L(t',r)$ and hence deduce that $t'\in (\sigma_a^r,\tau_a^r)$, on which time interval $X$ is making an excursion above $r$.

By Lemma \ref{l:basic blue red}, we have $\cB^*_{-x,-y}(b)<\cB_{r,y}(a)$ for $y\in (r,x)$. Recall that $\cR^*\le \cB^*$ by the comparison principle; then we have $\cR^*_{-x,-y}(b)<\cB_{r,y}(a)$. Together with \eqref{eq proof explore}, it shows that for all $t\in (\sigma_a^r, t')$, $L(t,X_t)\le L(t',X_t)<\cB_{r,X_t}(a)$ while  $L(t,X_t)\ge L(\sigma_a^r,X_t)=\cR_{r,X_t}(a-)$. In other words, $\varrho \in \cT$ on $(\sigma_a^r,t')$. It follows that $t_{\cT} \ge t'$. Making $x\to z$, we get that $t_{\cT}=\inf\{t>\sigma_a^r \colon X_t=z\}$.  We deduce  (ii).

At time $\tau_a^r$, $L(\tau_a^r,x)=\cR_{r,x}(a)\ge \cB_{r,x}(a)$  for any $x\ge r$ hence the whole spindle has been visited. If $t\in (\sigma_a^r,\tau_a^r)$, we have $L(t,x)=L(t_{\cT},x)$ for all $x>r$ sufficiently close to $r$. By (ii), $L(t_{\cT},y)<\cB_{r,y}(a)$ for $y\in (r,z)$. Hence for $x>r$ small enough, $L(t,x)<\cB_{r,x}(a)$ and any point $(b,x)$ with $b\in (L(t,x),\cB_{r,x}(a))$ remains  unvisited by $\varrho$. This completes the proof of (iii).
 \end{proof}

The next proposition describes excursions of $\varrho$  in the right part of a spindle $\cT$, i.e.\ after time $t_{\cT}$.
We summarize our observations of the exploration procedure in Figure~\ref{f:exploration}. 

\begin{proposition}[Exploration of the right part]
\label{p:right return}
    Almost surely: for every  $t\ge 0$ and every spindle $\cT$,  setting $(a,r)$ and $(c,z)$ for its bottom and top points,  if $t> t_{\cT}$ and $\varrho_t\in \cT$, then there exist $t_0<t<t_1$ and $r_0\in (X_t,z)$ such that:
    
    (i)  $\varrho_{t_0}=\varrho_{t_1}=(\cB_{r,{r_0}}(a),r_0)$;

    (ii) by time $t$, $\varrho$ has visited all points $((\mathcal{B}_{r,x}(a), x)$ for $x \in [r_0, z]$, but has not visited any points $(\mathcal{B}_{r,x}(a), x)$ for $x \in (r, r_0)$; 

    (iii) for any $s\in (t_0,t_1)$, we have $\varrho_s \in \cT$  and $X_s\in (r,r_0)$. 
\end{proposition}

We will use the following observation in the proof.
\begin{lemma}\label{l: green line}
    Almost surely, for any $t\ge 0$,  $x\in \bR$ and any spindle $\cT=\cT_{(a,r)}$: if $\cT$ has been discovered before time $t$ and $(L(t,x),x)\in \cT$ then
    
    (i) $t\in (\sigma_a^r,\tau_a^r)$,
    
    (ii) $\cB_{r,y}(a)>L(t,y)$ for all $y\in (r,x]$.
\end{lemma}

\begin{proof}
Statement (i) is a consequence of Proposition \ref{prop:explore}, noting that $\cT$ has unexplored points at level $x$ at time $t$ by \eqref{eq:range_explorer}.  To prove (ii), let $b=L(t,x)$ and observe that  $\tau_{b}^{x}\ge t$. It implies 
\begin{equation}\label{eq:proof green line + 1}
    \cR^*_{-x,-y}(b)=L(\tau_b^x,y)\ge L(t,y) \quad \text{for all } y\le x.
\end{equation}

\noindent Since $(b,x)\in \cT$, we conclude with Lemma \ref{lem:cT} (ii).
 \end{proof}

\begin{proof}[Proof of Proposition~\ref{p:right return}]
By Proposition \ref{prop:explore} (ii), we have $L(t_{\cT},z)=L(\sigma_a^r,z)$. By definition of the top point, it holds that $\cR_{r,z}(a-)=\cB_{r,z}(a)$. Then, since $\cR_{r,z}(a-)=L(\sigma_a^r,z)$, we have $L(t_{\cT},z)=\cB_{r,z}(a)$. Observe that  $X$ accumulates local time at level $z$ immediately after $t_{\cT}$ by Lemma \ref{l: no max right} (ii) and (iv) and \eqref{no_increase}. It follows that $ L(t,z)>  L(t_{\cT},z)=\cB_{r,z}(a)$, as $t>t_{\cT}$. Let $r_0:=\inf\{y\in (r,z)\colon L(t,y)=\cB_{r,y}(a)\}$. By continuity, $L(t,r_0)=\cB_{r,r_0}(a)$ and $r_0<z$. By Lemma \ref{l: green line}~(ii) (applied with $x = X_t$), we also have $r_0> X_t$ and one cannot have $L(t,x)<\cB_{r,x}(a)$ for some $x\in (r_0,z)$: this would imply that $(L(t,x),x)\in \cT$, which would contradict  $L(t,r_0)=\cB_{r,r_0}(a)$ by Lemma~\ref{l: green line} (ii) with $y=r_0$. In view of \eqref{eq:range_explorer}, it already yields  (ii).

Let $t_0 = \sup\{u<t \colon X_u=r_0\}$ and $t_1 = \inf\{u>t \colon X_u=r_0\}$. Then $L(t_0,r_0)=L(t_1,r_0)=L(t,r_0)$, which equals $\cB_{r,r_0}(a)$ by definition of $r_0$. Hence (i) holds. 

By definition of $t_0$ and $t_1$, we also have $X<r_0$ on $(t_0,t_1)$. Since $X_t<r_0<z=X_{t_{\cT}}$ and $t>t_{\cT}$, we necessarily have $t_0\ge t_{\cT}$ by definition of $t_0$ and (ii). 
We now show that $\varrho_s \in \cT$ for all $s\in (t_0,t_1)$, which will also imply that $X>r$ on this time interval.  
Suppose by contradiction that it fails, and pick $s\in (t_0,t_1)$  such that $\varrho_s\notin \cT$. As $X<r_0$ on $(t_0,t_1)$, we also have $\sup_{ [s\wedge t,s\vee t]}X<r_0$. Therefore, for all $x<r_0$ close enough to $r_0$, we have $L(s,x)=L(t,x)<\cB_{r,x}(a)$ and $(L(s,x), x)\in \cT$. Since the spindle has not been fully explored yet at time $s$, we have $s<\tau_a^r$ by Proposition~\ref{prop:explore}~(iii) and hence deduce that $X_s\in (r,r_0)$. Applying Lemma~\ref{l: green line}~(ii) to $s$ and $x>X_s$ close enough to $r_0$ with $(L(s,x), x)\in \cT$, we have $L(s,X_s)< \cB_{r,X_s}(a)$, which contradicts  to $\varrho_s\notin\cT$.
This completes the proof.
\end{proof}


 \begin{figure}[htbp]
\centering
       \scalebox{0.4}{ 
        \def\svgwidth{0.9\columnwidth}
\begingroup%
  \makeatletter%
  \providecommand\color[2][]{%
    \errmessage{(Inkscape) Color is used for the text in Inkscape, but the package 'color.sty' is not loaded}%
    \renewcommand\color[2][]{}%
  }%
  \providecommand\transparent[1]{%
    \errmessage{(Inkscape) Transparency is used (non-zero) for the text in Inkscape, but the package 'transparent.sty' is not loaded}%
    \renewcommand\transparent[1]{}%
  }%
  \providecommand\rotatebox[2]{#2}%
  \newcommand*\fsize{\dimexpr\f@size pt\relax}%
  \newcommand*\lineheight[1]{\fontsize{\fsize}{#1\fsize}\selectfont}%
  \ifx\svgwidth\undefined%
    \setlength{\unitlength}{202.29766125bp}%
    \ifx\svgscale\undefined%
      \relax%
    \else%
      \setlength{\unitlength}{\unitlength * \real{\svgscale}}%
    \fi%
  \else%
    \setlength{\unitlength}{\svgwidth}%
  \fi%
  \global\let\svgwidth\undefined%
  \global\let\svgscale\undefined%
  \makeatother%
  \begin{picture}(1,1.1019338)%
    \lineheight{1}%
    \setlength\tabcolsep{0pt}%
    \put(0,0){\includegraphics[width=\unitlength,page=1]{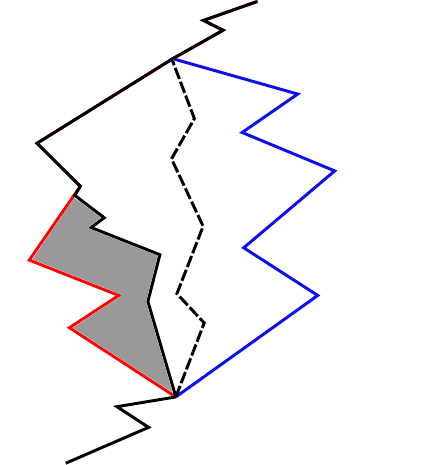}}%
    \put(0.38347143,0.09657674){\makebox(0,0)[lt]{\lineheight{1.25}\smash{\begin{tabular}[t]{l}\huge $(a,r)$\end{tabular}}}}%
    \put(0.32261268,0.98828517){\makebox(0,0)[lt]{\lineheight{1.25}\smash{\begin{tabular}[t]{l}\huge $(c,z)$\end{tabular}}}}%
    \put(-0.00042807,0.41096293){\makebox(0,0)[lt]{\lineheight{1.25}\smash{\begin{tabular}[t]{l}\huge \textcolor{red}{$L(\sigma_a^r,\cdot)$}\end{tabular}}}}%
    \put(0.66499662,0.48352537){\makebox(0,0)[lt]{\lineheight{1.25}\smash{\begin{tabular}[t]{l}\huge \textcolor{blue}{$\cB_{r,\cdot}(a)$}\end{tabular}}}}%
    \put(0.07928729,0.07378618){\makebox(0,0)[lt]{\lineheight{1.25}\smash{\begin{tabular}[t]{l}\huge $L(t,\cdot)$\end{tabular}}}}%
    \put(0,0){\includegraphics[width=\unitlength,page=2]{exploration3.pdf}}%
    \put(0.2910208,0.55877021){\makebox(0,0)[lt]{\lineheight{1.25}\smash{\begin{tabular}[t]{l}\huge $\varrho_t$\end{tabular}}}}%
  \end{picture}%
\endgroup%

    }
    \qquad
    \scalebox{0.4}{ 
        \def\svgwidth{0.9\columnwidth}
        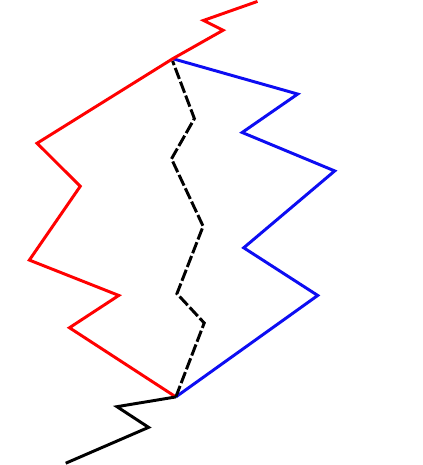
    }
    \caption{The dashed line divides the spindle $\cT=\cT_{(a,r)}$ into its left and right parts. \\ The left picture shows how $\varrho$ explores the left part as described in Proposition~\ref{prop:explore}. The process $\varrho$ enters the interior of the spindle $\cT$ at time $s_\cT=\sigma_a^r$, stays in $\cT$ until time $t_\cT$, when it exits from the top point $(c,z)$ of $\cT$. Note that $L(\sigma_a^r,x)=\cR_{r,x}(a-)$ for $x\ge r$, and the grey area represents the range of $\varrho$  over the time interval $[\sigma_a^r,t]$.
    \\ The right picture shows how the right part is explored, as explained in Proposition~\ref{p:right return}. The explorer $\varrho$ will re-enter and re-exit $\cT$ for a number of times, until completing the exploration of $\cT$ at time $\tau_a^r$. At each time $t_0$ when $\varrho$ re-enters $\cT$, $\varrho$ has already explored the part of the blue line above $r_0=X_{t_0}$. Subsequently, $X$ makes an excursion below $r_0$ which  ends at time $t_1$. For every $t_0< t< t_1$, $\varrho_t\in\cT$,  exploring the cyan area, and $\varrho$ exits the spindle at time $t_1$ with $\varrho_{t_0} = \varrho_{t_1}$. } 
\label{f:exploration}
\end{figure}

Let $\mathcal{E}$ (for ``explore'') denote the set of all times when $\varrho$ is exploring a spindle, i.e.
\begin{equation}\label{def:explore}
  \mathcal{E}:=\{t\ge 0 \colon \varrho_t \in \cT \textrm{ for some spindle }\cT,\, \text{ and } t>s_{\cT} \}.  
\end{equation}
 The set $\cE$ is open. Indeed, if $\varrho_t\in \cT$ and $t>s_{\cT}$, then either $\varrho_t$ is in the left part of $\cT$, so that $t\in (s_{\cT},t_{\cT})\subseteq \cE$, or $\varrho_t$ is in the right part of $\cT$, hence $t\in (t_0,t_1)\subseteq \cE$ in the notation of Proposition \ref{p:right return}. We end this subsection with a discussion on the structure of the complementary set $\mathcal{E}^c=\{t\ge 0 : t\notin \mathcal{E}\}$, which will play a crucial role in our analysis in Section~\ref{sec:regenerative}.   Observe that the bottom point   of a spindle cannot be in the interior of another spindle since spindles are disjoint, nor can it be on the left boundary of a spindle by Lemma \ref{l:bottom_boundary}. Therefore bottom  points of spindles do not belong to any spindle and $\cE^c$ includes all times $s_{\cT}$.  Recall the definition of the gasket $\rK$ in \eqref{eq:gasket}. We get
 \begin{equation}\label{eq:bottom_gasket}
     \{\varrho_{s_{\cT}}\colon \textrm{spindle } \cT\} \subset \rK,\quad \{s_{\cT} \colon \textrm{spindle } \cT\}\subset \cE^c.
 \end{equation}

\begin{lemma}\label{l:exit_explore}
    Almost surely, for every $t\ge 0$ that is an exit time of some spindle $\cT$, we have $t\in\cE^c$.
\end{lemma}
\begin{proof}
    The result holds by definition of $\cE$ if $\varrho_t$ does not lie in any spindle. The remaining case is that $\varrho_t\in\cT'$ for another spindle $\cT'$. 
    Let $(a,r)$ and $(a',r')$ be  the bottom points of $\cT$ and $\cT'$ respectively. Since $t$ is an exit time of  $\cT$, necessarily $\varrho_t$ lies on the right boundary of $\cT$ or is its top point, hence $X_t>r$ and $L(t,X_t)=\cB_{r,X_t}(a)$. Since $\varrho_t\in \cT'$, $\cR_{r',X_t}(a'-)\le L(t,X_t)<\cB_{r',X_t}(a')$. In particular,
    $$
    \cB_{r',X_t}(a') >\cB_{r,X_t}(a).
    $$
    \noindent In the case $r'\ge r$, the equation above and the coalescence property show that $a'>\cB_{r,r'}(a)$; we now prove that $a'>\cR^*_{-r,-r'}(a)$ in the case $r'<r$, and therefore justify the inclusion 
        \begin{equation}\label{eq:a'r'}
                (a',r')\in\{(b,x):r\le x<X_t,\, b>\cB_{r,x}(a)\}\cup\{(b,x):x<r,\, b>\cR^*_{-r,-x}(a)\}.
        \end{equation}
    Since $a'>0$, we restrict to the case $\cR^*_{-r,-r'}(a)>0$. Suppose $a'\le \cR^*_{-r,-r'}(a)$ by contradiction. 
        By  definition of the dual in \eqref{def:dual}, $\cR_{r',r}(b)\le a$ for all $b<\cR^*_{-r,-r'}(a)$. It follows that $\cR_{r',r}(a'-)\le a$ by its definition \eqref{def:left S}. Therefore, \eqref{eq:perfectflow-} yields
\vspace{-0.1cm}
    \begin{equation}\label{eq:exit_explore1}
        \cR_{r',x}(a'-)\le \cR_{r,x}(a-),\quad  \forall\, x\ge r.
    \end{equation}
Since $r<X_t$ and $\cB_{r,X_t}(a)<\cB_{r',X_t}(a')$, for all 
$r''\in(r,X_t)$ close enough to $X_t$, we have $\cB_{r,r''}(a)<\cB_{r',r''}(a')$. Moreover $\cR_{r,r''}(a-)<\cB_{r,r''}(a)$ for all $r''\in (r,X_t)$ by Proposition \ref{prop:explore} (ii) if $\varrho_t$ is the top point of $\cT$ and Proposition \ref{p:right return} (ii) if $\varrho_t$ is on its right boundary. We deduce that for all $r''<X_t$ close enough to $X_t$, 
$$
\cR_{r,r''}(a-)<\cB_{r,r''}(a)<\cB_{r',r''}(a').
$$

\noindent In view of this equation and \eqref{eq:exit_explore1}, any point $(a'',r'')$ with $a''\in (\cR_{r,r''}(a-),\cB_{r,r''}(a))$ would belong to $\cT\cap\cT'$, see \eqref{eq:spindle}. It gives the desired contradiction since spindles are disjoint.
We have proved \eqref{eq:a'r'}. 
    
    Since the set on the right-hand side of \eqref{eq:a'r'} is a domain, there exists $b<a'$ close enough to $a'$ such that $(b,r')$ is also in this domain. 
    In the case $r'\ge r$, since $b>\cB_{r,r'}(a)$, the perfect flow property shows that $\cB_{r',x}(b)\ge \cB_{r,x}(a)$ for all $x\ge r'$, and since $\cB\le \cR$ by the comparison principle, we also have $\cR_{r',x}(b)\ge \cB_{r,x}(a)$. In the case $r'<r$, $\cR_{r',r}(b)>a$ because $b>\cR^*_{-r,-r'}(a)$, therefore $\cR_{r',x}(b)\ge \cR_{r,x}(a)\ge \cB_{r,x}(a)$ for all $x\ge r$ by the same arguments. So we proved that
    \begin{equation}\label{eq:proof_exit_explore}
        \cR_{r',x}(b)\ge \cB_{r,x}(a),\, \forall\, x\ge \max(r,r').
    \end{equation}

\noindent Since $t$ is an exit time of $\cT$, by Lemma~\ref{l: no max right} (ii) we can find some $t'<t$ such that $\varrho_u\in\cT$ for all $u\in[t',t)$ and $X_{t'}< X_t$. Let $x\in (X_{t'},X_t)$ and $u_x:=\sup\{u \in[t',t)\colon X_u=x\}$.  Since $\varrho_{u_x}\in\cT$, we have $L(t,x)=L(u_x,x)=L(u_x,X_{u_x})<B_{r,x}(a)$. By \eqref{eq:proof_exit_explore}, $L(t,x)<\cR_{r',x}(b)$ which is smaller than $\cR_{r',x}(a'-)$ in view of $b<a'$. Recall that $\cR_{r',x}(a'-)=L(\sigma_{a'}^{r'},x)$ as seen in Section \ref{s:PRBM}. Together with the former inequality, we deduce that $t<\sigma_{a'}^{r'}=s_{\cT'}$ by Proposition~\ref{prop:explore} (i), hence $t\notin \cE$ by definition of $\cE$. 
\end{proof}

\begin{proposition}\label{p:dense times}
    Almost surely,  we have  $\cE^c=\overline{\{s_{\cT},\, \cT \textrm{ spindle}\}} $. Moreover, this set has no isolated points: for any spindle $\cT$, there exist spindles $\cT'$ with $s_{\cT'}<s_{\cT}$ arbitrarily close to $s_{\cT}$.
\end{proposition}

\begin{proof}

We first prove the following lemma.
\begin{lemma}\label{lem:not-exploring}
Almost surely for every $u>t\ge 0$: if ($t\in \mathcal{E}^c$, $\varrho_u$ is in the interior of some spindle $\cT$ and  $X_u>X_t$), then $s_{\cT} \in [t,u)$.
\end{lemma}
\begin{proof}[Proof of the lemma]
 Let $(a,r)$ be the bottom point of $\cT$. We have $s_{\cT}=\sigma_a^r<u<\tau_a^r$ since $\varrho_u$ is in the interior of $\cT$. Suppose $s_{\cT}<t$, hence $X_t\in (r,X_u)$. By Lemma \ref{l: green line} (ii) with $(u,X_u,X_t)$ in place of $(t,x,y)$,  we would have $\cB_{r,X_t}(a)>L(u,X_t)\ge L(t,X_t)\ge L(\sigma_a^r,X_t)=\cR_{r,X_t}(a-)$, hence $\varrho_t\in \cT$ with $t>s_{\cT}$, which yields a contradiction to $t\not\in \mathcal{E}$.
\end{proof}

Let us go back to the proof of the proposition. Since $\cE^c$ includes all times $s_{\cT}$ by \eqref{eq:bottom_gasket} and is a closed set, we deduce the inclusion  $\overline{\{s_{\cT},\, \cT \textrm{ spindle}\}} \subseteq \cE^c$. We now prove the reverse  inclusion, i.e., almost surely, for every $t\ge 0$: if $t\notin \cE$, then $t$ is a point or an accumulation point of the set $\{s_{\cT},\, \cT \textrm{ spindle}\}$.    
    First suppose that there exist $u>t$ arbitrarily close to $t$ such that $X_u>X_t$. Then by Proposition \ref{c:ancestor-time_fixed}, there exist $u'>t$ arbitrarily close to $t$ such that $\varrho_{u'}$ is in the interior of some spindle and $X_{u'}>X_t$. The claim therefore readily follows from Lemma~\ref{lem:not-exploring} in this case.
    
     We next consider the complementary event  when $X_u\le X_t$ for every $u>t$ close enough to $t$.  By Lemma \ref{l: no max right} (iv), $t$ cannot be an exit time. 
     For any rational $u<t$,  $\varrho_{u}$ is in some spindle $\cT$ by Proposition \ref{c:ancestor-time_fixed}, and we let $v>u$ be the time when $\varrho_v$ first exits
    $\cT$ after time $u$. 
   Since $t$ is not an exit time, we must have $v<t$ and thus by another use of Proposition \ref{c:ancestor-time_fixed} and Lemma \ref{l: no max right} (iv), there exists $w \in (v,t)\cap\mathbb{Q}$ such that $X_{w}>X_v$ and $\varrho_w$ is in the interior of some spindle $\cT'$. Recall that exit times are not in $\cE$ by Lemma \ref{l:exit_explore}, hence $v\notin \cE$.
   Applying Lemma~\ref{lem:not-exploring} to $v$ and $w$, we deduce
   that $s_{\cT'} \in [v,w)$. In particular, $s_{\cT'}\in (u,t)$. Since $u$ can be chosen arbitrarily close to $t$, it proves the claim. 

    We have proved so far $\overline{\{s_{\cT},\, \cT \textrm{ spindle}\}} = \cE^c$. Consider now some time $s_{\cT}$. 
    For the second statement of the proposition, by Lemma~\ref{l:exit_explore}, it is enough to prove that there are exit times  $t<s_{\cT}$   arbitrarily close to $s_{\cT}$. If it were not the case,  then $s_{\cT}$ would be itself the exit time of some spindle $\cT'$.
     In particular, the bottom point of $\cT$ would lie on the right boundary of $\cT'$ or be its top point, which  contradicts Lemma~\ref{l:bottom_boundary}. 
    \end{proof}

\begin{corollary}\label{c:s-t}
    Almost surely, we also have $\cE^c=\overline{\{t_{\cT},\, \cT \textrm{ spindle}\}}$. 
\end{corollary}
  \begin{proof}
    One inclusion is clear since $t_{\cT} \notin \cE$ for any spindle $\cT$ by Lemma \ref{l:exit_explore}. For the other, by Proposition
    \ref{p:dense times},  for any spindle $\cT$, there is a sequence $s_{\cT_k}, k\ge 1$ strictly increasing towards $s_{\cT}$. Then the sequence $t_{\cT_k}, k\ge 1$ would also converge to $s_{\cT}$. This implies  $\overline{\{s_{\cT},\, \cT \textrm{ spindle}\}} \subseteq \overline{\{t_{\cT},\, \cT \textrm{ spindle}\}}$ hence the other inclusion by Proposition \ref{p:dense times}. 
  \end{proof}

We show that the range of $\varrho$ restricted to $\cE^c$ is the gasket $\rK$.
\begin{proposition}\label{p:explore-complement}
   Almost surely,  $\{\varrho_t,\, t\in \mathcal{E}^c\}=\rK$, where $\mathscr{K}$ is the gasket defined in \eqref{eq:gasket}. 
\end{proposition}
\begin{proof}
 Recall  the characterization of the gasket $\rK$ after \eqref{eq:gasket}. For any point $(b,x)\in \mathscr{K}$ with $b>0$,
 the hitting time of $(b,x)$ by $\varrho$ is not in $\cE$. This is immediate when $(b,x)$ does not belong to any spindle. The remaining situation is when $(b,x)$ belongs to the left boundary of some spindle $\cT=\cT_{(a,r)}$. Observing that $x> r$ and  $L(s_\cT,x)=\cR_{r,x}(a-)= b$, equation \eqref{eq:range_explorer} shows that  $\varrho$ has already visited $(b,x)$  at time $s_{\cT}$. Therefore the hitting time of $(b,x)$ is strictly smaller than $s_{\cT}$. It implies that $\rK\cap (0,\infty)\times \bR \subset \{\varrho_t,\, t\in \mathcal{E}^c\}$.  Since $\cE^c$ is closed and $\varrho$ is continuous and proper, the set on the right-hand side is closed. Taking the closure, we deduce that $\rK\subset \{\varrho_t,\, t\in \mathcal{E}^c\}$.

     Conversely, $ \{\varrho_{s_{\cT}}\colon \cT \textrm{ spindle}\} \subset \rK$ by \eqref{eq:bottom_gasket}. Taking the closure  and recalling that $\rK$ is closed, Proposition \ref{p:dense times} implies the reverse inclusion. 
\end{proof}

\subsection{Spindles discovered before a given time}

In this subsection we consider the collection of all spindles discovered before a given time in the terminology of Section \ref{s:explore}. Such spindles can be either partially or fully explored.
The main object is to establish in Proposition~\ref{p:green} a decomposition according to the boundary line between the discovered and undiscovered regions, analogous to Proposition~\ref{WN^t}.
Recall the notation \eqref{def:I}.
\begin{definition}[Boundary of discovered spindles]\label{green}
Let $t\ge 0$. We define a process $\cG^t =(\cG^t_x,\, x\ge I_t)$ as follows: for each $x\ge I_t$, 
\begin{enumerate}[(1)]
\item if  $(L(t,x),x)$ is in a spindle $\cT_{(a,r)}$ discovered before time $t$, then $\cG_x^t = \cB_{r,x}(a)$. 
\item Otherwise, $\cG_x^t = L(t,x)$.\pagebreak
\end{enumerate}
\end{definition}

Note that $\cG^t \ge L(t,\cdot)$ by definition. In (1), we could equivalently say that the spindle is partially explored at time $t$ since it is discovered but not fully explored.
Let us make a few more comments on the definition of $\cG^t$. See Figure~\ref{fig:green construction} for an illustration.
\begin{remark}\label{r:green line}
\begin{enumerate}[(i)]
    \item 
    The graph $(\cG^t_x, x)_{x \geq I_t}$ describes the right boundary of the partially explored spindles. The unexplored regions within these spindles are precisely the connected components of 
    \begin{equation}\label{eq:G_t-un_exp}
            \{(b,x)\in \bR_+\times (I_t,\infty) : L(t,x) < b < \cG^t_x\}.
    \end{equation}
    Let us justify this. At time $t$. for a partially explored spindle $\cT_{(a,r)}$ with top point $(c,z)$, let $r_1 := \inf\{x > r : L(t,x) \ge \cB_{r,x}(a)\}$.  
    Then $r_1 = z$ if $\varrho$ is exploring the left part at time $t$, and $r_1\in (r,z)$ if it is exploring the right part.  For $x \in (r, r_1)$, we have $L(t,x) < \cB_{r,x}(a)$, so $(L(t,x), x) \in \cT_{(a,r)}$ and thus $\cG^t_x = \cB_{r,x}(a)$ by definition. 
    Therefore, the set 
    $\{(b,x)\in \bR_+\times \bR : x \in (r, r_1), L(t,x) < b < \cG^t_x\}$is a connected component of  the set in \eqref{eq:G_t-un_exp} and is also the unexplored part of $\cT_{(a,r)}$ (by Proposition~\ref{p:right return}).
    Conversely, any such connected component is the unexplored part of a spindle, as $L(t,x) < b < \cG^t_x$ implies that $(L(t,x), x)$ lies in a partially explored spindle, say $\cT$, and that $(b,x)$ in the same spindle $\cT$ remains unvisited by $\varrho$ at time $t$. The connected component containing $(b,x)$ is the unexplored part of $\cT$ by the former reasoning. 
   
    \item We have $\cG^t_x=L(t,x)$ for all $x\ge H_t$, where 
    \begin{equation}\label{def:H}
        H_t:= \inf\{x\ge X_t \colon \cG^t_x=L(t,x) \}.
    \end{equation}
     Indeed, if $\cG^t_x>L(t,x)$ for a certain $x>H_t$, then $(L(t,x),x)$ is in some spindle $\cT=\cT_{(a,r)}$ discovered before time $t$ by definition of $\cG^t$. By Lemma \ref{l: green line}, $t\in(\sigma_a^r,\tau_a^r)$ and $\cB_{r,y}(a) >L(t,y)$ for every $y\in (r,x]$. Observe that $t>s_{\cT}$ implies that $L(t,y)\ge L(s_\cT,y)=\cR_{r,y}(a-)$ for such $y$. We deduce that $(L(t,y),y)\in \cT$ by definition of $\cT_{(a,r)}$ in \eqref{eq:spindle}, hence $\cG^t_y=\cB_{r,y}(a)>L(t,y)$ for every $y\in (r,x]$ by definition of $\cG^t$. Moreover $X$ is making an excursion above $r$ on the time interval $(\sigma_a^r,\tau_a^r)$ by Proposition \ref{prop:explore} (iii), hence $t\in(\sigma_a^r,\tau_a^r)$ shows that $r<X_t$. Taking $y=H_t\ge X_t>r$ yields $\cG^t_{H_t}>L(t,H_t)$ which contradicts the definition of $H_t$ in \eqref{def:H}.
    
    \item The domain $\{(b,x)\in \bR_+\times \bR\colon 0 \le b<\cG^t_{x}\}$, is left-continuous and non-decreasing in time $t\ge 0$. It absorbs a spindle instantaneously after $\varrho$ discovers it, while remaining unchanged when $\varrho$ is in the interior of a spindle. 

    \item  There  may be points $(L(t,x), x)$, $x\in (I_t,M_t)$ in the notation \eqref{def:I}, lying in a spindle undiscovered at time $t$, hence for which $\cG^t_x=L(t,x)$. It happens when $\varrho$ touches the left boundary of a spindle $\cT_{(a,r)}$ before its discovery time $s_{\cT_{(a,r)}}=\sigma_a^r$, i.e.\ $L(t,x)=L(\sigma_a^r,x)$ with $t<\sigma_a^r$ and $x>r$. 
    Such points must satisfy $x \geq X_t$.  
     Indeed, if $I_t < x < X_t$, then after time $t$, $X$ must cross level $x$ to reach the value $X_{\sigma_a^r} = r<x$ at time $\sigma_a^r$, so $L(\cdot,x)$ must also accumulate at this crossing time, which implies $L(t,x)<L(\sigma_a^r,x)$.
\end{enumerate}
\end{remark}

\begin{figure}[htbp]
\centering
       \scalebox{0.4}{ 
        \def\svgwidth{0.75\columnwidth}
\begingroup%
  \makeatletter%
  \providecommand\color[2][]{%
    \errmessage{(Inkscape) Color is used for the text in Inkscape, but the package 'color.sty' is not loaded}%
    \renewcommand\color[2][]{}%
  }%
  \providecommand\transparent[1]{%
    \errmessage{(Inkscape) Transparency is used (non-zero) for the text in Inkscape, but the package 'transparent.sty' is not loaded}%
    \renewcommand\transparent[1]{}%
  }%
  \providecommand\rotatebox[2]{#2}%
  \newcommand*\fsize{\dimexpr\f@size pt\relax}%
  \newcommand*\lineheight[1]{\fontsize{\fsize}{#1\fsize}\selectfont}%
  \ifx\svgwidth\undefined%
    \setlength{\unitlength}{129.80881939bp}%
    \ifx\svgscale\undefined%
      \relax%
    \else%
      \setlength{\unitlength}{\unitlength * \real{\svgscale}}%
    \fi%
  \else%
    \setlength{\unitlength}{\svgwidth}%
  \fi%
  \global\let\svgwidth\undefined%
  \global\let\svgscale\undefined%
  \makeatother%
  \begin{picture}(1,1.56326938)%
    \lineheight{1}%
    \setlength\tabcolsep{0pt}%
    \put(0,0){\includegraphics[width=\unitlength,page=1]{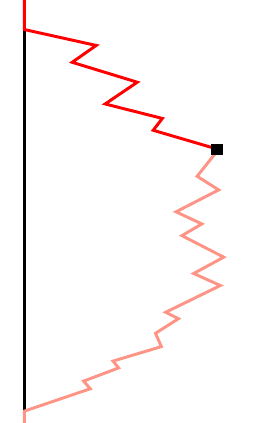}}%
    \put(0.68129584,0.97484826){\makebox(0,0)[lt]{\lineheight{1.25}\smash{\begin{tabular}[t]{l}\Huge $\varrho_t$\end{tabular}}}}%
    \put(-0.00201666,0.02379298){\makebox(0,0)[lt]{\lineheight{1.25}\smash{\begin{tabular}[t]{l}\Huge $I_t$\end{tabular}}}}%
    \put(0,0){\includegraphics[width=\unitlength,page=2]{green_construction1.pdf}}%
  \end{picture}%
\endgroup%

    }
    \hfill
    \scalebox{0.4}{ 
        \def\svgwidth{0.75\columnwidth}
\begingroup%
  \makeatletter%
  \providecommand\color[2][]{%
    \errmessage{(Inkscape) Color is used for the text in Inkscape, but the package 'color.sty' is not loaded}%
    \renewcommand\color[2][]{}%
  }%
  \providecommand\transparent[1]{%
    \errmessage{(Inkscape) Transparency is used (non-zero) for the text in Inkscape, but the package 'transparent.sty' is not loaded}%
    \renewcommand\transparent[1]{}%
  }%
  \providecommand\rotatebox[2]{#2}%
  \newcommand*\fsize{\dimexpr\f@size pt\relax}%
  \newcommand*\lineheight[1]{\fontsize{\fsize}{#1\fsize}\selectfont}%
  \ifx\svgwidth\undefined%
    \setlength{\unitlength}{129.80881939bp}%
    \ifx\svgscale\undefined%
      \relax%
    \else%
      \setlength{\unitlength}{\unitlength * \real{\svgscale}}%
    \fi%
  \else%
    \setlength{\unitlength}{\svgwidth}%
  \fi%
  \global\let\svgwidth\undefined%
  \global\let\svgscale\undefined%
  \makeatother%
  \begin{picture}(1,1.56326938)%
    \lineheight{1}%
    \setlength\tabcolsep{0pt}%
    \put(0,0){\includegraphics[width=\unitlength,page=1]{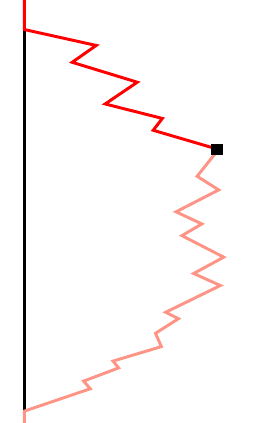}}%
    \put(0.68129584,0.97484826){\makebox(0,0)[lt]{\lineheight{1.25}\smash{\begin{tabular}[t]{l}\Huge $\varrho_t$\end{tabular}}}}%
    \put(-0.00201666,0.02379298){\makebox(0,0)[lt]{\lineheight{1.25}\smash{\begin{tabular}[t]{l}\Huge $I_t$\end{tabular}}}}%
    \put(0,0){\includegraphics[width=\unitlength,page=2]{green_construction2.pdf}}%
  \end{picture}%
\endgroup%

    }
    \hfill
    \scalebox{0.4}{ 
        \def\svgwidth{0.75\columnwidth}
\begingroup%
  \makeatletter%
  \providecommand\color[2][]{%
    \errmessage{(Inkscape) Color is used for the text in Inkscape, but the package 'color.sty' is not loaded}%
    \renewcommand\color[2][]{}%
  }%
  \providecommand\transparent[1]{%
    \errmessage{(Inkscape) Transparency is used (non-zero) for the text in Inkscape, but the package 'transparent.sty' is not loaded}%
    \renewcommand\transparent[1]{}%
  }%
  \providecommand\rotatebox[2]{#2}%
  \newcommand*\fsize{\dimexpr\f@size pt\relax}%
  \newcommand*\lineheight[1]{\fontsize{\fsize}{#1\fsize}\selectfont}%
  \ifx\svgwidth\undefined%
    \setlength{\unitlength}{130.14702744bp}%
    \ifx\svgscale\undefined%
      \relax%
    \else%
      \setlength{\unitlength}{\unitlength * \real{\svgscale}}%
    \fi%
  \else%
    \setlength{\unitlength}{\svgwidth}%
  \fi%
  \global\let\svgwidth\undefined%
  \global\let\svgscale\undefined%
  \makeatother%
  \begin{picture}(1,1.55925101)%
    \lineheight{1}%
    \setlength\tabcolsep{0pt}%
    \put(0,0){\includegraphics[width=\unitlength,page=1]{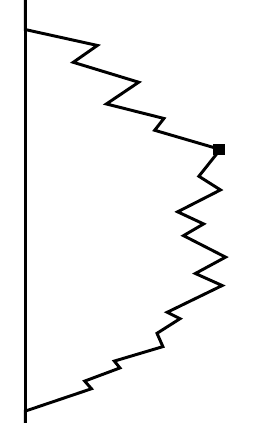}}%
    \put(0.68212413,0.97231491){\makebox(0,0)[lt]{\lineheight{1.25}\smash{\begin{tabular}[t]{l}\Huge $\varrho_t$\end{tabular}}}}%
    \put(-0.00201568,1.22013017){\makebox(0,0)[lt]{\lineheight{1.25}\smash{\begin{tabular}[t]{l}\Huge $H_t$\end{tabular}}}}%
    \put(0,0){\includegraphics[width=\unitlength,page=2]{green_construction3.pdf}}%
    \put(0.00058733,0.0237311){\makebox(0,0)[lt]{\lineheight{1.25}\smash{\begin{tabular}[t]{l}\Huge $I_t$\end{tabular}}}}%
    \put(0,0){\includegraphics[width=\unitlength,page=3]{green_construction3.pdf}}%
    \put(0.67168562,1.25974192){\makebox(0,0)[lt]{\lineheight{1.25}\smash{\begin{tabular}[t]{l}\Huge \textcolor{green}{$\cG^t$}\end{tabular}}}}%
    \put(0.49929225,0.61966755){\makebox(0,0)[lt]{\lineheight{1.25}\smash{\begin{tabular}[t]{l}\Huge $L(t,\cdot)$\end{tabular}}}}%
  \end{picture}%
\endgroup%

    }
    \caption{In the left picture, we represent some spindles hitting $L(t,\cdot)$. In the middle picture, we only keep spindles which are partially explored at time $t$, hence discarding for example the top spindle as in Remark~\ref{r:green line} (iv), then trace the parts of the  blue/right boundaries of the remaining spindles which lie at the right of $L(t,\cdot)$. In the right picture, $\cG^t$ is the concatenation of all the corresponding blue lines, merging with $L(t,\cdot)$ above $H_t$ by Remark~\ref{r:green line} (ii).  }
    \label{fig:green construction}
\end{figure}

We introduce the process
$\cY:=(\cY_x,\, x\ge I_T)$, where
\begin{equation}\label{def:Y}
\cY_x = \cG^T_{x}-L(T,x),\qquad x\ge I_T.
\end{equation}
\begin{proposition}\label{p:green}
Fix $(b,s)\in \bR_+\times \bR$. Let $T=\tau_b^s$ and $W^T$ given by \eqref{eq:WT}.  The process $\cY$ is  a non-killed ${\rm BESQ}(2-\delta\, |_{s}\, -\delta)$ process, which is the flow line starting at $(0,I_T)$ of the non-killed version of the flow $\cB-L(T,\cdot)$ given by Proposition~\ref{WN^t} 
In particular, $\cY$ is driven by $W^T$.
\end{proposition}

\begin{proof}
Let  $\cB^T:=\cB-L(T,\cdot)$. 
Recall by Proposition \ref{WN^t} that $\cB^T$ is a killed ${\rm BESQ}(0\, |_{I_T}\, 2-\delta\, |_{s}\, -\delta)$ flow driven by $W^T$; then we denote by  $\widetilde{\cB}^T$ the non-killed version of $\cB^T$ and  by $\widetilde{\cY}$ the flow line of $\widetilde{\cB}^T$ starting from $(0,I_T)$.  We want to show that $\cY=\widetilde\cY$. 

For any $x\ge I_T$ with $\cY_x>0$, by definition of $\cG^T$, $(L(T,x),x)$ is in some spindle $\cT_{(a,r)}$ which is partially explored at time $T$, with $L(T,r) =a$ and $L(T,x)<\cG^T_x=\cB_{r,x}(a)$. By Lemma \ref{l: green line} (ii), it implies that $\cB_{r,y}(a)-L(T,y)>0$ for all $y\in (r,x]$. This is to say, the flow line of $\cB^T$ starting from $(0,r)$ does not touch $0$ on $(r,x]$ and thus $\widetilde{\cB}^T_{r,x}(0)=\cB^T_{r,x}(0)$. We conclude that 
\[
\widetilde{\cB}^T_{r,x}(0)=\cB^T_{r,x}(0) =\cB_{r,x}(L(T,r))-L(T,x) =\cB_{r,x}(a)-L(T,x)= \cG^T_x - L(T,x) = \cY_x. 
\]
Since $\widetilde{\cY}_r\ge 0=\widetilde{\cB}^T_{r,r}(0)$, the perfect flow property of $\widetilde{\cB}^T$ yields $\widetilde{\cY}_x\ge \widetilde{\cB}^T_{r,x}(0)=\cY_x$. This proves  $\widetilde \cY\ge \cY$.

Note that $\cY_{I_T}=\widetilde{\cY}_{I_T}=0$. It remains to show $\cY_x\ge \widetilde{\cY}_x$ for any $x>I_T$. Let $r_x:=\sup\{y\in[I_T,x]: \widetilde{\cY}_y=0\}$. Since $\widetilde \cY\ge \cY\ge 0$, we have $\cY_x=\widetilde{\cY}_x=0$ if $r_x=x$. So we only need to consider the case $r_x<x$, in which
\begin{equation}\label{eq:Ytilde nonzero}
    \widetilde{\cY}_y>0, \;\forall \,y\in (r_x,x].
\end{equation}
We first observe that $r_x< X_T=s$ in this case. Indeed, if $r_x\ge s$, then since $\widetilde{\cB}^T$ restricted on $[s,\infty)$ is a ${\rm BESQ}^{-\delta}$ flow, it would imply  $\widetilde{\cY}_y=0$ for all $y\ge r_x$: in particular, we would have $\cY_{x}=\widetilde{\cY}_x=0$, a contradiction to $r_x<x$. 
Moreover, \eqref{eq:Ytilde nonzero} implies $r_x > I_T$ almost surely, since $\widetilde{\cY}$ touches $0$ infinitely often near $I_T$. So we must have $r_x \in (I_T,s)$. 

Now let $a_x := L(T,r_x)$. Since $r_x \in (I_T,X_T)$, we have  $a_x>0$ by \eqref{no_increase}, which implies $T \geq \sigma_{a_x}^{r_x}$. 
As $\widetilde{\cY}_{r_x} =0$, the perfect flow property of $\widetilde{\cB}^T$ gives $\widetilde{\cY}_y = \widetilde{\cB}^T_{r_x,y}(0)$ for all $y \geq r_x$. From \eqref{eq:Ytilde nonzero}, we have for $y \in (r_x,x]$:
\begin{equation}\label{eq:green_proof}
0 <\widetilde{\cY}_y  = \widetilde{\cB}^T_{r_x,y}(0) = \cB^T_{r_x,y}(0) = \max(\cB_{r_x,y}(a_x) - L(T,y),0).
\end{equation}
In particular, we have $ \cB_{r_x,y}(a_x) - L(T,y)>0$. So we conclude that 
\[
\cR_{r_x,x}(a_x-) = L(\sigma_{a_x}^{r_x},x) \leq L(T,x) < \cB_{r_x,x}(a_x).
\]
This shows $(L(T,x),x)\in \cT_{(a_x,r_x)}$ in the notation of \eqref{eq:spindle}.  Since $\cT_{(a_x,r_x)}$ is contained in a spindle, we deduce that $(L(T,x),x)$ belongs to a spindle discovered before time $T$, and thus $\cG^T_x\ge \cB_{r_x,x}(a_x)$ by definition of $\cG^T$. It follows that $\cY_x=\cG^T_x-L(T,x)\ge \cB^T_{r_x,x}(0)=\widetilde{\cY}_x$ where the last equality comes from \eqref{eq:green_proof} with $y=x$. The proof is complete.
 \end{proof}

Fix $(b,s)\in \bR_+\times \bR$. With $T=\tau_b^s$, define $W^{T,-},W^T$ as in \eqref{eq:WT-} and \eqref{eq:WT}. 
Recall by Proposition \ref{p:markov S} that $W^{T}$ is a white noise independent of $W^{T,-}$, and that $L(T,\cdot)$ and $I_T$ are measurable w.r.t.\ $W^{T,-}$. 
 Let $\widehat{W}^{T,-}$ be $W^{T,-}_\cY$ as defined  in  \eqref{eq:W-} with  $\cW=W^T$ and $Y=\cY$ given by \eqref{def:Y}, i.e.  
\[
\widehat{W}^{T,-} (\dd \ell, \dd x) = W^T (\dd \ell, \dd x), \qquad \ell \le  \cY_x,\ x\ge I_T. 
\]
We furthermore define
\begin{equation}\label{eq:G^bs}
\rG^{b,s}:=\sigma(W^{T,-}, \widehat{W}^{T,-}).
\end{equation} 
We refer to Figure \ref{f:Gbs} for a description of this $\sigma$-field. Intuitively, $\rG^{b,s}$ collects the information at the left of
the line $\cG^T$. For instance, we will show in Proposition \ref{c:measurability} that all spindles discovered before time $T$ are measurable w.r.t. $\rG^{b,s}$.  
For future purpose (in particular used in the proof of Theorem~\ref{thm:PPP}), we also consider sometimes 
\begin{equation}\label{eq:G^bs_y}
\rG^{b,s}_y:=\sigma\Big(W^{T,-}, \widehat{W}_{|(-\infty, y)}^{T,-}\Big), \qquad y\in \mathbb{R}, 
\end{equation} 
where $\widehat{W}_{|(-\infty, y)}^{T,-}$ is the trace of $\widehat{W}^{T,-}$ on $\mathbb{R}_+\times(-\infty, y)$. Informally speaking, $\rG^{b,s}_y$ is $\rF_{T}$ together with the $\sigma$-field at the left of $(\cG^{T}_s,\, s\le y)$, where we recall that $\rF$ is the natural filtration of the Brownian motion.

 \begin{figure}[htbp]
\centering
       \scalebox{0.4}{
        \def\svgwidth{\columnwidth}
        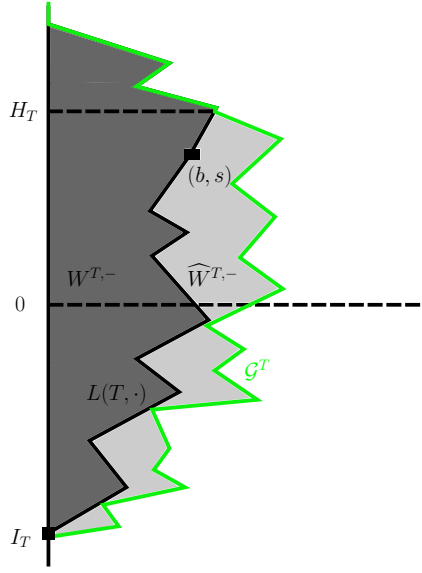
    }
\caption{For $T=\tau_b^s$, the two lines $L(T,\cdot)$ and $\cG^T$ are respectively drawn in black and in green in this figure. The dark and light grey areas represent the martingale measures $W^{T,-}$ and $\widehat{W}^{T,-}$.}
\label{f:Gbs}
\end{figure}

We define now the noise which will drive the flows at the right of $\cG^T$. To ensure some invariance in distribution illustrated in the upcoming proposition, we will rather work with a translation of this noise. With the hitting time $H_T$ defined as in \eqref{def:H}, we set
\begin{equation}\label{eq:W-hat}
   \widehat{W}^T (\dd \ell, \dd x) = 
   \begin{cases}
        W(\dd \ell, H_T+ \dd x) &  \ell\ge 0,  x< I_T-H_T,\\
        W^T(\cY_{x+H_T} + \dd \ell, H_T +\dd x), & \ell\ge 0,  x> I_T-H_T.
    \end{cases}
\end{equation}
In other words, in the notation \eqref{eq:WN translation}, $\widehat{W}^T$ is $\theta_{-H_T}W$ on $(-\infty,I_T-H_T)$ and is $\theta_{-H_T}W^{T,+}_\cY$ on $(I_T-H_T,\infty)$, where $W^{T,+}_\cY$ is defined as in  \eqref{eq:W+} with $\cW=W^T$ and $Y=\cY$. \pagebreak

In the notations \eqref{eq:S translation} and \eqref{eq:S-g}, we introduce the flows $\widehat{\cB}^T:=\theta_{-H_T}(\cB - \cG^T) $ and $\widehat{\cR}^T:=\theta_{-H_T} (\cR - \cG^T) $, where we implicitly set $\cG^T_x:=0$ for all $x<I_T$.

\begin{proposition}[Decomposition at $\cG^T$]\label{WN:Green}
Fix $(b,s)\in \bR_+\times\bR$ and let $T=\tau_b^s$. Let $\widehat{W}^T$ be defined as in \eqref{eq:W-hat}. 
Then $\widehat{W}^T$ is a white noise on $\bR_+\times \bR$, independent of $\rG^{b,s}$. 
Moreover,  the flows $\widehat{\cB}^T$ and $\widehat{\cR}^T $ defined above are respectively a ${\rm BESQ}(0\,|_{0}\,-\delta)$ flow and  a  ${\rm BESQ}(\delta\,|_{0}\,0)$ flow driven by $\widehat W^T$. In particular, $(\widehat{\cB}^T,\widehat{\cR}^T)$ has the same distribution as $(\cB,\cR)$.
\end{proposition}

\begin{proof}
 Recall that $W^T$  is independent of $(W^{T,-},I_T)$. By Proposition~\ref{p:green}, the process $\cY$ in \eqref{def:Y} is
 a non-killed $\besq$ flow line driven by $W^T$ starting from $(0,I_T)$. Let $W_\cY^{T,+}$ be the white noise  on $\bR_+ \times [I_T, \infty)$ defined as in  \eqref{eq:W+} with $\cY$ and $W^T$. Applying Proposition~\ref{p:decomp} (iii), the white noise $W^{T,+}_{\cY}$   is independent of $\widehat{W}^{T,-}=W^{T,-}_{\cY}$. In view of the definition of $\rG^{b,s}$ in \eqref{eq:G^bs}, we deduce that the concatenation 
\[
\begin{cases}
    W^{T,+}_\cY &\textrm{on  }  (I_T, \infty), \\
    W^T=W &\textrm{on } (-\infty, I_T),
\end{cases}
\]
 yields a white noise on $\bR_+\times \bR$ independent of $\rG^{b,s}$. By  Proposition~\ref{p:decomp} (ii),  $\cY$  is measurable w.r.t. $\rG^{b,s}$. Since $H_T=\inf\{x\ge s\colon \cY_x=0\}$ by definition \eqref{def:H}, so is $H_T$. Translating the latter white noise by $-H_T$ spatially as in \eqref{eq:WN translation}, we still get a white noise on $\bR_+\times \bR$ independent of $\rG^{b,s}$, which 
coincides with $\widehat{W}^T$ defined as in \eqref{eq:W-hat}.

 We first consider $(I_T,\infty)$. By Proposition \ref{WN^t} (ii), on $(I_T,\infty)$, $\cB-L(T,\cdot)$ is a killed ${\rm BESQ}( 2-\delta\, |_{s}\, -\delta)$ flow driven by $W^T$, while by Proposition \ref{p:green}, the process $\cY$ is the flow line starting at $(0,I_T)$ of the non-killed ${\rm BESQ}( 2-\delta\, |_{s}\, -\delta)$ flow driven by $W^T$, which vanishes when hitting 0 at $H_T\ge s$. By Proposition~\ref{p:decomp}~(i) with $(\cB-L(T,\cdot),\cY)$ in place of  $(\cS,Y)$, $\cB-\cG^T$ is a ${\rm BESQ}(0\, |_{H_T}\, -\delta)$ flow driven by $W_\cY^{T,+}$. 
 Although the proposition was stated rigorously for the non-killed version of $\cB - L(T, \cdot)$, the conclusion holds because the flow lines of both processes stay below $\cY$ after their first meeting point.
 
 Similarly, $\cR - L(T,\cdot)$ is on $(I_T,\infty)$ a ${\rm BESQ}(2\, |_{s}\, 0)$ flow driven by $W^T$ by Proposition \ref{WN^t} (i), hence $\cR-\cG^T$ is a ${\rm BESQ}(\delta\,|_{H_T}\,0)$ flow driven by $\theta_{H_T}\widehat{W}^T$ by applying Proposition \ref{p:decomp} (i) to $(\cR-L(T,\cdot),\cY)$ in place of  $(\cS,Y)$.

Finally, on $(-\infty,I_T)$, $\cG^T=0$ and we recall that  $\cB$, resp. $\cR$, is a ${\rm BESQ}^0$, resp. ${\rm BESQ}^{\delta}$ flow driven by $W=\theta_{H_T}\widehat{W}^T$.
Applying the perfect flow property at level $I_T$ and translating both flows by $-H_T$ yields the result.
\end{proof}

\begin{proposition}\label{c:measurability}
Fix $(b,s)\in (0,\infty)\times \bR$ and let $T=\tau_b^s$. Then the collection of spindles $\cT$ discovered before time $T$, their entrance times $s_{\cT}$ and their first exit times $t_{\cT}$ are measurable with respect to $\rG^{b,s}$. 

More precisely, with $\rG_s^{b,s}$ denoting the $\sigma$-field defined in \eqref{eq:G^bs_y}, the following are $\rG_s^{b,s}$-measurable: 
\begin{enumerate}[(i)]
\item the bottom points $(a,r)$ of any spindle discovered before time $T$,
\item the processes $(\cR_{r,x}(a-),\, x\ge r)$ for $(a,r)$ as in (i),
\item the spindles discovered before time $T$ which do not contain $(b,s)$.
\end{enumerate}
\end{proposition}

Note that, we restrict $b$ to $(0, \infty)$ in Proposition~\ref{c:measurability} to simplify the proof. This excludes the degenerate case $b=0$, which would require a separate discussion and is not useful for our later purposes.
The proof of Proposition~\ref{c:measurability} is postponed to the Appendix~\ref{sec:appendix}.

We finish the section with a lemma on spindles discovered after time $T$.
\begin{lemma}\label{l:spindles after T}
    Fix $(b,s)\in (0,\infty)\times\bR$ and let $T=\tau_b^s$. Almost surely, bottom points of spindles discovered after time $T$ do not lie on $(\cG^T_x,x)_{x\ge I_T}$.
\end{lemma}
\begin{proof}
    If $\cG^T_x>L(T,x)$, $(\cG^T_x,x)$ is on the blue boundary of a spindle hence cannot be the bottom point of a spindle by Lemma \ref{l:bottom_boundary}. In the remaining case $\cG^T_x=L(T,x)$. If $x>s$, $L(T,x)=\cR_{s,x}(b)$ and we use again Lemma \ref{l:bottom_boundary}  to see that $(\cG^T_x,x)$ is not the bottom point of a spindle. When $x\in (I_T,s)$, $\sigma_a^x<T$ with $a=L(T,x)$, hence  any spindle of the form $\cT_{(a,x)}$ has been discovered  before time $T$.
\end{proof}

\section{Embedding of the marked L\'evy process}\label{sec:main}

\subsection{The regenerative property}\label{sec:regenerative}

Recall from Proposition \ref{prop:explore} that, for every spindle $\cT$, after time $s_{\cT}$ when the exploration process $\varrho$ first discovers $\cT$ at its bottom point, $\varrho$ stays in the spindle, exploring the left part until it reaches the top point at time $t_{\cT}$.  We consider the collection of the time intervals $(s_{\cT},t_{\cT})$ and define
\begin{equation}\label{def:A-kappa}
    \mathcal{O}:=\bigcup_{\text{spindle } \cT} (s_{\cT},t_{\cT}), \quad
A_t:=\int_0^t \ind{s\in \mathcal{O}} \dd s,\quad 
\kappa_u:=\inf\{t\ge 0: A_t>u\}.
\end{equation}
The time-change $\kappa$ ignores the times when $\varrho$ is in the right part of a spindle, such that each jump of $\kappa$ corresponds to an excursion described in Proposition~\ref{p:right return}, as indicated by the following lemma.  Recall the set $\mathcal{E}$ defined in \eqref{def:explore}.

\begin{lemma}\label{lem:kappa-jump}
     Almost surely, for every $u>0$ with $\kappa_{u-} < \kappa_u$, 
    there exists a spindle $\cT$ such that $\varrho$ is making an excursion in the right part of $\cT$ during the time interval $[\kappa_{u-},\kappa_u]$, as described in Proposition~\ref{p:right return}. In particular, $\varrho_{\kappa_{u-}} = \varrho_{\kappa_{u}}$ is on the right boundary of $\cT$, $\varrho_t\in \overset{\circ}{\cT}$ for every $t\in (\kappa_{u-} , \kappa_u)$ and $\kappa_{u-},\kappa_u\notin\cE$.

\end{lemma}
\begin{proof}
     By Proposition~\ref{c:ancestor-time_fixed}, one can find $t\in (\kappa_{u-},\kappa_u)$ such that $\varrho_t$ is in the interior of some spindle $\cT$. Moreover,  $\varrho_t$ is exploring the right part of $\cT$, because otherwise $A$ is strictly increasing at $t$ and $\kappa$ cannot have a jump. Then there exist $t_0<t<t_1$ given by Proposition~\ref{p:right return}, such that $ \varrho_{t_0} = \varrho_{t_1}$ lies on the right boundary of $\cT$ and $\varrho$ is inside $\cT$ during the time interval $(t_0, t_1)$. 
     Hence $A$ stagnates during $(t_0, t_1)$, and we must have $(t_0, t_1) \subseteq (\kappa_{u-},\kappa_u)$. 
     On the other hand, we observe that $t\in\cE$ for all $t\in(t_0,t_1)$ while  $t_1\notin \cE$ by 
     Lemma \ref{l:exit_explore}, and therefore $t_0\notin \cE$ since $\varrho_{t_0}=\varrho_{t_1}$ and $t_0\le t_1$.  By Proposition~\ref{p:dense times}, there exist two sequences of spindles $(\cT_{1,n})_{n\ge 1}$ and $(\cT_{2,n})_{n\ge 1}$ such that $s_{\cT_{n,1}}\uparrow t_0$ and $s_{\cT_{n,2}}\downarrow t_1$ as $n\to +\infty$. We deduce that $A_t<A_{t_0}$ for all $t<t_0$ while $A_t>A_{t_1}$ for all $t>t_1$, leading to $t_0=\kappa_{u-}$ and $t_1=\kappa_u$.
\end{proof}

\begin{lemma} \label{l:kappa_A} Almost surely:

(i) The process $\varrho\circ \kappa$ is continuous.

(ii) For any spindle $\cT$, $\kappa(A_{s_{\cT}})=s_{\cT}$ and $\kappa(A_{t_{\cT}})=t_{\cT}$.
\end{lemma}
\begin{proof}
Statement (i) follows from Lemma \ref{lem:kappa-jump}. Let us prove (ii). Since $A_t>A_{s_{\cT}}$ for any $t>s_{\cT}$ by definition of $A$, we have $\kappa(A_{s_{\cT}})=s_{\cT}$. 
 Since $t_\cT\notin\cE$, by Proposition~\ref{p:dense times}, there exists a sequence of spindles $(\cT_n)_{n\ge 1}$ such that $s_{\cT_n}\downarrow t_\cT$. It implies that for all $t>t_\cT$, there exists some $s_{\cT_n}\in(t_\cT,t)$ such that $A_t> A_{s_{\cT_n}}\ge A_{t_\cT}$, where the first inequality is strict since $t\mapsto A_t$ is strictly increasing on $(s_{\cT_n},\min\{t,t_{\cT_n}\})$.
Therefore $\kappa(A_{t_{\cT}})=t_{\cT}$. 
\end{proof}

Define $\mathcal{A}$ to be the closure of the jump times of $\kappa$, i.e. 
 \begin{equation}\label{eq:cA}
     \cA=\overline{\{ u\ge 0: \kappa_u \ne \kappa_{u-} \}}. 
 \end{equation}

Recall that our model is invariant with the space-time plane scaled by a factor $c>0$ as in \eqref{eq:scaling-inv}, which corresponds to $\varrho^{(c)}_t =c^{-1} \varrho_{c^2 t}$. 
Then $\cA^{(c)} = c^{-2}\cA$ is associated with $\varrho^{(c)}$. 
We deduce the following result with $b=c^{-2}$:
 
 \begin{lemma}[Scaling-invariance]\label{l:scaling A}
For any $b>0$, the set $\{b u \colon  u\in \cA\}$ has the same distribution as $\cA$.    
 \end{lemma}

The next result gives a characterization of the set $\cA$. 

\begin{proposition}\label{p:charact A}
 Almost surely:

(i) We have the identity $\cA=\{A_t\colon t\notin \mathcal{E}\}$. In particular,
\begin{equation}\label{eq:Set-A}
\mathcal{A} = \overline{\{A_{s_{\cT}}\colon \, \cT \textrm{ spindle} \}} = \overline{\{A_{t_{\cT}}\colon\, \cT \textrm{ spindle} \}}. 
\end{equation}

(ii) $\mathcal{A}^c = \bigcup_{\textrm{spindle }\cT } (A_{s_{\cT}}, A_{t_{\cT}})$. 

(iii) $\mathcal{A}$ has Lebesgue measure zero and has no isolated points. 

(iv) $\{\kappa_{u-},\kappa_{u}, \,u\in \cA\}=\cE^c$. 

\end{proposition}

\begin{proof}

    (i)  Let $u\ge 0$ such that $\kappa_{u-}\neq \kappa_u$. 
    We have $u=A_{\kappa_u}$ by continuity of $A$ and $\kappa_u\not\in \mathcal{E}$ by Lemma~\ref{lem:kappa-jump}. 
    Recall that $\mathcal{E}^c$ is closed. 
     Since $A_\infty=\infty$ and $A$ is continuous, the set $\{A_t\colon t\notin \mathcal{E}\}$ is a closed set, and we deduce that  $\cA\subseteq \{A_t\colon t\notin \mathcal{E}\}$. 
    
    To prove the reverse inclusion, let us consider any $t\notin \mathcal{E}$.  By Corollary 
    \ref{c:s-t}, there are spindles $\cT'$ with  first exit times $t_{\cT'}$ arbitrarily close to $t$.  By Lemma \ref{l: no max right} (iii), there exists $u>t_{\cT'}$ arbitrarily close to $t_{\cT'}$ such that $X_u<X_{t_{\cT'}}$, implying that $\varrho$ returns to $\cT'$ at times arbitrarily close to $t_{\cT'}$. Thus,  
    there exists  $s>t_{\cT'}$ arbitrarily close to $t_{\cT'}$ such that 
    $\varrho_s\in \cT'$. For such $s$, we have $\kappa_v>\kappa_{v-}$ where $v=A_s$, hence $A_s\in \cA$. Since $s$ can be taken arbitrarily close to $t$, by continuity of $A$ and the fact that $\cA$ is closed, we deduce that $A_t \in \cA$. This proves that $\{A_t\colon t\notin \mathcal{E}\}\subseteq \cA$. Equation \eqref{eq:Set-A} then follows from Proposition \ref{p:dense times} and Corollary~\ref{c:s-t}.

    (ii) If $\varrho_t$ is in the right part of a spindle, then $A_t=A_{t_1}$ in the notation of Proposition \ref{p:right return}, and $t_1\notin \mathcal{E}$ by Lemma \ref{l:exit_explore}. It implies that $A_t \in \cA$ by (i).  Therefore $\cA$ coincides with the set of $A_t$'s for $t\notin \mathcal{E}$ or $t$ in the right part of a spindle, i.e. $\cA=\{A_t,\, t\notin \mathcal{O}\}$ in the notation of \eqref{def:A-kappa}. Since $A$ is strictly increasing on the open set $\mathcal{O}$, we have $\cA=A(\mathcal{O}^c)=(A(\mathcal{O}))^c$, where $A(C)$ denotes the image of $C$ by the map $t\mapsto A_t$.

    (iii) 
    By Proposition \ref{c:ancestor-time_fixed}, the set $\{ t : t \notin \mathcal{E} \}$ has Lebesgue measure zero. From (i) and the absolute continuity of $t \mapsto A_t$, it follows that $\cA$ also has Lebesgue measure zero. Moreover, $\cA$ has no isolated points due to \eqref{eq:Set-A} and the fact that $\{ A_{s_{\cT}} : \cT \text{ spindle} \}$ has no isolated points. Indeed, $\{ s_{\cT} : \cT \text{ spindle} \}$ has no isolated points by Proposition \ref{p:dense times}, and $t \mapsto A_t$ is strictly increasing on this set.

    (iv) For any $u>0$ such that $\kappa_{u-}\neq \kappa_u$, we have 
    $\kappa_u,\kappa_{u-}\in \cE^c$ by Lemma \ref{lem:kappa-jump}. For $u\in \cA$ such that $\kappa_{u-}=\kappa_u$, $u$ is an accumulation point of the set $\{v\colon \kappa_{v-}\neq \kappa_v\}$ by definition of $\cA$. Since $\cE^c$ is closed, and $\kappa$ is continuous at $u$, we deduce that $\kappa_u \in \cE^c$. This implies $\{\kappa_{u-}, \kappa_u\}\subseteq \cE^c$.  
    
    For the converse inclusion, let $t\in \cE^c$ and $u =A_t$. Then $u\in \cA$ by (i).  By definition of $\kappa$, $\kappa_{u-}\le t \le \kappa_u$. In the case $\kappa_{u-}=\kappa_u$, we have $\kappa_u=t$. Otherwise $\kappa_{u-}<\kappa_u$ and   $(\kappa_{u-},\kappa_u)\subseteq \cE$ by Lemma \ref{lem:kappa-jump}, hence $t=\kappa_{u-}$ or $t=\kappa_u$. This completes the proof. 
 \end{proof}


\begin{lemma}\label{lem:kappa-u-fix}
Fix $u> 0$. 
Almost surely, we have $u\in \cA^c$. 
\end{lemma}
\begin{proof}
Since $\cA$ almost surely has Lebesgue measure zero by Proposition~\ref{p:charact A} (iii), we deduce by Fubini's theorem that 
\[
\int_0^{\infty} \bP(u\in \cA) \dd u 
=\mathbb{E}\Big[\int_0^{\infty} \ind{u\in \cA} \dd u\Big] =0 . 
\]
So $\bP(u\in \cA) = 0$ for Lebesgue-almost every $u>0$. 
    Because of the scaling-invariance, we have  
$\bP(u\in \cA)  = \bP(1\in \cA)$ for every $u>0$. 
So  $\bP(u\in \cA)=0$  for every $u>0$. 
\end{proof}

 We can now give further properties of $\cA$.  
To this end, let us introduce a new filtration. 
Recall the $\sigma$-algebra $\rG^{b,s}$ defined in \eqref{eq:G^bs}. For fixed $u\ge 0$, we define 
\begin{equation}\label{eq:G_u}
    \rG_u := \left\{E\in \rF_\infty : E\cap \{\kappa_u \le \tau_b^s\} \in \rG^{b,s},\, \forall\, (b,s)\in (0,\infty)\times \bR  \right\}.
\end{equation}

\begin{lemma}\label{l:Gu-filtration}
With the above definition, $(\rG_u,\,u\ge 0)$ defines a filtration. The random variable $\kappa_u$,  the collection of spindles $\cT$ discovered before $\kappa_u$, their entrance times $s_\cT$ and exit times $t_{\cT}$ together with their images $A_{s_{\cT}}$, $A_{t_{\cT}}$ by $A$, are measurable w.r.t. $\rG_u$.
\end{lemma}
\begin{proof}
We first prove that for any $u\ge 0$ and $(b,s), (b',s') \in (0,\infty)\times \bR$,
    \begin{align}
        \{\kappa_u \le \tau_b^s\} & \in \rG^{b,s} \label{eq:kappau-bs},\\
        \label{eq:kappau-bs-b's'}
    \{\kappa_u\le \min(\tau_b^s,\tau_{b'}^{s'})\} &\in \rG^{b,s}\cap\rG^{b',s'}.
    \end{align}
The case $u=0$ is trivial since $\kappa_0=0$, so we assume $u>0$. On the event $\{u\in \cA^c\}$, $t\to A_t$ is strictly increasing around $t=\kappa_u$ and hence $\{\kappa_u \le \tau_b^s\} = \{ u\le A_{\tau_b^s}\}$. Observing that a.s.  
\begin{equation}\label{eq:A-taubs}
A_{\tau_b^s}=\sum_{\cT'\colon s_{\cT'}<\tau_b^s} (\min(t_{\cT'},\tau_{b}^s)-s_{\cT'}),
\end{equation}
 we deduce by Proposition \ref{c:measurability} that $ \{ u\le A_{\tau_b^s}\}\in \rG^{b,s}$. We conclude with Lemma \ref{lem:kappa-u-fix}. Furthermore, replacing  $\tau_b^s$ in \eqref{eq:A-taubs} with $\min(\tau_b^s,\tau_{b'}^{s'})$ shows that, up to a negligible set, 
\[
    \{\kappa_u\le \min(\tau_b^s,\tau_{b'}^{s'})\}=\Big\{u\le A_{\min(\tau_b^s,\tau_{b'}^{s'})}\Big\} \in \rG^{b,s}\cap\rG^{b',s'}.
\]
We turn now to the proof of the lemma.
In view of \eqref{eq:kappau-bs}, it is plain to check that $\rG_u$ is a $\sigma$-algebra. We next prove that $\rG_u\subseteq \rG_v$ for $v\ge u$. 
    Let $E\in \rG_u$ and fix $(b,s)\in (0,\infty)\times \bR$. We need to show that $E\cap \{\kappa_v\le \tau_b^s\} \in \rG^{b,s}$. 
    We observe that $E\cap  \{\kappa_u \le \tau_b^s \} \in \rG^{b,s}$ by definition of $\rG_u$ and that $\{\kappa_v \le \tau_b^s\}\in \rG^{b,s}$ by \eqref{eq:kappau-bs}. Therefore
    \[
     E\cap \{\kappa_v \le \tau_b^s\} =E\cap  \{ \kappa_u \le \tau_b^s\}\cap  \{\kappa_v \le \tau_b^s\}\in \rG^{b,s}. 
    \]
 It completes the proof of $(\rG_u,\, u\ge 0)$ being a filtration. Let $(b',s')\in (0,\infty)\times \bR$ and $t>0$. Notice that for any  $(b,s)\in (0,\infty)\times \bR$, $\{ t<\tau_{b'}^{s'}<\tau_b^s  \} \in \rF_{\tau_b^s} \subseteq \rG^{b,s}$. Together with \eqref{eq:kappau-bs} and \eqref{eq:kappau-bs-b's'}, we deduce that 
 \[
 \{\kappa_u>\tau_{b'}^{s'}>t\} \cap \{\kappa_u\le \tau_b^s\} = \{ t<\tau_{b'}^{s'}<\tau_b^s  \} \cap \{\kappa_u\le \tau_b^s\}\cap \{\kappa_u\le \min\{\tau_b^s,\tau_{b'}^{s'}\} \,\}^c \in \rG^{b,s}.
 \]
 
 \noindent It follows by definition of $\rG_u$ that $\{\kappa_u>\tau_{b'}^{s'}>t\}\in\rG_u$ for all $t>0$. Therefore $\tau_{b'}^{s'}\ind{\kappa_u>\tau_{b'}^{s'}}$ is measurable w.r.t.\ $\rG_u$. For $u>0$, writing $\kappa_u=\sup\{\tau_{b_p}^{s_p}\ind{\kappa_u>\tau_{b_p}^{s_p}}\colon(b_p,s_p)\in \cD  \}$ where $\cD$ is a countable dense set in $(0,\infty)\times \bR$, we conclude that $\kappa_u$ is $\rG_u$-measurable. 
 
      Now let $\cT$ be the spindle containing $(b',s')$.
      We have $\{s_{\cT}\le \kappa_u\}=\{A_{s_{\cT}}\le u\}$ a.s. and, on the event $\{s_{\cT}\le \kappa_u\le \tau_b^s\}$, $A_{s_{\cT}}=\sum_{\cT'\colon t_{\cT'}<\min(\tau_b^s,\tau_{b'}^{s'})} (t_{\cT'}-s_{\cT'})$, which is $\rG^{b,s}$-measurable by Proposition \ref{c:measurability}.  Hence 
      $\{s_{\cT}\le \kappa_u\}\in \rG_u$ by definition of $\rG_u$. We can also check that for any measurable set $C$, $\{(\cT,s_{\cT},t_{\cT}) \in C\}\cap\{s_{\cT}\le \kappa_u\}\in \rG_u$ by another use of Proposition \ref{c:measurability}.  For $v\le u$,  $\{A_{s_{\cT}}\le v\}\cap\{s_\cT\le \kappa_u\}=\{s_{\cT}\le \kappa_v\}$ a.s. which is in $\rG_v\subseteq\rG_u$. It proves the measurablity of $A_{s_\cT}$ w.r.t. $\rG_u$ for the spindle $\cT$ discovered before $\kappa_u$. Similarly, for any $v> 0$, $\{A_{t_{\cT}}< v\}\cap\{s_\cT\le \kappa_u\} = \cup_{w\in\mathbb{Q}_+, w< \min(u,v)}(\{s_\cT\le \kappa_w\}\cap\{t_\cT-s_\cT< v-w\})$ a.s. which is in $\rG_u$.  The proof is complete.
\end{proof}

We define for each $u\ge 0$,
\begin{align}\label{eq:D_u}
    D_u:=\inf\{s>u: s\in \cA\}.
\end{align}
Observe that $D_u$ is a.s. finite since $\cA$ is unbounded by Lemma \ref{l:scaling A} and that, for a fixed  $u>0$, since $u\in \cA^c$ a.s. by Lemma \ref{lem:kappa-u-fix}, we have $\kappa_u \in (s_{\cT}, t_{\cT})$ for some spindle $\cT$ by Proposition \ref{p:charact A} (ii) and $D_u = A_{t_\cT}>u$. Using that $D_u$ is  $\rG_u$-measurable by Lemma~\ref{l:Gu-filtration}, we deduce that it is a stopping time w.r.t. the filtration $(\rG_v)_{v\ge 0}$.  The following proposition will be proved in the next section.
\begin{proposition}[Regenerative property]\label{cA:range}
Let $\cA$ be defined as in \eqref{eq:cA}. 
Then $\cA$ is a $(\rG_u)_{u\ge 0}$-progressive set.  
Moreover, the set $\mathcal A$ is regenerative, in the sense that, for every $u\ge 0$, the set 
    $$
    \cA\circ \theta_{D_u}:=\cA\cap[D_u,+\infty)-D_u 
    $$
    is independent of  $\rG_{D_u}$  
    and has the same distribution as $\cA$.
\end{proposition}


\subsection{A time-changed  exploration process and the proof of Proposition~\ref{cA:range}}
\label{s:time-change}

Consider a time $T$ of the form either $T=\tau_b^s$
for fixed $(b,s)\in (0,\infty)\times\bR$ or $T=\kappa_u$ for a fixed $u> 0$.  
We now define an exploration process $\widehat{\rho}^T$ of the part to the right of $\cG^T$: 
 it is obtained from $(\varrho_s, s\ge T)$ by neglecting the parts when $\varrho$ is visiting spindles discovered prior to $T$. Recall that the right boundary of partially explored spindles is given by $\cG^T$ by Remark \ref{r:green line} (i).   Therefore we define\footnote{we hide the dependence on $T$ in the notation $\Lambda,\lambda$ for clarity}
$$
\Lambda_u:=\int_T^u \ind{L(v,X_v)\ge  \cG^T_{X_v}} \dd v,\quad
\lambda_t := \inf\{u\ge T: \Lambda_u > t \}, \qquad u\ge T,\ t\ge 0,
$$
where we used the convention $\cG^T_x=0$ for $x<I_T$. We observe that $\Lambda$ is the occupation time process of $\varrho$ in the random area $\{(b,x)\in \bR_+\times \bR: b\ge \cG^T_x\}$. In other words,  $\Lambda$ records the time spent by $\varrho$ in spindles discovered after time $T$. Note that the time $\lambda_t$ is not  an $(\mathscr{F}_t)_{t\ge 0}$-stopping time. 
By Corollary~\ref{c:spindle zero measure}, a.s. $(b,s)$ is in some spindle.  Similarly, by Lemma \ref{lem:kappa-u-fix}, and  Proposition \ref{p:charact A} (ii),  $u\in (A_{s_{\cT}},A_{t_{\cT}})$ for some spindle $\cT$. Hence in both cases $T=\tau_b^s$ and $T=\kappa_u$, $\varrho_T$ is in some spindle $\cT$ and $\Lambda_u=0$ for all $u\ge T$ until the first exit time of $\cT$ after $T$. This exit time is in $\cE^c$ by Lemma \ref{l:exit_explore} and Proposition \ref{p:dense times} implies the explorer enters new spindles at times  arbitrarily close. It implies the identity 
\begin{equation}\label{eq:lambda0}
    \lambda_0
     = \inf\{u\ge T\,:\, \varrho_{u}\notin \cT\}.
    \end{equation}

In the notation \eqref{def:H}, we have $X_{\lambda_0}=H_T$, $L(\lambda_0, x)=\cG^T_{x}$ for all $x\ge H_T$ and $L(\lambda_0, x)\le \cG^T_{x}$ for all $x\in \bR$. Define the processes
\begin{align}
    \widehat{X}_t^T &:= X_{\lambda_t}-H_T,\qquad t\ge 0, \label{eq:Xhat}\\
    \widehat{L}^T(t,x) &:= \big( L(\lambda_t,x+H_T)-\cG^T_{x+H_T} \big) \vee 0 , \qquad t\ge 0, x\in \bR, \label{eq:Lhat}\\
    \widehat{\varrho}^T_t &:= (\widehat{L}^T(t,\widehat{X}^T_t),\widehat{X}^T_t)=(L(\lambda_t,X_{\lambda_t})-\cG^T_{X_{\lambda_t}},X_{\lambda_t}-H_T), \qquad t\ge 0.\label{eq:rhohat}
\end{align}

The following proposition establishes a Markov-type property for the exploration process which will be at the heart of the proof of Proposition~\ref{cA:range}.

\begin{proposition}\label{p:rhohat} 
Let $(b,s)\in (0,\infty)\times \bR$ and $T=\tau_b^s$. Define $\widehat{W}^T,\widehat{\cB}^T, \widehat{\cR}^T$ as in  Proposition~\ref{WN:Green}.  With the notation above, almost surely:   
\begin{enumerate}[(i)]
    \item The process $\widehat{X}^T$ admits $\widehat{L}^T$ as its local times. 
    \item $\widehat{\varrho}^T$ is continuous and $\widehat{L}^T$ is bicontinuous. 
    \item The joint distributions of $(\widehat{W}^T,\widehat{\varrho}^T,\widehat{\cB}^T,\widehat{\cR}^T)$ and $(W,\varrho,\cB,\cR)$ are equal. 
    \item $(\widehat{\varrho}^T,\widehat{\cB}^T,\widehat{\cR}^T)$ is independent of $\rG^{b,s}$.
\end{enumerate}
\end{proposition}
\pagebreak
\begin{proof}
\begin{enumerate}[(i)]
    \item By the (extended) occupation times formula (see \cite[Exercise 1.15, Chapter VI]{RevuzYor} for example) and a change of variables, we have for every $r>0$ and positive Borel function $f$ on $\bR$,
    \begin{align*}
        &\int_0^r f(\widehat{X}^T_t+H_T) \dd t = \int_0^r f(X_{\lambda_t}) \dd t 
        = \int_{\lambda_0}^{\lambda_r} f(X_u) \ind{L(u,X_u)\ge \cG^T_{X_u}} \dd u\\
        = &\int_{\bR} \dd x \int_{\lambda_0}^{\lambda_r} f(x) \ind{L(u,x)\ge \cG^T_x} \dd_u L(u,x)\\
        = &\int_{\bR} f(x) \big( (L(\lambda_r,x)-\cG^T_x) \vee 0 \big) \dd x \\
        = &\int_{\bR} f(x)\widehat{L}^T(r,x-H_T)\dd x = \int_{\bR} f(x+H_T)\widehat{L}^T(r,x)\dd x
    \end{align*}
    where we used in the second equality that $L(u,x)\le \cG^T_x$ for all $u\le \lambda_0$. Therefore $\widehat{X}^T$ admits $\widehat{L}^T$ as its local times. 
    
    \item 
    Since  $t\mapsto \varrho_t$ and $x\mapsto \cG^T_x$ are continuous, in order to prove the continuity of $\widehat{\varrho}^T$, we only need to prove that $\varrho_{\lambda_{t-}}=\varrho_{\lambda_t}$ for every $t\ge 0$, and thus only for $t> 0$ satisfying $\lambda_{t-}< \lambda_{t}$. 
    Let  $u\in (\lambda_{t-},\lambda_t)$ such that $\varrho_{u}$ is in the interior of some spindle.  Since $t\mapsto\lambda_t$ is strictly increasing,  $\lambda_{t-}>\lambda_0$. 
    By definition of $\lambda$ and Remark \ref{r:green line} (i), $\varrho_{u}$ is in a spindle that was discovered before time $T$, and by \eqref{eq:lambda0} and $u>\lambda_0$, it is in the right part of that spindle. Let $u_0,u_1>T$ be the pair $t_0$ and $t_1$ associated to $u$ as given in Proposition \ref{p:right return}. Since $u_0,u_1\in \mathcal{E}^c$, by Proposition~\ref{p:dense times}, there exist spindles discovered at times   arbitrarily close to $u_0$ and $u_1$ respectively. This yields  $\lambda_{t-}=u_0$ and $\lambda_t=u_1$. Proposition \ref{p:right return} (i) implies that  $\varrho_{\lambda_{t-}}=\varrho_{\lambda_t}$.

    We now prove the bicontinuity of $\widehat{L}^T$. Since $(t,x)\mapsto L(t,x)$ and $x\mapsto \cG^T_x$ are continuous, it suffices to prove that $L(\lambda_{t-},x)=L(\lambda_t,x)$ for all $t\ge 0$ and $x\in\bR$ satisfying $L(\lambda_t,x)> \cG^T_{x}$. With the notation of the last paragraph, $\lambda_{t-}=u_0$,  $\lambda_t=u_1$, and on the time interval $(u_0,u_1)$, $\varrho$ only visits points inside the spindle. So $L(u_1,X_s)\le \cG^T_{X_s}$ for all $s\in (u_0,u_1)$. Hence for any $x\in \bR$ such that $L(u_1,x)>\cG^T_x$, $X$ cannot visit $x$ on $(u_0,u_1)$ and thus $L(u_0,x)=L(u_1,x)$ as desired.

    \item From (i) and the definitions of $\varrho$, $\widehat{\varrho}^T$, it suffices to prove $(\widehat{W}^T,\widehat{X}^T,\widehat{\cB}^T,\widehat{\cR}^T)\overset{(d)}{=}(W,X,\cB,\cR)$. By the definitions of the flows $\cR$ and $\cB$ and Proposition \ref{WN:Green}, $(\widehat{W}^T,\widehat{\cB}^T,\widehat{\cR}^T)$ has the distribution of $(W,\cB,\cR)$. We show now that $\widehat{\cR}^T$ is the forward local time flow of $\widehat{X}^T$ in the terminology of \cite{aidekon2023stochastic}, i.e. for $a\ge 0$ and $y\ge x$,
\begin{equation}\label{eq:loctimeflow Xhat}
    \widehat{\cR}^{T}_{x,y}(a)=\widehat{L}^T (\widehat{\tau}^{T,x}_a,y) 
\end{equation}
        
    \noindent where $\widehat{\tau}^{T,x}_a:=\inf\{t>0: \widehat{L}^T (t,x)>a\}$  is the right-continuous inverse of $\widehat{L}^T(\cdot,x)$. It would complete the proof since by \cite[Proposition 2.5 \& Remark 2.8]{aidekon2023stochastic},  $\widehat{X}^T$ is a measurable function of its forward local time flow. Note in particular that $(\widehat{\varrho}^T,\widehat{\cB}^T,\widehat{\cR}^T)$ is therefore measurable w.r.t. $\widehat{W}^T$. It remains to prove   \eqref{eq:loctimeflow Xhat}.

    Recall that $\widehat{\cR}^{T}_{x,y}(a) = \cR_{x',y'}(a+\cG^T_{x'})-\cG^T_{y'} = L(\tau^{x'}_{a+\cG^T_{x'}},y')-\cG^T_{y'}$ where $(x',y')=(x+H_T,y+H_T)$. We shall prove that $\tau^{x'}_{a+\cG^T_{x'}} = \lambda_{\widehat{\tau}^{T,x}_a}$, so that 
    \[
       \widehat{L}^T (\widehat{\tau}^{T,x}_a,y)= L(\lambda_{\widehat{\tau}^{T,x}_a},y')-\cG^T_{y'} 
       = L(\tau^{x'}_{a+\cG^T_{x'}},y')-\cG^T_{y'} =\widehat{\cR}^{T}_{x,y}(a), 
    \]
    giving \eqref{eq:loctimeflow Xhat}.
    For any $t>\widehat{\tau}^{T,x}_a$, we have $L(\lambda_t,x')-\cG^T_{x'}>a$ and thus $\lambda_t>\tau^{x'}_{a+\cG^T_{x'}}$. Taking $t\downarrow \widehat{\tau}^{T,x}_a$ yields $\lambda_{\widehat{\tau}^{T,x}_a}\ge \tau^{x'}_{a+\cG^T_{x'}}$. Finally, the other direction follows since, by definition of $\widehat{L}^T$ and its bicontinuity,  
    \[    a=\widehat{L}^T(\widehat{\tau}^{T,x}_a,x)\ge L(\lambda_{\widehat{\tau}^{T,x}_a},x')-\cG^T_{x'}.
    \]
 
\item We noticed in (iii) that $(\widehat{\varrho}^T,\widehat{\cB}^T,\widehat{\cR}^T)$ is measurable w.r.t.\ $\widehat{W}^T$. The statement follows from the independence of $\widehat{W}^T$ and $\rG^{b,s}$ by Proposition \ref{WN:Green}.\qedhere
\end{enumerate}    
\end{proof}

We now deal with the case $T=\kappa_u$. As in Proposition \ref{WN:Green}, we let  $\widehat{\cB}^T = \theta_{-H_T}(\cB-\cG^T)$ and $\widehat{\cR}^T = \theta_{-H_T}(\cR-\cG^T)$ in the notation \eqref{eq:S translation} and \eqref{eq:S-g}.
\begin{proposition}\label{p:rhohat kappa}
Let $u>0$ and $T=\kappa_u$. Then 
\begin{enumerate}[(i)]
    \item $(\widehat{\varrho}^T,\widehat{\cB}^T,\widehat{\cR}^T,\widehat{L}^T)$ is independent of $\rG_u$ and has the same law as $(\varrho,\cB,\cR,L)$. 

    \item Almost surely, no spindle discovered after time $T$ has its bottom point on $(\cG^T_x,x)_{x\ge I_T}$.
\end{enumerate} 
\end{proposition}

\begin{proof}
We prove (i) by approximation. Almost surely, by Lemma \ref{lem:kappa-u-fix}, $u\in \cA^c$, so $u\in (A_{s_{\cT}},A_{t_{\cT}})$ for some spindle $\cT$ by Proposition \ref{p:charact A} (ii). Let $(b_p,s_p)_{p\ge 1}$ be a countable deterministic set dense in $(0,\infty)\times \mathbb{R}$ and $t_p:=A_{T_p}$, $T_p=\tau_{b_p}^{s_p}$. 
For each $k\ge 1$, let $p(k)$ be the index in $\{1,2,\ldots,k\}$ which realizes the minimum of $T_p$ over $p\in \{1,2,\ldots,k\}$ such that $\kappa_u \le T_p <t_{\cT}$, and $p(k)=\infty$ if no such $p$ exists. Hence   $ T_{p(k)} = \inf\{T_p \colon 1\le p\le k,\, \kappa_u\le T_p <t_{\cT}\}$, with the convention that $\inf \emptyset =\infty$ and $T_{\infty} = \infty$. Since $\{T_p:p\ge 1\}$ is dense in $\bR_+$, there must exist $K$ large enough with $p(K)<+\infty$.

For every $(b,s),(b',s')\in\bR_+\times\bR$, $\{\tau_b^s\le\tau_{b'}^{s'}\}\in \rF_{\tau_b^s}\cap\rF_{\tau_{b'}^{s'}} \subseteq \rG^{b,s}\cap\rG^{b',s'}$ by Proposition \ref{p:markov S} (ii). Moreover, if $\cT'$ is the spindle containing $(b,s)$, the event $\{s_{\cT'}<\kappa_u<t_{\cT'}\}\in \rG_u$ by Lemma \ref{l:Gu-filtration}, hence   $\{s_{\cT'}<\kappa_u<t_{\cT'}\}\cap\{\kappa_u\le \tau_b^s\}\in \rG^{b,s}$. By Proposition \ref{c:measurability}, $\{ \tau_b^s < t_{\cT'}\}\in \rG^{b,s}$. We deduce that $\{\kappa_u\le \tau_b^s<t_{\cT}\}\in \rG^{b,s}$ where we recall that $\cT$ is the spindle associated with $\kappa_u$. By \eqref{eq:kappau-bs} and \eqref{eq:kappau-bs-b's'}, 
\[
\{\tau_b^s<\kappa_u\le \tau_{b'}^{s'}\} = \big(\{\kappa_u\le \tau_{b'}^{s'}\}\backslash \{\kappa_u\le \min\{\tau_{b'}^{s'},\tau_b^s\}\, \}\big) \in \rG^{b',s'}.
\]
It implies that for any $p=1,\ldots,k$,
\[
\{p(k)=p\} = \{\kappa_u\le T_p<t_\cT\} \cap \bigg(\bigcap_{q\in \{1,\ldots,k\}\backslash\{p\}} \big(\{T_q<\kappa_u\le T_p\} \cup \{T_q>T_p\} \big)\bigg) \in \rG^{b_p,s_p}. \vspace{-0.1cm}
\]
 Let $E\in \rG_u$.  By definition of $\rG_u$ in \eqref{eq:G_u}, $E\cap \{\kappa_u\le T_p \} \in \rG^{b_p,s_p}$ .    Then $E_p:=E\cap  \{p=p(k)\}\in \rG^{b_p,s_p}$.  On $E_p$, $\cG^{T}=\cG^{T_p}$, and thus $(\widehat{\varrho}^T,\widehat{\cB}^T,\widehat{\cR}^T,\widehat{L}^T)=(\widehat{\varrho}^{T_p},\widehat{\cB}^{T_p},\widehat{\cR}^{T_p},\widehat{L}^{T_p})$
which is distributed as $(\varrho,\cB,\cR,L)$ by Proposition \ref{p:rhohat} (i)-(iii) and is independent of $\rG^{b_p,s_p}$ by Proposition \ref{p:rhohat} (iv).  We deduce that for every bounded measurable functional $h$,
\begin{align*}
    &\bE\left[ h(\widehat{\varrho}^T,\widehat{\cB}^T,\widehat{\cR}^T,\widehat{L}^T)\mathbbm{1}_{E\, \cap \{p(k)<+\infty\}}\right]\\
    =\,&\sum_{p=1}^k  \bE\left[ h(\widehat{\varrho}^{T_p},\widehat{\cB}^{T_p},\widehat{\cR}^{T_p},\widehat{L}^{T_p})\mathbbm{1}_{E_p}\right] =   
    \bE\left[ h(\varrho,\cB,\cR,L)\right] \,\bP(E\cap \{p(k)<+\infty\}).
\end{align*}
We complete the proof of (i) by letting $k\to\infty$. Since $\cG^T=\cG^{T_p}$ on the event $\{p(k)=p\}$, (ii) holds because of Lemma~\ref{l:spindles after T}. 
\end{proof}

\begin{proof}[Proof of Proposition~\ref{cA:range}]
  The right-continuous process $(\kappa_u)_{u\ge 0}$ is $(\rG_u)_{u\ge 0}$-adapted, and thus $(\rG_u)_{u\ge 0}$-progressive. It follows that $\cA = \overline{\{u\ge 0\colon \kappa_u\ne \kappa_{u-}\}}$ is $(\rG_u)_{u\ge 0}$-progressive. Fix $u>0$ and let $T=\kappa_u$. We have already observed that $\kappa_u\in (s_{\cT},t_{\cT})$ for some spindle $\cT$ a.s.\ and $D_u=A_{t_{\cT}}$.
In view of \eqref{eq:Set-A} and Lemma \ref{l:Gu-filtration}, we also obtain that $\cA\cap [0,D_u]$ is measurable w.r.t.\ $\rG_u$.

We turn to the proof of the regenerative property. Denote by $\mathscr{T}(\cR,\cB)$, resp. $\mathscr{T}(\widehat{\cR}^T,\widehat{\cB}^T)$ the spindles constructed from the flows $\cR$ and $\cB$, resp. $\widehat{\cR}^T$ and $\widehat{\cB}^T$. By construction and Proposition~\ref{p:rhohat kappa} (ii), the spindles $\mathscr{T}(\cR,\cB)$ discovered by $\varrho$ after time $\lambda_0$ are (translations by $\cG^T$ of) the spindles $\mathscr{T}(\widehat{\cR}^T,\widehat{\cB}^T)$ discovered by $\widehat{\varrho}^T$ (in the same chronological order), and the time spent in the left parts of $\mathscr{T}(\cR,\cB)$ by $\varrho$ after time $\lambda_0$ is the time spent in the left parts of $\mathscr{T}(\widehat{\cR}^T,\widehat{\cB}^T)$ by $\widehat{\varrho}^T$. Therefore, by \eqref{def:A-kappa} and \eqref{eq:Set-A}, the set  $\cA\circ \theta_{D_u}$ is the set $\cA$ constructed from  $(\widehat{\varrho}^T,\mathscr{T}(\widehat{\cR}^T,\widehat{\cB}^T))$ instead of $(\varrho, \mathscr{T}(\cR,\cB))$. By Proposition~\ref{p:rhohat kappa} (i), $\cA\circ \theta_{D_u}$ is independent of $\rG_u$ and has the same law as $\cA$. 

We now prove that $\cA\circ \theta_{D_u}$ is independent of $\rG_{D_u}$. Since $D_u$  is a $(\rG_u)_{u\ge 0}$-stopping time,  there exists a sequence of $(\rG_u)_{u\ge 0}$-stopping times $(T_n,\, n\ge 1)$ taking countable values decreasing to $D_u$. It follows by the last paragraph that $\cA\circ\theta_{D_{T_n}}$ is independent of $\rG_{T_n}$, hence of $\rG_{D_u}$. Since $\cA$ does not have isolated points by Proposition~\ref{p:charact A} (iii) and $[u,D_u)\subset \cA^c$, there exist $v\in \cA$, $v>D_u$ arbitrarily close to $D_u$. For any such $v$, $T_n\in [D_u,v)$ for $n$ large enough, for which $D_{T_n}\le v$. It proves that $D_{T_n}\downarrow D_u$. Passing to the limit, we deduce that  $\cA\circ \theta_{D_{u}}$ is also independent of $\rG_{D_u}$.
\end{proof}

\subsection{Representation of spindles as a marked L\'evy process}
\label{sec:markedLevy}

Since the set $\cA$ is unbounded (consequence of Lemma~\ref{l:scaling A}), regenerative and progressively measurable w.r.t. $(\rG_u)_{u\ge 0}$ (Proposition~\ref{cA:range}) and has no isolated points (Proposition~\ref{p:charact A} (iii)), then by \cite[Theorem~2.1]{Bertoin1999} there exists a strictly increasing subordinator  $\eta$ (without killing) whose range equals $\cA$ a.s.\ and the process 
\begin{equation}\label{def:L(u)}
   \cL (u):= \inf \{\ell \ge 0\colon \eta_\ell >u\}, \qquad u\ge 0,
\end{equation}

\noindent is $(\rG_{u})$-adapted and a.s.\ continuous. 
Such a subordinator (in fact, its inverse) is unique up to a linear scaling of time, which is equivalent to multiplication of a constant to $\cL$; see \cite[Section 2.1]{Bertoin1999}.

Note that
 \begin{equation}\label{eq:bijection subor spindles}
     \cA^c = \bigcup_{\ell:\,  \Delta\eta_\ell>0} (\eta_{\ell-},\eta_\ell) = \bigcup_{\cT \text{spindle}} (A_{s_\cT},A_{t_\cT}),
 \end{equation}

\noindent where the second equality follows from Proposition~\ref{p:charact A} (ii). This identity provides  a bijection between the spindles $\cT$ and the jumps of the subordinator $\eta$.  Recall  the definition of the gasket $\rK$ in \eqref{eq:gasket}.

\begin{proposition}\label{p:ki}
Define a mapping $\ki \colon [0,\infty) \to \bR_+\times \bR$ by 
\begin{equation}\label{def:ki}
\ki:=\varrho\circ \kappa\circ \eta.
\end{equation}

(i) The range of $\ki$ is the gasket $\rK$.

(ii) The jump times of the process $\ki$  coincide with those of $\eta$.

(iii) For a jump time $\ell$ of $\eta$, if $\cT$ is the spindle associated through \eqref{eq:bijection subor spindles}, we have
\begin{equation}\label{eq:xi-X}
        (\eta_{\ell-},\eta_{\ell})=(A_{s_{\cT}},A_{t_{\cT}}),\;(\ki_{\ell-},\ki_{\ell})=(\varrho_{s_{\cT}},\varrho_{t_{\cT}}). 
    \end{equation}
\end{proposition}

We will see in Section \ref{s:shredded} that the range of 
$\ki$ is exactly $\{\ki(\ell),\,\ell\in \bR_+\}$, i.e. taking the
closure is superfluous. In words, $\ki$ is a process lying in the gasket, jumping at the jump times of $\eta$, from the bottom point to the top point of the associated spindle.

\begin{proof}[Proof of Proposition~\ref{p:ki}]
    By Proposition \ref{p:explore-complement} and Proposition \ref{p:charact A} (iv), we have $\rK=\{\varrho_{\kappa_u}\colon u\in \cA\}$.  We then apply that the range of $\eta$ is $\cA$ and that $\varrho\circ \kappa$ is continuous by Lemma \ref{l:kappa_A}  to prove (i). Moreover the process $\ki$  can only jump  at the jump times of $\eta$ by continuity of $\varrho\circ \kappa$. It remains to prove (iii) which will also imply that jump times of $\eta$ are jump times of $\ki$. Let $\ell$ be a jump time of $\eta$. We have $\eta_{\ell-}=A_{s_{\cT}}$ and $\eta_\ell=A_{t_{\cT}}$ by definition of $\cT$. By continuity of $\varrho\circ \kappa$, the left and right limits of $\ki$ at $\ell$ are respectively $\varrho \circ \kappa(A_{s_{\cT}})$ and $\varrho \circ \kappa(A_{t_{\cT}})$. We then apply Lemma \ref{l:kappa_A} (ii).
\end{proof}

The goal of this section is to show that the process  $\xi := X\circ \kappa\circ \eta$, i.e.\ the vertical coordinate of $\ki$, is a L\'evy process (Theorem \ref{thm:Levy}).

\pagebreak

Let us define a point process by marking up each jump of $\eta$ with an excursion obtained from the corresponding spindle. 
To this end, we introduce the space $\mathscr{E}$ of non-negative continuous 
excursions away from zero, i.e.\
\[
\mathscr{E} = \left\{ f\colon \bR_+\to \bR_+ \text{ is continuous, and} \;  \exists \, z\ge 0 \text{ s.t. } f(y)= 0,  \forall y\ge z\right\}.
\] 
For $f\in \mathscr{E}$, set $\zeta (f):= \sup \{ t\ge 0\colon f(t)>0\}$. We add a cemetery point $\partial$.
Define a point process $\fN= (\mathbf{e}_\ell,\, \ell \ge 0)$ on $\bR_+\times \mathscr{E}$ as follows:
\begin{enumerate}
    \item if $\eta_{\ell-}=\eta_{\ell}$, then $\mathbf{e}_\ell := \partial$;
    \item if $\eta_{\ell-}<\eta_{\ell}$, we let $\mathbf{e}_\ell:=\Big(\max(\cB_{r,r+x}(a)-\cR_{r,r+x}(a-),0)\Big)_ {x\ge0}$ where $(a,r)$ is the bottom point of the spindle $\cT$ of \eqref{eq:xi-X}. Then $\zeta(\mathbf{e}_\ell)=\Delta \xi_\ell$ is the height of $\cT$, i.e. $z-r$ where $(c,z)$ is the top point of $\cT$.
\end{enumerate}
See Figure~\ref{fig:markedLevy} for an illustration.

\begin{figure}[htbp]
\centering
       \scalebox{0.9}{ 
        \def\svgwidth{\columnwidth}
        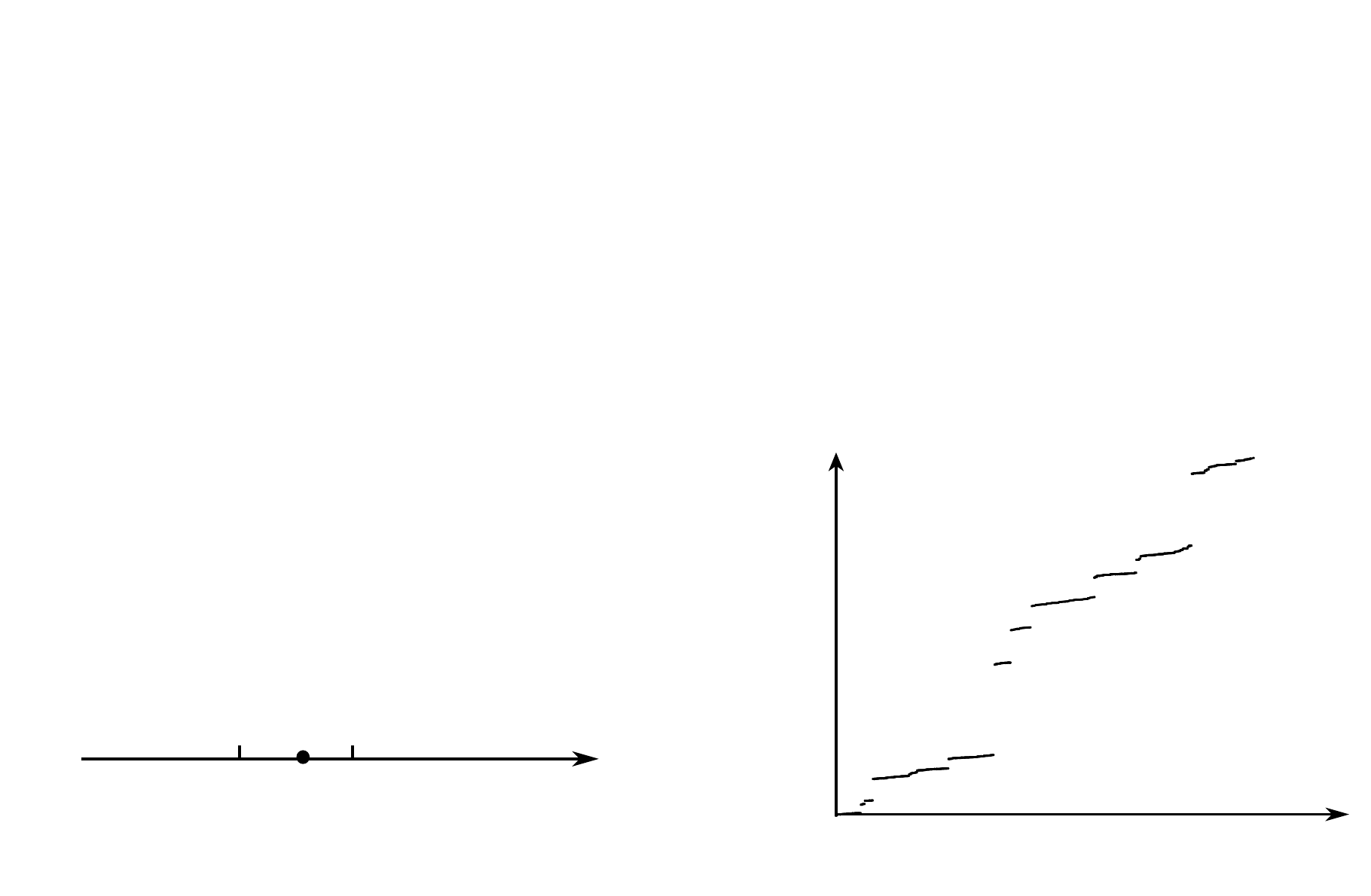
    }
\caption{In this figure we show how to construct the marked L\'evy process from the collection of spindles. On the top left,  $\varrho$ discovers the spindle $\cT$  at time $s_{\cT}$ and explores its left part during the time interval $(s_{\cT}, t_{\cT})$.  
The bottom left is the time axis of the exploration process $\varrho$. The transformation $r\mapsto A_r$ concatenates the left-exploration intervals, so that the closure of the endpoints of the intervals form the set $\mathcal{A} = \overline{\{A_{s_{\cT}}\colon \cT \textrm{ spindle} \}} = \overline{\{A_{t_{\cT}}\colon \cT \textrm{ spindle} \}}$. 
On the bottom right the subordinator $\eta$ has its range given by $\cA$, such that each spindle $\cT$ corresponds to a jump time $\ell>0$ of $\eta$ with $(A_{s_{\cT}}, A_{t_{\cT}}) = (\eta_{\ell-}, \eta_{\ell})$. 
On the top right we show the L\'evy process $\xi$ (coloured in purple), which has the same jump times as $\eta$. The jump time $\ell$ corresponds to the spindle $\cT$ with $(X_{s_{\cT}}, X_{t_{\cT}}) = (\xi_{\ell -}, \xi_{\ell})$. The jumps of $\xi$ are further marked by excursions $(\fe_{\ell}:\ell\ge 0)$ (some of them filled with gray), given by the width of the spindle $\cT$, and the lifetime of each excursion $\fe_\ell$ equals the size of the corresponding jump $\Delta \xi_{\ell}$.
}
\label{fig:markedLevy}
\end{figure}

Recall that $\delta\in (0,2)$. Pitman and Yor \cite[Section~3]{PitmYor82} construct an excursion measure $\nu^{(-\delta)}_{\mathtt{BESQ}}$ associated with $\besq^{- \delta}$, which is a $\sigma$-finite measure on 
the space of excursions $\mathscr{E}$. 
In this paper we set $\nu^{(-\delta)}_{\mathtt{BESQ}}$ re-scaled such that 
\begin{equation}\label{eq:nu_BESQ}
    \nu^{(-\delta)}_{\mathtt{BESQ}} (f\colon \zeta(f) \ge y) = \frac{\delta }{2^{\frac\delta 2+1} \Gamma(1-\frac\delta 2)\Gamma(1+\frac\delta 2)}  y^{-1-\frac\delta 2}, \qquad y>0.
\end{equation}
\noindent Note that a $\rm BESQ^{-\delta}$ process may be seen as a $\rm BESQ^{4+\delta}$ process conditioned to hit 0. Therefore for $y>0$, under $\nu^{(-\delta)}_{\mathtt{BESQ}}(\cdot \mid \zeta(f)=y)$, 
the process $f$ is a $\rm BESQ^{4+\delta}$ bridge  from $0$ to $0$ over $[0,y]$ (\cite[Page~454]{PitmYor82}).

\begin{theorem}\label{thm:PPP}
The point process $\mathbf{N}= (\mathbf{e}_\ell,\, \ell \ge 0)$ is a Poisson point process, whose intensity $\nu$ is a multiple of the ${\rm BESQ}^{-\delta}$ excursion measure $\nu_{\mathtt{BESQ}}^{(-\delta)}$.
\end{theorem}

Recall that a subordinator is determined by its range up to a linear scale of time. For $b >0$, the subordinator $\eta' = (\eta_{b \ell}, \ell \ge 0)$ has the same range as $\eta$  and corresponds to the Poisson point process $\fN' = (\fN_{b \ell}, \ell\ge 0)$  with intensity   $\nu' = b \nu$. In the sequel, we will always choose $b>0$ appropriately such that the Poisson point process $\fN$ exactly 
 has intensity  $\nu_{\mathtt{BESQ}}^{(-\delta)}$ given by \eqref{eq:nu_BESQ}; such properly rescaled $\fN$ is uniquely (in law) determined by the PRBM $X$.

\begin{proof}[Proof of Theorem~\ref{thm:PPP}]
We first show that  $\mathbf{N}$ is a Poisson point process. The jumps of the subordinator $\eta$ forming a Poisson point process,  it suffices to show that conditionally on $\eta$, the collection $\{\mathbf{e}_\ell\colon \Delta\eta_\ell>r\}$ 
are independent random variables whose conditional distributions only depend on the value of $\Delta \eta_\ell$. Let us prove this statement. 
Fix $r>0$, and denote by $\cT^{r,i}$ the $i$-th spindle $\cT$ discovered by $\varrho$ such that $t_{\cT}-s_{\cT}>r$. By \eqref{eq:xi-X} and the equality $A_{t_{\cT}}-A_{s_{\cT}}=t_{\cT}-s_{\cT}$ by definition of $A$ in \eqref{def:A-kappa}, this indeed corresponds to all jumps of $\eta$ with $\Delta\eta_\ell>r$. Let $(\mathbf{e}^{r,i})_{i\ge 1}$ denote the associated excursions in $\fN$. We only have to prove that $(\mathbf{e}^{r,i},t_{\cT^{r,i}}-s_{\cT^{r,i}})_{i\ge 1}$ are i.i.d.

By Proposition~\ref{p:rhohat kappa}, for each fixed $u> 0$, the collection of excursions $\mathbf{e}_\ell$ associated with spindles discovered after time $\kappa_u$ is independent of $\rG_u$ with always the same distribution. In particular, the excursion associated with the first spindle $\cT$ discovered after time $\kappa_u$ such that  $t_{\cT}-s_{\cT}>r$  is independent of $\rG_u$ and has the same law as $\fe^{r,1}$. It implies  that $(\mathbf{e}^{r,i},t_{\cT^{r,i}}-s_{\cT^{r,i}})_{i\ge 1}$ are i.i.d. Indeed, fix $i\ge 1$ and take $E\in \rG_{u^{r,i}}$ where  $u^{r,i}:=A_{s_{\cT^{r,i}}}$ is a $\rG$-stopping time by Lemma \ref{l:Gu-filtration}. Then for every positive measurable function $g$,  \vspace{-0.2cm}
\begin{align*}
\bE[\mathbbm{1}_E\,  g(\mathbf{e}^{r,i+1})] &= 
      \lim_{n\to\infty} \sum_{k=0}^\infty\bE[\mathbbm{1}_E \ind{k2^{-n}\le u^{r,i}<(k+1)2^{-n}<u^{r,i+1}}g(\mathbf{e}^{r,i+1})]\\
      &=
      \lim_{n\to\infty} \sum_{k=0}^\infty\bE[\mathbbm{1}_E \ind{k2^{-n}\le u^{r,i}<(k+1)2^{-n}<u^{r,i+1}}] \bE[g(\mathbf{e}^{r,1})] \\
      &=
      \bP(E) \bE[g(\mathbf{e}^{r,1})],
\end{align*}
where we applied the previous observation to $u=(k+1)2^{-n}$.

We compute now the intensity $\nu$ of the Poisson point process $\mathbf{N}$. 
For $d>0$, let $T_d({\mathbf e}) = \inf\{s\ge 0 \colon \mathbf{e}(s) =d \}$ be the hitting time of $d$ by $\mathbf{e}$. We first prove that $\nu(T_d({\mathbf e})<\infty)\in(0,+\infty)$. Denote by $t_n:=\inf\{t>0:X_t=-n^{-1}\}$ for all $n\in\bN$. Then $L(t_n,x)=\cR_{-n^{-1},x}(0)$ for all $x\ge -n^{-1}$, and by Proposition~\ref{prop:explore} (iii), all spindles discovered by $\varrho$ before time $t_n$ have been fully explored. Since $L(t_n,x)\to 0$ as $n\to \infty$ uniformly in $x\in\bR$, it follows that $\sup_{x\in\bR}L(t_n,x)<d$ for all $t_n$ small enough, which implies that $\cB_{r,x}(a)-\cR_{r,x}(a-)< d$ for all spindles $\cT$ with $s_\cT<t_n$, and hence $\nu(T_d({\mathbf e})<\infty)$ is finite. We now observe that for all $(b,s)\in(0,+\infty)\times\bR$ and $T=\tau_b^s$, by the construction of $\cG^T$ as in Definition~\ref{green}, 
\[\bP\big(\exists\,\cT=\cT_{(a,r)},\, x>r \text{ with } s_\cT<T,\, \cB_{r,x}(a)-\cR_{r,x}(a-)\ge d\big) \ge \bP\big(\cG^T_x-L(T,x)\ge d \text{ for some } x>I_T\big).\] 
The probability on the right-hand-side is positive by Proposition~\ref{p:green}. It follows that $\nu(T_d({\mathbf e})<\infty)>0$.

Then by \cite[(3.1)]{PitmYor82}, it remains to prove that, for any $d>0$, the law of $(\mathbf{e} (t+T_d({\mathbf e})),t\ge 0)$ under $\nu(\cdot \mid T_d({\mathbf e})<\infty)$ is that of a ${\rm BESQ}^{-\delta}_d$ process. Let us prove it. 
Let $\overline{\cT}=\cT_{(\bar{a},\bar{r})}$ be the first spindle discovered by $\varrho$, such that $\cB_{\bar{r},x}(\bar{a})-\cR_{\bar{r},x}(\bar{a}-)=d$ for some $x>r$. Let $\bar{x}$ be the minimal such $x$. We have to show that the process
\begin{equation}\label{eq:proof_poisson}
    \max(\cB_{\bar{r},\bar{x}+t}(\bar{a})-\cR_{\bar{r},\bar{x}+t}(\bar{a}-),0),\quad t\ge 0
\end{equation}
\noindent is distributed as  a ${\rm BESQ}^{-\delta}_d$ process.

Recall by Proposition~\ref{p:green} that, for a fixed point $(b,s)\in (0,\infty)\times \bR$ with $T= \tau^{s}_b$, $(\cG^{T}_y - L(T,y), y\ge s)$ is a ${\rm BESQ}^{-\delta}$ process. 
Moreover, if $T= \tau^{s}_b$ happens just after the entrance time  of the spindle $\overline{\cT}$ by $\varrho$, i.e.\  $T>s_{\overline{\cT}}$ is close enough to $s_{\overline{\cT}}$, then we have by Proposition~\ref{prop:explore} that $\cG^{T}_y - L(T,y)= \cB_{\bar{r},y}(\bar{a})-\cR_{\bar{r},y}(\bar{a}-)$ for $y\ge \bar{x}$. 
We use an approximation argument to formulate this idea.

Let $(b_p,s_p)_{p\ge 1}$ be a deterministic countable dense set of points in $(0,\infty)\times \bR$, $T_p:=\tau_{b_p}^{s_p}$ and $(a_p,r_p)$ be the bottom point of the spindle $\cT_p$ which contains $(b_p,s_p)$.  Define
\begin{align}
    \widetilde{x}_p &:= \inf\{y\ge r_p: \cB_{r_p,y}(a_p) - \cR_{r_p,y}(a_p-) =d\}, 
\end{align}
\noindent with the usual convention $\inf \emptyset:=\infty$.
Intuitively, for each $k\ge 1$, we consider those $p\in \{1,\ldots, k\}$, such that $(b_{p},s_{p})$ is in the left part of the spindle $\cT_{p}$, whose maximum width is greater than $d$ with $s_{p} <\widetilde{x}_{p}$. Then we define $p(k)\in \{1,\ldots, k\}$ such that  $T_{p(k)}$ is the minimum time among such points. 
More precisely, we define an event  $E_{k,p}$, such that $p(k) = p$ on the event $E_{k,p}$, in the following way. 

Recall the notation $\cG^{T_p}$ and $H_{T_p}$ in Definition \ref{green} and equation \eqref{def:H}. We let
\begin{align}
     x_p&:=\inf\{y\in [r_p,H_{T_p}): \cB_{r_p,y}(a_p) - \cR_{r_p,y}(a_p-) =d\}, \nonumber \\
    y_p&:= \max(s_p,x_p), \label{def:y_p}
\end{align}
Since $\cG^{T_p}$ traces the right boundary of partially explored spindles at time $T_p$, $x_p$ can be rewritten as  
\begin{equation}\label{def:xp}
    x_p=\inf\{y\in [r_p,H_{T_p}): \cG^{T_p}_y - \cR_{r_p,y}(a_p-) =d\}.
\end{equation}

\noindent For $k\ge 1$ and $1\le p\le k$, let $E_{k,p}$ be the event 
$E_{k,p,1}\cap E_{k,p,2}\cap E_{k,p,3}$ where
\begin{align*}
    E_{k,p,1} &:=  \{ s_p<x_p<+\infty,\, L(T_p,x_p)=\cR_{r_p,x_p}(a_p-)\}, \\
    E_{k,p,2} & := \{\widetilde{x}_q=+\infty, \, \text{ for all } 1\le q\le k \text{ s.t } s_{\cT_q}<s_{\cT_p} \}, \\
    E_{k,p,3} &:= \{T_q\ge T_p, \, \text{ for all } 1\le q\le k \text{ s.t } \cT_q=\cT_p \text{ and } s_q\le y_p\}.
\end{align*}

\noindent In words, the event $E_{k,p,1}$ corresponds to the following: the spindle $\cT_p$ reaches width $d$ at a level above  $s_p$,  the point $(b_p,s_p)$ belongs to the left part of the spindle in the terminology of Section~\ref{s:explore}, and the explorer $\varrho$ does not visit level $\widetilde{x}_p=x_p=y_p$ between the entrance time  of $\cT_p$ and  $T_p$. The event $E_{k,p,2}$ says that, every spindle, which is discovered before entering $\cT_p$ and contains a point $(b_q,s_q)$, has maximal width strictly smaller than $d$. Finally, the event $E_{k,p,3}$ states that $T_p$ is the minimal value among those $T_q$ corresponding to the same spindle $\cT_q=\cT_p$ and satisfying  $s_q\le y_p$. Intersected with $E_{k,p,1}$ (which yields $T_p<T_q$ for any index $q$ with $\cT_q=\cT_p$ and $s_q \ge x_p=y_p$),  this implies that $T_p$ is the minimum over all $T_q$ with $\cT_q=\cT_p$. See Figure~\ref{fig:proof_PPP} for an illustration. 

 \begin{figure}[htbp]
\centering
       \scalebox{0.7}{ 
        \def\svgwidth{0.6\columnwidth}
        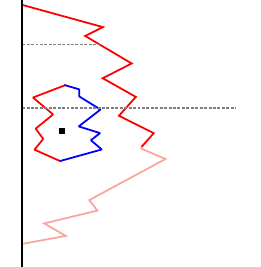
    }
\caption{This figure displays a realization of the event $E_{k,p}$ in the case $k=5,\, p=4$, where $(b_q,s_q),\, 1\le q\le 5$ are marked in black. 
Since the width of $\cT_4$ reaches $d$ at level $x_4\in(s_4, H_{T_4})$, and $L(T_4,x_4)=\cR_{s_4,x_4}(b_4)=\cR_{r_4,x_4}(a_4-)$, it belongs to $E_{5,4,1}$. 
The entrance times of $\cT_{1}$ and $\cT_3$ are both strictly smaller than $s_{\cT_4}$, while the widths of $\cT_1$ and $\cT_3$ are strictly smaller than $d$; thus it is in $E_{5,4,2}$. 
In addition, $(b_2,s_2)\in\cT_2=\cT_4$ and $s_2\le y_4=x_4$. Moreover, $b_2\ge L(T_4,s_2)$, which yields that $T_2\ge T_4$. This implies that it lies in $E_{5,4,3}$. Since the entrance time of $\cT_5$ is strictly greater than $s_{\cT_4}$, the shape of $\cT_5$ is irrelevant for the event. }
\label{fig:proof_PPP}
\end{figure}

Observe that $E_{k,p}, p=1,\ldots ,k$ are disjoint.  
Let $E_k:=\bigcup_{p=1}^k E_{k,p}$. 
Define $p(k):= p$ on $E_{k,p}$ and $p(k):=\infty$ on $E_k^c$. 
Recall that  $\overline{\cT} = \cT_{(\bar{a},\bar{r})}$ is the first spindle discovered by $\varrho$ with width greater than $d$. 
Almost surely, for $k$ large enough, 
\begin{equation}\label{eq:p(k)}
    p(k)<\infty,\, x_{p(k)}=\bar{x},\, (a_{p(k)},r_{p(k)})=(\bar{a},\bar{r}).
\end{equation}

\noindent  This happens when the left part of the spindle $\overline{\cT}$ contains a point $(b_p,s_p)$ with $s_p<\bar{x}$ and $L(T_p,\bar{x})=\cR_{\bar{r},\bar{x}}(\bar{a}-)$ for $p\le k$; in which case $p(k)$ is chosen to be the index $p$ which minimizes $T_p$ over all such points.

For $y\in \mathbb{R}$, let $\rG^{b_p,s_p}_y$ be the $\sigma$-field defined in \eqref{eq:G^bs_y}. Note that 
\begin{enumerate}[(i)]
    \item $L(T_p,\cdot)$ is $\rF_{T_p}$-measurable,
    \item $\cG^{T_p}$ is $(\rG_y^{b_p,s_p})_{y \in \bR}$-adapted, by (i) and Proposition \ref{p:green},
    \item $H_{T_p}$ is a $(\rG_y^{b_p,s_p})_{y \in \bR}$-stopping time by Proposition \ref{p:green},
    \item $(a_p,r_p)$ and $(\cR_{r_p,y}(a_p-),\, y\ge r_p)$ are $\rG^{b_p,s_p}_{s_p}$-measurable by Proposition \ref{c:measurability} (i)-(ii).
\end{enumerate}
In (ii) and (iii), we used Proposition~\ref{p:decomp}~(ii) with $Y=(\cG^{T_p}-L(T_p,\cdot),\, y\ge I_{T_p})$. \\

The random variable  $y_p$ defined in \eqref{def:y_p} is a $(\rG^{b_p,s_p}_y)_{y\in \bR}$-stopping time. Indeed $y_p\ge s_p$ and for $s>s_p$, the event $\{y_p\le s\}=\{x_p\le s\}$ is $\rG^{b_p,s_p}_s$-measurable, in view of \eqref{def:xp} and (ii)--(iv) above. 
We then see that  $E_{k,p,1} \in \rG_{y_p}^{b_p,s_p}$ by use of (i) and (iv). The event $E_{k,p,2}$ is measurable w.r.t. the spindles discovered before the entrance time of $\cT_p$ , hence by Proposition \ref{c:measurability} (iii),  $E_{k,p,2} \in \rG_{s_p}^{b_p,s_p}$. Finally, $E_{k,p,3}$ can be reformulated as the event  
\[
(b_q,s_q) \notin \big\{(b,s): s\in [r_p,y_p],\, b\in [\cR_{r_p,s}(a_p-), L(T_p,s)]\big\},\quad \forall\, q\in \{1,\ldots,k\}\setminus\{p\}.
\]

\noindent We deduce that $E_{k,p,3}\in  \rG_{y_p}^{b_p,s_p}$ by (i) and (iv). We proved that $E_{k,p}\in \rG_{y_p}^{b_p,s_p}$. Recall that $W^{T_p,-}$ is independent of $W^{T_p}$ by Proposition \ref{p:markov S} (i), hence of $\widehat{W}^{T_p,-}$. By Proposition \ref{p:green} and the strong Markov property at level $y_p$, conditionally on $\rG^{b_p,s_p}_{y_p}$ and $\{s_p<x_p<\infty\}$, $(\cG^{T_p}_{y}-L(T_p,y),\, y\ge x_p)$ is a ${\rm BESQ}^{-\delta}_d$ process. We have already observed that $\cB_{r_p,y}(a_p)=\cG^{T_p}_y$ for $y\in [r_p,H_{T_p}]$. On $E_{k,p,1}$, for all $y\ge x_p$, we have $\cR_{r_p,y}(a_p-)=L(T_p,y)$. We deduce that on the event $E_{k,p}$, 
\[
\max(\cB_{r_p,y}(a_p)-\cR_{r_p,y}(a_p-),0) = \cG^{T_p}_y-L(T_p,y),\qquad \forall\, y\in [x_p,H_{T_p}].
\]
 This identity stays true when $y\ge H_{T_p}$: the left-hand side is $0$ since $H_{T_p}$ corresponds to the top point of $\cT_p$ on $E_{k,p,1}$, 
 and the right-hand side is $0$ by Remark \ref{r:green line} (ii).  Therefore we proved that, conditionally on $E_{k,p}$, $(\max(\cB_{r_p,y}(a_p)-\cR_{r_p,y}(a_p-),0),\, y\ge x_p)$ is a ${\rm BESQ}^{-\delta}_d$ process. 
 
 It remains to take the union $E_k=\bigcup_{p=1}^k E_{k,p}$ and send $k\to\infty$. By \eqref{eq:p(k)}, we conclude that the process in \eqref{eq:proof_poisson} is indeed a ${\rm BESQ}^{-\delta}_d$ process.
\end{proof}

 \begin{theorem}\label{thm:Levy}
 The process $\xi:=X\circ \kappa \circ\eta$ is a $(1+\frac\delta2)$-stable L\'evy process, spectrally positive starting from $0$. Its L\'evy measure is $\nu^{(-\delta)}_{\mathtt{BESQ}} (\zeta(f) \in  d y)$ and its Laplace exponent is ${2^{-\frac\delta 2}}q^{1+\frac\delta 2} /\Gamma(1+\frac\delta2)$, $q\ge 0$.
 \end{theorem}

\begin{proof}
Fix $u>0$ and set $T=\kappa_u$. By Lemma~\ref{lem:kappa-u-fix} and Proposition \ref{p:charact A} (ii), $u\in (A_{s_{\cT}},A_{t_{\cT}})$ for some spindle $\cT$, i.e. $\varrho_T$ is in the left part of $\cT$. Then equation \eqref{eq:xi-X} shows that $\eta_{\cL(u)}=A_{t_{\cT}}$ and $\xi_{\cL(u)}=X_{t_{\cT}}$.

Define $(\widehat{X}^T,\widehat{\varrho}^T,\widehat{\cB}^T,\widehat{\cR}^T)$ as in Proposition~\ref{p:rhohat kappa} and let $(\widehat{\kappa}^T, \widehat{\cA}^T)$ be associated with $\widehat{\varrho}^T$ and the spindles of $(\widehat{\cB}^T,\widehat{\cR}^T)$, in the same way as in \eqref{def:A-kappa} and \eqref{eq:cA}. 
Recall that we have observed that $\widehat{\cA}^T = \cA \circ \theta_{D_u}$ in the proof of Proposition~\ref{cA:range} and $D_u=A_{t_{\cT}}$ by the discussion before Proposition \ref{cA:range}, which is also $\eta_{\cL(u)}$.   
Let $\widehat{\eta}^{T}_{s} := \eta_{s+\mathcal{L}(u)} - \eta_{\mathcal{L}(u)}$, $s\ge 0$. Then the range of $\widehat{\eta}^T$ is $\widehat{\cA}^T$, and we have 
\[
\widehat{X}^T\circ \widehat{\kappa}^T\circ  \widehat{\eta}^T = \xi_{\cL(u)+\cdot}-\xi_{\cL(u)}.\vspace{-0.1cm}
\]

\noindent The left-hand side is independent of $\rG_u$ and distributed as $\xi$ by Proposition~\ref{p:rhohat kappa}. Recall that  $\cL(u)$ is $\rG_u$-measurable. Observing that $X_{t_{\cT}}=\xi_{\cL(u)}$ is the vertical coordinate of the top point of a spindle discovered before time $\kappa_u$, we deduce by Lemma~\ref{l:Gu-filtration} that $\xi_{\cL(u)}$  is $\rG_u$-measurable. Hence for any $v<u$, $\xi_{\cL(v)}$ is also $\rG_u$-measurable. We conclude that, for any $u>0$, the process 
$
(\xi_{\cL(u)+\ell}-\xi_{\cL(u)},\, \ell\ge 0)
$
is independent of $(\xi_\ell,\,\ell\in [0,\cL(u)])$ and distributed as $\xi$. Using the right-continuity of $\xi$, we deduce that $\xi$ is a L\'evy process. It is spectrally positive by construction, and is recurrent since for example, its range contains the levels of all bottom points of spindles.

    To determine the L\'evy measure of $\xi$, recall from \eqref{eq:xi-X} that there is a bijection between the sizes of jumps $\Delta\xi_\ell$ and the heights of the spindles, i.e. the lifetimes of excursions $\zeta (\mathbf{e}_{\ell})$. 
    Therefore, it follows from Theorem~\ref{thm:PPP} that the L\'evy measure of $\xi$ is  $\nu(\dd y) := \nu^{(-\delta)}_{\mathtt{BESQ}} (f\colon \zeta(f) \in \dd y)$ given by \eqref{eq:nu_BESQ}. 
    
        We now show that $\xi$ is strictly stable with index $1+\frac{\delta}{2}$. Since we also know that $\xi$ is  spectrally positive, the L\'evy--Khintchine exponent is thus uniquely determined by the L\'evy measure. We will use the scaling properties of our model. 
        For $c>0$, let $\varrho^{(c)}_t:=\frac{1}{c}\varrho_{c^2 t}$, $t\ge 0$, which  has the same law as $\varrho$. 
        By Lemma~\ref{l:scaling A} and the arguments above, the set $\cA^{(c)}$ associated to $\varrho^{(c)}$ is $c^{-2}\cA$ and has the same law as $\cA$. 
        Recall that a subordinator is uniquely determined by its range up to a time-scale, then a subordinator with range $\cA^{(c)} = c^{-2}\cA$ must be of the form $\eta^{(b,c)}_\ell:=c^{-2}\eta_{b\ell}$ for some $b>0$. 
    The Poisson point process $\fN$ associated to $\varrho^{(c)}$ is thus $\fN^{(b,c)} = (\mathbf{e}^{(b, c)}_{\ell}, \ell \ge 0)$, 
    where each excursion $y\mapsto \mathbf{e}^{(b,c)}_{\ell} (y)$ satisfies 
    \[
    \mathbf{e}^{(b,c)}_{\ell} (y)
    = \frac{1}{c}\mathbf{e}_{b\ell} (c y), \qquad y\ge 0. \vspace{-0.1cm}
    \]

\noindent By Theorem~\ref{thm:PPP} and the scaling property of \besq, we deduce that $\fN^{(b,c)}$ is a Poisson point process with intensity $bc^{-1-\frac{\delta}{2}}\nu_{\mathrm{BESQ}}^{(-\delta)}$. Requiring that $\fN^{(b,c)}$ has the law of $\fN$, we have $b=c^{1+\frac{\delta}{2}}$.  Setting with natural notation $\xi^{(b,c)}:=X^{(c)}\circ \kappa^{(c)}\circ\eta^{(b,c)}$,  we deduce that  $\xi^{(b,c)}_\ell=c^{-1}\xi_{b\ell}$ by definition\footnote{We compute that $X^{(c)}_t=\frac1cX_{c^2t}$, $A^{(c)}_t=\frac1{c^2}A_{c^2t}$, $\kappa^{(c)}_t=\frac1{c^2}\kappa_{c^2t}$}. Since $\xi^{(b,c)}$ is distributed as $\xi$ and $b=c^{1+\frac{\delta}{2}}$, we conclude that $\xi$ is $(1+\frac{\delta}{2})$-stable.
\end{proof}

\begin{remark}\label{rmk:scaffolding}
Let us regard $\fN$ as a Poisson random measure with intensity  $\mathrm{Leb} \otimes \nu^{(-\delta)}_{\mathtt{BESQ}}$ and write $\fN = \sum_{i\in \mathbb{N}} \delta(t_i,f_i)$.  
Recall that the jumps of $\xi$ are given exactly by the lifetimes of the excursions $f_i$, such that
	\[\sum_{\Delta \xi(t)> 0} \delta(t, \Delta \xi (t)) = \sum_{i\in \mathbb{N}} \delta(t_i, \zeta(f_i)).\] 
Therefore, the L\'evy-It\^o decomposition of $\xi$ leads to the identification of the following limit (see \cite[Theorem 19.2]{Sato}): for every $t\ge 0$,
\begin{equation}\label{eq:scaffolding}
	\xi(t)=  \lim_{z\downarrow 0} \left( 
	\int_{[0,t]\times \{g\in \mathscr{E}\colon \zeta(g)>z\}} \zeta(f)  \fN(ds,df) - \frac{(1+\frac\delta 2)t}{(2z)^{\frac\delta 2}\Gamma(1-\frac\delta 2)\Gamma(1+\frac\delta 2)}
	\right).
\end{equation}
In this sense, the process $\xi$ is called the \emph{scaffolding function} of $\fN$ in the language of \cite{Paper1-1}.
\end{remark}

Based on well-known properties of stable processes and BESQ excursions, many explicit calculations can be performed. We refer to \cite{Paper0,Paper1-2} for several examples. In particular, we present a description of the intersection of the gasket $\mathscr{K}$ with any horizontal line. 
\begin{corollary}\label{cor:stb-sub}
    Fix any $y\in \bR$. The set $(\bR_+\times \{y\}) \cap \mathscr{K}$ is the range of a $\frac\delta2$-stable subordinator with Laplace exponent $\lambda \to \lambda^{\frac\delta 2}$.  
\end{corollary}
\begin{proof}
    It follows from \cite[Proposition~2.9 and Proposition~3.2(i)]{Paper1-2}. 
\end{proof}


\section{Applications}\label{sec:constr}

In this section we use the marked L\'evy process to connect the spindle partition with the self-similar interval partition evolutions  \cite{Paper1-1,Paper1-2,FRSW2}, as well as the shredded sphere \cite{BjCuSi22}.

\subsection{Self-similar interval partition evolutions}
For $M\ge 0$, an \emph{interval partition} $\beta=\{U_i,i\in I\}$ of $[0,M]$ is a (finite or countably infinite) collection of disjoint open intervals $U_i=(a_i,b_i)\subseteq (0,M)$, 
such that the (compact) set of partition points 
$G(\beta):= [0,M]\setminus\bigcup_{i\in I}U_i$ has zero Lebesgue measure.  We refer to $\|\beta\|:=\sum_{U\in\beta}{\rm Leb}(U)$ as the 
\em total mass \em of $\beta$. 
For $c>0$ we also define a scaling map by
\[
c \beta:= \{(ca,cb)\colon (a,b)\in \beta\}.
\]

Let $\mathcal{I}_{H}$ denote the space of all 
interval partitions of $[0,M]$ for all $M\ge 0$. 
By convention, we denote $\emptyset$ the only interval partition that is an empty set.  
This space is equipped with the metric $d_H$, defined by applying the Hausdorff metric to the sets of partition points: 
for every $\gamma,\gamma'\in \mathcal{I}_{H}$, 
\[
d_{H} (\gamma, \gamma')
:= \inf \bigg\{b\ge 0\colon G(\gamma)\subseteq \bigcup_{x\in G(\gamma')} (x-b,x+b),~
G(\gamma')\subseteq \bigcup_{x\in G(\gamma)} (x-b,x+b) \bigg\}.
\]
Although $(\mathcal{I}_H, d_H)$ is not complete, the induced topological space is Polish \cite[Theorem~2.3]{Paper1-0}.
In particular, we identify a vector of positive integers $(n_1, \ldots, n_k)$ (so-called an \emph{integer composition}) with the interval partition of $[0,n_1+\cdots+n_k]$ given by  
\[
 \{ (s_{i-1},s_{i}), 1\le i\le k \}, \quad \text{where}~ s_i=n_1+\cdots+n_i.
\]

Self-similar interval partition evolutions appear as the scaling limits of a class of continuous-time Markov chains, parametrised by $\alpha\in(0,1)$, $\theta\ge 0$, which are closely related to the well-known Chinese restaurant processes due to Dubins and Pitman 
(see e.g.\@ \cite{CSP}) and their ordered variants studied in \cite{PitmWink09}. 
We interpret the model as describing the number of customers at linearly ordered occupied tables, which evolves in continuous-time as customers arrive and depart one-by-one. Given the state $C(t) = (n_1, \ldots , n_k)\in \mathcal{I}_H$ with $n_1, \ldots , n_k\in\bN$ at time $t\ge 0$, the dynamics is as follows:
\begin{itemize}
	\item For each occupied table with $n_i$ customers, 
	a new customer joins that table at rate $n_i - \alpha$;
	\item A new customer starts a new table to the \emph{left} of the leftmost occupied table at rate $\theta$; 
	\item  Between any two adjacent occupied tables, or to the \emph{right} of the rightmost table, a new customer enters and begins a new table there at rate $\alpha$; 
    \item Each customer departs independently at rate $1$. 
\end{itemize}
By convention, the chain jumps from the null vector $\emptyset$ to state $(1)$ at rate $\theta$, hence $\emptyset$ is absorbing if $\theta= 0$. 
At every time $t\ge 0$, let $C(t)$ denote the left-to-right ordered vector of customer counts at occupied tables, viewed as a state in $\mathcal{I}_H$. 
This defines a continuous-time Markov chain $(C(t), t\ge 0)$,  with c\`adl\`ag paths on the metric space $(\mathcal{I}_H, d_H)$, referred to as a \emph{Poissonised up-down ordered Chinese restaurant process (PCRP) with parameters 
	$\alpha$, $\theta$}. With initial state $C(0)$, we denote the process by $\mathrm{PCRP}^{(\alpha)}_{C(0)}(\theta)$.

\begin{theorem}[{\cite[Theorem 1.1]{ShiWinkel-2}}]\label{thm:crp-ip} Let $\alpha\in(0,1)$ and $\theta \ge 0$. 
	For $n\in \bN$, let  $(C^{(n)}(t),\, t\ge 0)$ be a $\mathrm{PCRP}^{(\alpha)}(\theta)$ starting from $C^{(n)}(0)= \gamma^{(n)}$.
	Suppose that  the initial interval partitions $ \frac{1}{n}  \gamma^{(n)}$ converge in distribution to 
	$\gamma\in \mathcal{I}_H$ as $n\to \infty$, under $d_H$.
	Then there exists an $\mathcal{I}_H$-valued path-continuous Hunt process $(\beta(t), t\ge 0)$ starting from $\beta(0)= \gamma$, such that 
	\begin{equation}\label{mainthmeq}
	\Big(\frac{1}{n}  C^{(n)}(2 n t),\, t\ge 0 \Big)
	\underset{n\to \infty}{\longrightarrow}  (\beta(t),\, t\ge 0) , \quad \text{in distribution}, 
	\end{equation}
in the space of $(\mathcal{I}_H,d_H)$-valued c\`adl\`ag functions under the Skorokhod $J_1$-topology (see e.g.\@ \cite{Billingsley}).
\end{theorem}
The limiting diffusion $(\beta(t),\, t\ge 0)$ on $\mathcal{I}_H$ is called a \emph{self-similar interval partition evolution} ($\mathrm{SSIP}^{(\alpha)}(\theta)$-evolution). 
When $\theta =0$ (resp.\ $\theta =\alpha$), it is also called a \emph{type-1 (resp.\ type-0) evolution} in \cite{Paper1-2}.

\subsection{Interval partitions and a Ray--Knight theorem}

We still keep assumption \eqref{eq:mu} 
and set 
\[
\alpha := \frac{\delta}{2}  \in (0,1). 
\]

Given the spindle partition obtained from the coupled flows $\cR$ and $\cB$ as in Definition~\ref{def:red-blue}, we can  define an interval-partition-valued process in a straightforward way.

\begin{theorem}\label{thm:RK-IPE}
Let $\alpha =\frac{\delta}{2}\in (0,1)$. 
Fix any $z<0$ and consider the spindle partition of the region to the left of the flow line $(\cR_{z,z+y} (0), y\ge 0)$.  
For $y\ge 0$, set 
\begin{equation}\label{eq:skewer}
   \beta^y:= \{ (\cR_{r,z+y}(a-), \cB_{r,z+y}(a)) \colon  
 \cT_{(a,r)}\text{ a spindle, } \, \cB_{r,z+y}(a)\le \cR_{z,z+y} (0))\} .
\end{equation}
Then the process $\big(\beta^y,\, y\in [0,|z|)\big)$ is an $\mathrm{SSIP}^{(\alpha)}(\alpha)$-evolution starting from zero; and $(\beta^y,\, y\ge |z|)$ is an $\mathrm{SSIP}^{(\alpha)}(0)$-evolution. 
The total mass process $\big(\|\beta^y\|,\, y\ge 0)\big)$ is equal to $(\cR_{z,z+y} (0), y\ge 0)$. 
\end{theorem}

Note that Corollary~\ref{c:spindle zero measure} alone does not suffice to determine whether $\|\beta^y\| = \cR_{z,z+y}(0)$ for all $y\ge 0$. However, by Theorem~\ref{thm:RK-IPE} we  strengthen Corollary~\ref{c:spindle zero measure} into the following result:
\begin{corollary}\label{c:zero measure 1D}
Almost surely, for every $y\in \bR$, 
 the set $(\bR_+\times \{y\}) \cap \mathscr{K}$ has $\bR$-Lebesgue measure 0.
\end{corollary}

Recall that  $\cR$ is built from the local time of the PRBM $X$, so we have $\cR_{z,z+y} (0) = L(T^{X}_{z}, z+y)$ for $y\ge 0$, where $T^{X}_{z}= \inf\{t\ge 0\colon X_t <z\}$. 
Theorem~\ref{thm:RK-IPE} provides a measurable refinement of the Ray--Knight theorem by partitioning the local time process into intervals.

To prove Theorem~\ref{thm:RK-IPE} we make use of the representation developed in Section~\ref{sec:markedLevy}. With the notation therein, let $\mathbf{N}$ be the Poisson point process defined in Theorem~\ref{thm:PPP} and $\xi = X\circ \kappa\circ \eta$ the L\'evy process obtained in Theorem~\ref{thm:Levy}.

\begin{lemma}\label{lem:xi-X}
    For every $x\le 0$, 
set $T^{X}_{x}= \inf\{t\ge 0\colon X_t <x\}$
and  $T^{\xi}_{x}= \inf\{\ell\ge 0\colon\ \xi_\ell <x\}$. Then we have almost surely  $\kappa\circ \eta(T^{\xi}_x) = T^{X}_x$ for every $x\le 0$. 
\end{lemma}
\begin{proof}
 We notice that, almost surely for every $x<0$,  
    \[
    x = \inf_{\ell < T^{\xi}_x} X\circ\kappa\circ \eta(\ell) \ge \inf_{t < \kappa\circ \eta(T^{\xi}_x)} X_t. 
    \]
    This inequality holds because the time change ensures that the first infimum is taken over a subset of the times considered in the second infimum. Consequently, we have $\kappa\circ \eta(T^{\xi}_x) \ge T^{X}_x$. 

    For the reverse inequality, note that we can find $t>T^{X}_x$ arbitrarily close to $T^X_x$ such that $\varrho_t \in (0,\infty) \times (-\infty ,x)$ and that $\varrho_t$ is in the interior of a spindle $\cT$. This implies that the bottom point of $\cT$ is below level $x$ and that $s_{\cT} <t$. So we have $X_{s_{\cT}} <x$. By \eqref{eq:xi-X} 
    one has $\xi_{\ell-}<x$ for $\ell$ such that $\eta_{\ell-}=A_{s_{\cT}}$. Let $\ell'<\ell$ be such that $\xi_{\ell'}<x$. 
    Since $\kappa\circ \eta$ is strictly increasing, we have
    $\kappa(A_{s_{\cT}})>\kappa\circ \eta(\ell')>\kappa\circ \eta(T^{\xi}_x)$. From Lemma \ref{l:kappa_A}, we obtain  $t > s_{\cT} > \kappa\circ \eta(T^{\xi}_x)$. Since $t>T^{X}_x$ can be arbitrarily close, we conclude that $T^{X}_x \ge \kappa\circ \eta(T^{\xi}_x)$.  
\end{proof}

 \begin{proof}[Proof of Theorem~\ref{thm:RK-IPE}]
Let us write $\fN= \sum_{i\in \mathbb{N}} \delta_{(t_i,f_i)}$, such that $f_i$ is the width of the spindle $\cT_i$. 
Recall that $\fN$ is a Poisson random measure on $\bR_+\times \mathscr{E}$ with intensity measure $\mathrm{Leb} \otimes \nu^{(-\delta)}_{\mathtt{BESQ}}$, and that $\xi$ is 
a spectrally positive stable L\'evy process of index $1+\frac\delta2$. Recall from Remark~\ref{rmk:scaffolding} that the jumps of $\xi$ are given exactly by the lifetime of the excursions $f_i$, that is 
	\[\sum_{t\in[0,T]\colon\!\Delta \xi(t)> 0} \delta_{(t, \Delta \xi (t))} = \sum_{i\in \mathbb{N}} \delta_{(t_i, \zeta(f_i))}.\]

	Fix $z< 0$.   
	For the pair $(\fN,\xi)$ stopped at $T^{\xi}_{z}$, we define its \emph{skewer} at level $y\ge 0$ by the interval partition
	\begin{equation}\label{eq:skewer def}
		\widetilde{\beta}^y :=     
		\{ (M^y(t-),M^y(t)) \colon M^y(t-)<M^y(t), t\in [0,T^{\xi}_{z}] \}, 
	\end{equation}
    where, with the convention that $M^y(0-)=0$,  
    	\[
	M^y(t) = \int_{[0,t\wedge T^{\xi}_{z}]\times\mathscr{E}} f\big(z+ y- \xi(s-) \big)\ind{z+y\ge \xi(s-)} \fN(ds,df)
    =\sum_{i\in\mathbb{N}\colon\! t_i\in[0,t\wedge T^{\xi}_{z}]}f_i\big(z+y-\xi(t_i-)\big)\ind{z+y\ge \xi(t_i-)}.
	\]
Then it is known from   \cite[Theorem~1.8]{Paper1-2}
that 
 	$(\widetilde{\beta}^y,\, y\in [0,|z|\,])$ is an $\mathrm{SSIP}^{(\alpha)}(\alpha)$-evolution starting from zero and $(\widetilde{\beta}^y,\, y\ge |z|)$ is an $\mathrm{SSIP}^{(\alpha)}(0)$-evolution.

Finally, let us identify the two processes $\beta$ with $\widetilde{\beta}$, from which the claim would follow. 
By Proposition \ref{prop:explore} (iii), at time $T^X_z$, all spindles discovered by $\varrho$ before time $T^X_z$ have been fully explored. They are exactly those spindles whose right boundary are at the left of $\cR_{z,z+y} (0)=L(T^X_z,z+y)$ (those making contribution in \eqref{eq:skewer}).   
On the other hand, due to Lemma~\ref{lem:xi-X}, these spindles are given by $\mathbf{N}_z:= \fN\mid_{(0, T^{\xi}_{z}]\times \mathscr{E}}$. 
For each atom $(s_i, f_i)$ of $\fN$, which is given by the width of a spindle $\cT_{(a,r)}$, we have by \eqref{eq:xi-X} that $\xi_{s_i-} = r$ and thus 
\(M^y(s_i ) - M^y(s_i -) = \cB_{r,z+y}(a)-\cR_{r,z+y}(a-). \)
This implies that $\|\beta^y\| =\|\widetilde{\beta}^y\|$ and that $M^y(s_i )\le \cB_{r,z+y}(a)$. 

Moreover, it is known from \cite[Theorem 1.4]{Paper1-2} that the total mass evolution $(\|\widetilde{\beta}^y\| ,y\ge 0)$ is a $\besq(2\alpha\,|_{|z|} \,0)$ process starting from $0$, so has the same law as $(\cR_{z, z+y}(0) ,y\ge 0)$. But we also know that almost surely $\|\widetilde{\beta}^y\| \le \cR_{z, z+y}(0)$ for all $y\ge 0$. This permits us to identify the two processes, that is $\|\widetilde{\beta}^y\| = \cR_{z, z+y}(0)$ for all $y\ge 0$. 

Therefore, we conclude that almost surely, 
$\|\beta^y\|=\|\widetilde{\beta}^y\| = \cR_{z, z+y}(0)$ for all $y\ge 0$. 
As a consequence, for each atom $(s_i, f_i)$ of $\fN$ and its corresponding spindle $\cT_{(a,r)}$, we also have $M^y(s_i )= \cB_{r,z+y}(a)$ and $M^y(s_i- )= \cR_{r,z+y}(a-)$ for all $y\ge \xi_{s_i-} -z$. We hence deduce that $(\beta^y, y\ge 0)= (\widetilde{\beta}^y, y\ge 0)$. 
In particular, $(\beta^y, y\ge 0)$ is indeed an interval-partition-valued process with each $\beta^y$ giving an interval partition of $[0,\cR_{z, z+y}(0)]$ (in the sense that the complement has zero one-dimensional Lebesgue measure). 
\end{proof}

 For $\alpha\in(0,1)$ and $\theta\ge 0$, a \emph{Poisson--Dirichlet $(\alpha,\theta)$-interval partition} was introduced by \cite{GnedPitm05,PitmWink09}, with motivation from regenerative compositions and random trees.  For each Poisson--Dirichlet $(\alpha,\theta)$-interval partition, the lengths of intervals (ranked from longer to shorter) follow the usual Poisson--Dirichlet $(\alpha,\theta)$ distribution. 
 
For a $\besq(2-2\alpha)$ bridge $(b_t ,\,t\in [0,1])$ from zero to zero, let  $A_{zero} = \{r\in [0,1]\colon b_r =0 \} $ be its zero set. Then the interval partition given by the open interval components of $[0,1]\setminus \overline{A_{zero}}$ is a  Poisson--Dirichlet $(\alpha,\alpha)$-interval partition. 
Consider the left-right reversal of the zero sets of a $\besq(2-2\alpha)$ process $(Z_t ,\,t\in [0,1])$, that is  $A'_{zero} = \{1-r\in [0,1]\colon Z_r =0 \}$, then we similarly obtain a Poisson--Dirichlet $(\alpha,0)$-interval partition. See \cite[Examples 3--4 and Sections 8.3--8.4]{GnedPitm05}.

\begin{corollary}
    Let $z< 0$, and consider the interval partition  of the interval $[0,\cR_{z,0} (0)]$ given by 
\[
\beta_{z}:=
\{ (\cR_{r,0}(a-), \cB_{r,0}(a)) \colon  
 \cT_{(a,r)}\text{ a spindle, } \, \cB_{r,0}(a)\le \cR_{z,0} (0)).  
\]
Then $\beta_{z}\eqdis G \gamma$, where $G$ has Gamma distribution $(\alpha, \frac{1}{2|z|})$ and $\gamma$ is a Poisson--Dirichlet interval partition with parameter $(\alpha,\alpha)$, independent of each other.  
\end{corollary}
\begin{proof}
Since we know that $\beta_{z}$ has the same law as an $\mathrm{SSIP}^{(\alpha)}(\alpha)$-evolution starting from empty at time $|z|$, the statement follows from \cite[Proposition~4.1]{Paper1-2}.
\end{proof}

\begin{corollary}
\label{cor:RK-IPE}
Let $\alpha =\frac{\delta}{2}\in (0,1)$ and $\beta^0 \in \mathcal{I}$. 
Consider a $\besq(\delta)$ flow $\widetilde{\cR}$ and a $\besq(0)$ flow $\widetilde{\cB}$. 
Define spindles associated to $(\widetilde{\cR},\widetilde{\cB})$ and consider the spindle partition of the region to the left of the flow line $(\widetilde{\cR}_{0,y} (0),\, y\ge 0)$.  
For $y\ge 0$, set 
\begin{equation}
   \widetilde\beta^y:= \{ (\widetilde{\cR}_{r,y}(a-), \widetilde{\cB}_{r,y}(a)) \colon  
 \cT_{(a,r)}\text{ a spindle, } \,  
 a\le \widetilde{\cR}_{0,r}(0) \} .
\end{equation}
Then the process $\big(\widetilde\beta^y,\, y\ge 0\big)$ is an $\mathrm{SSIP}^{(\alpha)}(\alpha)$-evolution starting from zero. 
\end{corollary}
\begin{proof}
    Fix any $x>0$. 
    Let $(\beta^y ,\, y\in [0,x])$ be given as in Theorem~\ref{thm:RK-IPE} associated with $(\cR,\cB)$ and $z = -x<0$, then it has the same law as  $(\widetilde\beta^y ,y\in [0,x])$. Indeed, $\beta$ is determined by the flow lines in the strip $\mathbb{R}_+\times [-x,0]$, which has the same law as $(\widetilde{\cR},\widetilde{\cB})$ in $\mathbb{R}_+\times [0,x]$ (up to a vertical shift). 
    We conclude by Theorem~\ref{thm:RK-IPE} that $\big(\widetilde\beta^y,\, y\ge 0\big)$ is an $\mathrm{SSIP}^{(\alpha)}(\alpha)$-evolution starting from zero. 
\end{proof}

We next construct interval partition evolutions starting from an arbitrary state. To this end, we define spindles restricted to $\mathbb{R}_+^2$ associated 
with the two flows $(\cR,\cB)$. 
Specifically, we only look at the family of $\cT_{(a,r)}$ with $a,r\ge 0$ and call it a \emph{spindle in $\mathbb{R}_+^2$} if $\cT_{(a,r)}$ is maximal under inclusion in this collection.

\begin{theorem}\label{thm:IPE-0}
Let $\alpha =\frac{\delta}{2}\in (0,1)$ and $\beta^0 \in \mathcal{I}$. 
For any $(b,b')\in \beta^0$ and $y\ge 0$, 
let 
\begin{align*}
U_{(b,b')}^y &= (\cR_{0,y}(b) , \cB_{0,y}(b')), \qquad y<\zeta_{(b,b')}:=\inf \{y\ge 0\colon \cR_{0,y}(b) = \cB_{0,y}(b')\},\\
 \beta^y_{(b,b')} &:= \{ (\cR_{r,y}(a-), \cB_{r,y}(a)) :
 \cT_{(a,r)}\text{ a spindle in $\mathbb{R}_+^2$, } \, \cB_{0,r}(b')< a\le \cR_{0,r} (b')\}.
 \end{align*}

Set 
\begin{equation}\label{eq:skewer-0}
   \beta^y:= \bigcup_{(b,b')\in \beta^0} \bigg( \{U_{(b,b')}^y\}\ind{y<\zeta_{(b,b')}} \cup \beta^y_{(b,b')}\bigg), \qquad y\ge 0.
\end{equation}
Then the process $\big(\beta^y,\, y\ge 0 \big)$ is an $\mathrm{SSIP}^{(\alpha)}(0)$-evolution starting from $\beta^0$. 
\end{theorem}
For each $(b,b')\in \beta^0$, we notice that boundary of spindles in $\mathbb{R}_+^2$ cannot cross the flow lines $(\cB_{0,r}(b'),r\ge 0)$ or $(\cR_{0,r}(b'),r\ge 0)$ by the comparison principle; so in this theorem we merge the spindles between each pair of these two flow lines, keeping the others, and consider the resulting interval partition evolution.    

\begin{proof}
Fix any $(b,b')\in \beta^0$. 
Applying  Proposition~\ref{p:decomp} and notations therein to $W$ and $Y = (Y_r = \cR_{0,r}(b),\, r\ge 0)$, we deduce that, the red and blue flows shifted by $Y$, denoted by $\cR^+$ and $\cB^+$ respectively, are $\besq(0)$ and $\besq(-\delta)$ flows driven by $W^+_Y$ respectively, independent of $W^-_Y$. 
In particular, $(U^y_{(b,b')},y\ge 0)$ shifted by $Y$, is a $\besq(-2\alpha)$ starting from $b'-b$. 

Let $Y' = (Y'_r = \cB^+_{0,r}(b'-b), r\ge 0)$ and $\zeta'$ be its absorption time. We know $\zeta'=\zeta_{(b,b')}$. 
    Applying  Proposition~\ref{p:decomp} again to $W^{+}_Y$ and $Y'$, we deduce that, shifting by $Y'$ and $\zeta'$,  
    \[
        \cB'_{r,x}(a) 
    := \max (\cB^+_{r+\zeta',x+\zeta'}(a+Y'_r) - Y'_x,0),  \qquad -\zeta'\le r\le x, a\ge 0,
    \]
  is a $\besq(0\,|_0\, -\delta)$ flow driven by $W^+_{Y'}$, and 
      \[
        \cR'_{r,x}(a) 
    := \max (\cR^+_{r+\zeta',x+\zeta'}(a+Y'_r) - Y'_x,0),  \qquad -\zeta'\le r\le x, a\ge 0,
    \]
      is a $\besq(\delta\,|_0\, 0)$ flow driven by $W^+_{Y'}$. 
      We know that boundary of spindles of $(\cR,\cB)$ in $\mathbb{R}_+^2$ cannot cross the flow lines $(\cB_{0,r}(b'),r\ge 0)$ or $(\cR_{0,r}(b'),r\ge 0)$ by the comparison principle. Therefore, spindles driven by $(\cR',\cB')$ coincide with those of $(\cR,\cB)$ in $\mathbb{R}_+^2$, with a shift of $Y+Y'$.      
      Applying Theorem~\ref{thm:RK-IPE} to $(\cR',\cB')$ associated with $W^+_{Y'}$, we deduce that 
      $ (\beta^y_{(b,b')}, y\ge 0)$ shifted by $Y'+Y$, is a  an $\mathrm{SSIP}^{(\alpha)}(\alpha)$-evolution starting from zero, running until $\zeta'$, and then continues as an $\mathrm{SSIP}^{(\alpha)}(0)$-evolution. 
      Moreover, it is independent of $(U^y_{(b,b')},y\ge 0)$ shifted by $Y$, both are independent of $W^-_{Y}$.  
      Therefore, $ \{U_{(b,b')}^y\}\ind{y<\zeta_{(b,b')}} \cup \beta^y_{(b,b')}$ shifted by $Y$ is an $\mathrm{SSIP}^{(\alpha)}(0)$-evolution starting from $(0,b'-b)$ by the decomposition \cite[Proposition 3.15]{Paper1-1}, independent of  $W^-_{Y}$.

      This argument applies to every $(b,b')\in \beta^0$; as a result, we can view $(\beta^y,y\ge 0)$ as the concatenation of $\mathrm{SSIP}^{(\alpha)}(0)$-evolution starting from each interval element, which is the $\mathrm{SSIP}^{(\alpha)}(0)$-evolution starting from $\beta^0$ by definition. 
\end{proof}

\begin{theorem}
\label{thm:IPE-alpha}
Let $\alpha =\frac{\delta}{2}\in (0,1)$ and $\beta^0 \in \mathcal{I}$. 
Consider a $\besq(2\delta \,|_0\, \delta)$ flow $\widetilde{\cR}$ and a $\besq(\delta \,|_0\, 0)$ flow $\widetilde{\cB}$. 
Define spindles associated to $(\widetilde{\cR},\widetilde{\cB})$. 
Let $\beta_1:=(\beta^y_1, y\ge 0)$
be as in Corollary~\ref{cor:RK-IPE}, and $\beta_2:=(\beta^y_2, y\ge 0)$ be as in Theorem~\ref{thm:IPE-0}. 
Then the process $\big(\beta^y_1\cup \beta^y_2,\, y\ge 0\big)$ is an $\mathrm{SSIP}^{(\alpha)}(\alpha)$-evolution starting from $\beta^0$. 
\end{theorem}
\begin{proof}
    Let $Y= (\widetilde{\cR}_{0,y} (0), y\ge 0)$. Apply Proposition~\ref{p:decomp} with $Y$, the shifted flows $(\widetilde{\cR}^+,\widetilde{\cB}^+)$ are $\besq(0)$ and $\besq(-\delta)$ flows driven by $W^+_Y$. Then it follows from Theorem~\ref{thm:IPE-0} that $\beta_2$ shifted by $Y$ is an $\mathrm{SSIP}^{(\alpha)}(0)$-evolution starting from $\beta^0$, independent of $\beta_1$. 
    Then the theorem is a consequence of the form of the semigroups of $\mathrm{SSIP}^{(\alpha)}(\alpha)$, see e.g.\ \cite{Paper1-2} or \cite[Proposition 2.9]{ShiWinkel-1}. 
\end{proof}

\subsection{De-Poissonization and Perkins' disintegration}

Let $(\beta^y,\, y\ge 0)$ be an $\mathrm{SSIP}^{(\alpha)}(0)$-evolution or an $\mathrm{SSIP}^{(\alpha)}(\alpha)$-evolution starting from some $\beta^0\in\mathcal I$.
Consider the total mass $(\Vert \beta^y\Vert,\, y\ge 0)$ and the time-change
\[
\rho(u) = \inf\left\{ y\ge 0: \int_0^y \frac{\dd z}{\Vert\beta^z\Vert} > u\right\},\qquad u\ge 0.
\]
Denote $h:=\Vert \beta^0\Vert$.
As shown in \cite[Theorem 1.4 \& 1.6]{Paper1-2}, in the case of $\mathrm{SSIP}^{(\alpha)}(\alpha)$-evolution, resp. $\mathrm{SSIP}^{(\alpha)}(0)$-evolution, the total mass process $(\Vert \beta^y\Vert,\, y\ge 0)$ is a $\besq^{2\alpha}_h$, resp. $\besq^0_h$ process, and the so-called \emph{de-Poissonized process}
\begin{align}\label{eq:beta_bar}
(\overline{\beta}^u,\, u\ge 0):= \left(\big\Vert \beta^{\rho(u)} \big\Vert^{-1} \beta^{\rho(u)},\, u\ge 0\right)
\end{align}
is a path-continuous Hunt process on $(\mathcal{I}_1,d_{\mathcal I})$ where $I_1 := \{\beta\in I : \Vert \beta \Vert = 1\}$, with a stationary distribution given by $\mathrm{PDIP}(\alpha,\alpha)$, resp. $\mathrm{PDIP}(\alpha,0)$.

Having established the connection between BESQ flows and the SSIP evolutions, we can apply Perkins' disintegration theorem for flows (developed in \cite{aidekon2023stochastic}) to deduce the following result, thereby solving Conjecture~2 in \cite{Paper1-2}.

\begin{theorem}\label{thm:IPE-independence}
In both cases, the process $(\overline{\beta}^u,\, u\!\ge\! 0)$ is independent of the total mass process $(\Vert \beta^y\Vert,\, y\!\ge\! 0)$.
\end{theorem}

\begin{proof}
We apply the realization of $\mathrm{SSIP}^{(\alpha)}(0)$-evolution via coupled BESQ flows in Theorem~\ref{thm:IPE-0} and $\mathrm{SSIP}^{(\alpha)}(\alpha)$-evolution in Theorem~\ref{thm:IPE-alpha}. Let us deal with the first case.
Let $\cR$ and $\cB$ be respectively a $\besq^{0}$ and a $\besq^{-\delta}$ flow driven by the same white noise $W$. Let $\beta^y$ be defined via \eqref{eq:skewer-0}. Simply rewrite $f(y):=\Vert \beta^y\Vert$. By Corollary~\ref{c:zero measure 1D}, $f(y) = \cR_{0,y}(h)$ for all $y\ge 0$. Then define for all $t\ge s\ge 0$ and $v\in[0,1]$,
\begin{align}\label{eq:Jacobi}
\overline{\cB}_{s,t}(v) := \frac{\cB_{\rho(s),\rho(t)}(v f\circ\rho(s))}{f\circ\rho(t)},\qquad \overline{\cR}_{s,t}(v) :=\frac{\cR_{\rho(s),\rho(t)}(v f\circ\rho(s))}{f\circ\rho(t)}.
\end{align}
We can also define the left-continuous version of $\overline{\cB}$ and $\overline{\cR}$ via \eqref{def:left S}, which is actually equivalent to \eqref{eq:Jacobi} but with the left-continuous version of $\cB$ and $\cR$ rather than $\cB$ and $\cR$ themselves. 
By \cite[Theorem 5.9]{aidekon2023stochastic}, the flows $\overline{\cB}$ and $\overline{\cR}$ are respectively a $\mathrm{Jacobi}(-\delta,\delta)$ and a $\mathrm{Jacobi}(0,0)$ flow (which are indeed driven by the same martingale measure appearing in \cite[(5.29)]{aidekon2023stochastic}), both independent of $f$. 
Consider some spindle $\cT_{(a,r)}$ lying on the left of $f=\cR_{0,\cdot}(h)$ with $r>0$. Then we can write $r=\rho(s)$ and $a=vf(r)=vf\circ\rho(s)$. The spindle $\cT_{(a,r)}$, under the transformation of flows via \eqref{eq:Jacobi}, is mapped to the region $\overline{\cT}_{(v,s)}:=\{(w,t):t>s,\, \overline{\cR}_{s,t}(v-)\le w<\overline{\cB}_{s,t}(v)\}$. We will refer to such region $\overline{\cT}_{(v,s)}$ as a \emph{J-spindle}. 

We also define the dual of $\overline{\cB}$ and $\overline{\cR}$ by setting, for all $0\ge t\ge s$ and $v\in[0,1]$,
\[
\overline{\cB}^*_{s,t}(v) := \inf\{w>0:\overline{\cB}_{-t,-s}(w)>v\}\wedge 1,\qquad  \overline{\cR}^*_{s,t}(v) := \inf\{w>0:\overline{\cR}_{-t,-s}(w)>v\}\wedge 1.
\]
By the definition of dual of BESQ flows in \eqref{def:dual}, we derive that
\begin{align}\label{eq:relation dual}
\overline{\cB}^*_{s,t}(v) = \frac{\cB^*_{-\rho(-s),-\rho(-t)}(v f\circ\rho(-s))}{f\circ\rho(-t)}\wedge 1,\qquad \forall\, 0\ge t\ge s,\ v\in[0,1].
\end{align}
The same equality holds with $(\overline{\cB}^*,\cB^*)$ replaced by $(\overline{\cR}^*,\cR^*)$. Recall that $f\circ\rho(-t)=\cR_{0,\rho(-t)}(h)$. Therefore $\overline{\cR}^*_{s,t}(v)<1$ indeed for all $0\ge x\ge r$ and $v\in[0,1)$ by Proposition~\ref{p:properties dual}~(i).
By the description of bottom point of spindles in Proposition~\ref{p:ancestor equiv} and the relation of two kinds of dual flows in \eqref{eq:relation dual}, we deduce that for all $(w,t)\in \overline{\cT}_{(v,s)}$,
    \begin{align*}
        s = \inf\{s'\in[0,s]: \overline{\cB}^*_{-t,-s'}(w) = \overline{\cR}^*_{-t,-s'}(w)\}\ge 0,\qquad v= \overline{\cB}^*_{-t,-s}(w) = \overline{\cR}^*_{-t,-s}(w)<1.
    \end{align*}
From the description above, we obtain that all J-spindles are measurable with respect to the coupled flows $(\overline{\cB},\overline{\cR})$.
Finally, combining the definition of $\overline{\beta}$ in \eqref{eq:beta_bar} and that of $\beta$ in \eqref{eq:skewer-0}, we derive that
\[
\overline{\beta}^u = \bigcup_{(w,w')\in \overline{\beta}^0} \left(\{ (\overline{\cR}_{0,u}(w),\overline{\cB}_{0,u}(w'))\}\ind{u<\overline{\zeta}_{(w,w')}} \cup \overline{\beta}^u_{(w,w')}\right),
\]
where
\begin{align*}
    \overline{\zeta}_{(w,w')} &:= \inf\{u\ge 0: \overline{\cR}_{0,u}(w)=\overline{\cB}_{0,u}(w')\} = \rho^{-1}(\zeta_{(wh,w'h)}),\\
    \overline{\beta}^u_{(w,w')} &:= \{(\overline{\cR}_{s,u}(v-),\overline{\cB}_{s,u}(v)): s\ge 0,\, \overline{\cT}_{(v,s)}~\text{a J-spindle,}~\overline{\cB}_{0,r}(w')<v<\overline{\cR}_{0,r}(w')\}.
\end{align*}
It follows that the process $\overline{\beta}$ is measurable with respect to the flows $(\overline{\cB},\overline{\cR})$, therefore also independent of $f=\Vert\beta\Vert$. 
The second case follows by a similar argument; the only difference is the dimension of red and blue flows $(\widetilde{\cR},\widetilde{\cB})$.
\end{proof}

\subsection{A connection with the stable shredded half-plane}
\label{s:shredded}

Let $\alpha = \frac{\delta}{2}\in (0,1)$ and consider the $1+\alpha$-stable process $\xi$ obtained from Theorem~\ref{thm:Levy}.  
Recall that the process $\ki$ in \eqref{def:ki}  has range the gasket $\rK$ by Proposition \ref{p:ki}. The following proposition characterizes when two numbers are mapped to the same point by $\ki$.

\begin{proposition}
Almost surely, for any $\ell_1,\ell_2\in \bR_+$, $\ki(\ell_1)=\ki(\ell_2)$ if and only if  
\begin{equation}\label{eq:identification}
    \xi_{\ell_1}=\xi_{\ell_2} \in \Big\{\inf_{\ell\in[\ell_1\land \ell_2,\ell_1\lor \ell_2]} \xi_\ell,\sup_{\ell\in[\ell_1\land \ell_2,\ell_1\lor \ell_2]} \xi_\ell\Big\}
\end{equation}
if and only if 
\begin{equation}\label{eq:identification bis}
    \xi_{\ell_1}=\xi_{\ell_2} \text{ and }~ \forall\, \ell \in (\ell_1\land \ell_2,\ell_1\lor\ell_2),\ \xi_{\ell-} \ge \xi_{\ell_1} \text{ or }\xi_{\ell}\le \xi_{\ell_1}.
\end{equation}
\end{proposition}

\begin{proof}
We may assume that $\ell_1<\ell_2$.  Note that \eqref{eq:identification} implies \eqref{eq:identification bis}. Therefore it is enough to show that $\ki(\ell_1)=\ki(\ell_2)$ implies \eqref{eq:identification} and $\ki(\ell_1)\neq \ki(\ell_2)$ implies the negation of \eqref{eq:identification bis}. Let $t_1=\kappa\circ \eta(\ell_1)$ and $t_2=\kappa\circ \eta(\ell_2)$. Observe that $\ki(\ell_1)=\varrho_{t_1}$, $\ki(\ell_2)=\varrho_{t_2}$ by definition of $\ki$ in \eqref{def:ki}. 

Suppose that $\varrho_{t_1}=\varrho_{t_2}$. Then $X$ makes an excursion above or below $X_{t_1}=X_{t_2}$ on the time interval $[t_1,t_2]$, and we have $\inf_{t\in[t_1,t_2]} X_t = X_{t_1}=X_{t_2}$ or $\sup_{t\in [t_1,t_2]} X_t=X_{t_1}=X_{t_2}$ respectively. 
As $\xi = X\circ \kappa\circ \eta$, it follows that 
 $\inf_{\ell\in[\ell_1,\ell_2]} \xi_\ell = \xi_{\ell_1}=\xi_{\ell_2}$ or $\sup_{\ell\in[\ell_1,\ell_2]} \xi_\ell = \xi_{\ell_1}=\xi_{\ell_2}$ which is \eqref{eq:identification}.

 Suppose now that $\varrho_{t_1}\neq \varrho_{t_2}$. If $X_{t_1}\neq X_{t_2}$, then $\xi_{\ell_1}\neq \xi_{\ell_2}$. If $X_{t_1}=X_{t_2}=x$, then $L(t_2,x)>L(t_1,x)$.
By Proposition~\ref{p:charact A} (iv), $\kappa(\cA)\subseteq \mathcal{E}^c$, therefore $t_1$ and $t_2$ are not in $\cE$. By Corollary~\ref{c:zero measure 1D}, there exists some spindle $\cT=\cT_{(a,r)}$ with $r<x$, such that $\cR_{r,x}(a-)<\cB_{r,x}(a)$ and $(\cR_{r,x}(a-),\cB_{r,x}(a))\subseteq(L(t_1,x),L(t_2,x))$. The inequality $L(t_2,x)\ge \cB_{r,x}(a)$ implies that $t_2>t_{\cT}$. As for $t_1$, either $L(t_1,x)< \cR_{r,x}(a-)$ hence $t_1< s_{\cT}$ is immediate, or $L(t_1,x)=\cR_{r,x}(a-)$, i.e. $\varrho_{t_1}\in \cT$, but since $t_1\in \cE^c$, we still get $t_1<s_{\cT}$ by the definition of $\cE$ in \eqref{def:explore}.  
By \eqref{eq:xi-X}, the spindle corresponds to a jump time $\ell$ of $\xi$ with $\xi_{\ell-} = X_{s_{\cT}} < x < X_{t_{\cT}} = \xi_{\ell}$, $\eta_{\ell-}=A_{s_{\cT}}$, $\eta_\ell=A_{t_{\cT}}$. Taking the image by $A$, $t_1<s_{\cT}<t_{\cT}<t_2$ implies $\eta_{\ell_1}\le \eta_{\ell-}<\eta_\ell\le \eta_{\ell_2}$, hence $\ell_1\le \ell\le \ell_2$. These inequalities are strict since $\xi_\ell$ and $\xi_{\ell-}$ are not equal to $x$. This is to say, $\xi$ has a jump crossing the level $x$ in the time interval $(\ell_1,\ell_2)$, hence \eqref{eq:identification bis} does not hold. This completes the proof. 
\end{proof}

We follow \cite{BjCuSi22}
 and define the stable shredded half-plane, analogously to the stable shredded sphere introduced in \cite{BjCuSi22}.  For any $\ell_1,\ell_2\in \bR_+$, let 
 \[
 D^u(\ell_1,\ell_2)=\xi_{\ell_1}+\xi_{\ell_2} -2 \inf_{\ell\in[\ell_1\land \ell_2,\ell_1\lor \ell_2]} \xi_\ell, \qquad D^d(\ell_1,\ell_2)=2 \sup_{\ell\in[\ell_1\land \ell_2,\ell_1\lor \ell_2]} \xi_\ell - \xi_{\ell_1}-\xi_{\ell_2}.
 \]

 \noindent Then let $D^*$ be the largest pseudo-distance smaller than $D^u$ and $D^d$, i.e.
 \[
 D^*(\ell_1,\ell_2):= \inf\bigg\{\sum_{i=1}^k D^u(a_i,b_i)+D^d(b_i,a_{i+1}): \ell_1=a_1,   a_{k+1}=\ell_2,b_1,\ldots,a_k,b_k\in \bR_+\bigg\},
 \]
 
 \noindent see \cite[equation (2.11)]{BjCuSi22}. From the definition of $D^*$, we have that $D^*(\ell_1,\ell_2)=0$ implies \eqref{eq:identification bis}, i.e. by \eqref{eq:identification}, $D^u(\ell_1,\ell_2)=0$ or $D^d(\ell_1,\ell_2)=0$. It is clearly an equivalence; see \cite[Theorem 4.3]{BjCuSi22} for another proof of this fact.
 We define the stable shredded half-plane as  $\mathscr{P}_\alpha := \mathbb{R}_+/\sim_{D^*}$ where $\ell_1\sim_{D^*} \ell_2$ if and only if $D^*(\ell_1,\ell_2)=0$. It turns $D^*$ into a metric on $\mathscr{P}_\alpha$. The projection map from $\bR_+$ to $\mathscr{P}_\alpha$ is denoted by $\pi$. Therefore
 \begin{equation}\label{pi-ki}
     \pi(\ell_1)=\pi(\ell_2) \text{ if and only if }\ki(\ell_1)=\ki(\ell_2).
 \end{equation}
 It follows that $\ki$ induces an injective mapping $\overline{\ki}$, by setting 
\begin{equation}\label{eq:bijection}
\begin{aligned}
    \overline{\ki}: \mathscr{P}_\alpha &\to \rK, \\
    \pi(\ell) &\mapsto \ki(\ell).
\end{aligned}
\end{equation}
We want to show that the mapping $\ki$ maps $\bR_+$ to $\rK$, i.e. that $\overline{\ki}$ is surjective. Recall Proposition \ref{p:ki}. Since $\rK$ is the range of $\ki$, it suffices to show that for all jump time $\ell$, there exists $\ell'\ge 0$ such that $\ki(\ell')=\ki(\ell-)$. It is implied by the following lemma.

 \begin{lemma}\label{l:charact K}
 If $\ell$ is a jump time of $\xi$ and $\ell'=\inf\{u>\ell\colon \xi_u=\xi_{\ell-}\}$, then $\kappa\circ\eta (\ell') = \tau_a^r$ where $(a,r)$ is the bottom point of the spindle associated with the jump time $\ell$. In particular, $\ki(\ell')=\ki(\ell-)$.
\end{lemma}
\begin{proof}
     Write $\cT=\cT_{(a,r)}$ for the spindle associated to $\ell$ by \eqref{eq:xi-X}, hence $\xi_{\ell-}=r$ and $\eta_{\ell}=A_{t_{\cT}}$. By Lemma \ref{l:kappa_A} (ii),  $\kappa\circ \eta(\ell)=t_{\cT}$, so that
     \begin{equation}\label{eq:Lkappa}
     L(\kappa \circ \eta(\ell'),r)\ge a.
     \end{equation}
     We have $\xi_u>r$ for $\ell \le u<\ell'$, which implies that the vertical coordinates of the bottom points of all new spindles discovered after time $t_{\cT}$ and before time $\kappa\circ\eta (\ell')$ are greater than $r$. 
     Let $t>\tau_a^r$. By definition, $L(t,r)>a$. By Corollary \ref{c:zero measure 1D}, there is a spindle $\cT'=\cT_{(a',r')}$ which intersects $(a,L(t,r))\times \{r\}$. Therefore $a\le \cR_{r',r}(a'-)<L(t,r)$, hence $s_{\cT'}>s_{\cT}$ (hence also $s_{\cT'}> t_{\cT}$) and $s_{\cT'}<t$. It implies that $\kappa\circ \eta(\ell')\le t$, and $\kappa\circ\eta (\ell')\le \tau_a^r$ by sending $t\to\tau_a^r$.
      Together with \eqref{eq:Lkappa}, $L(\kappa\circ\eta (\ell'),r)=a$. On the other hand,  since there will be jumps crossing level $r$ immediately after $\ell'$,  for every $\varepsilon>0$, there exists a spindle discovered between time $\kappa\circ\eta (\ell')$ and $\kappa\circ\eta (\ell'+\varepsilon)$ with  bottom point below level $r$ and  top point above $r$. Going back to the process $X$, we derive $L(\kappa\circ\eta (\ell'+\varepsilon),r)>L(\kappa\circ\eta (\ell'),r)=a$ for every $\varepsilon>0$. By the right-continuity of $\kappa\circ\eta$, we deduce $\kappa\circ\eta (\ell')\ge \tau_a^r$ as desired.
\end{proof}

Therefore, $\overline{\ki}$ is a bijection between $\mathscr{P}_\alpha$ and the gasket $\rK$. 
We next show that $\overline{\ki}$ is a homeomorphism, where the gasket, seen as a subset of $\bR_+\times \bR$, is equipped with the metric $d_E$, the Euclidean distance in $\bR_+\times\bR$. We start with a lemma.

        \begin{lemma}\label{l:convergence s_n}
        Let $(\ell_n)_{n\ge 1}$ be a sequence in $\bR_+$ such that $\ell_n \to \ell$ as $n\to \infty$. 
        \begin{enumerate}[(i)]
            \item Suppose that $\ell_n\ge \ell$ for all $n$, or $\ell$ is a continuity point of $\xi$. Then $\ki(\ell_n)\to \ki(\ell)$ and $\pi(\ell_n)\to \pi(\ell)$. 
            \item Suppose that $\ell_n<\ell$  for all $n$ and $\ell$ is  a jump time of $\xi$. Let $\ell':=\inf\{u>\ell \colon \xi_u=\xi_{\ell-}\}$. Then $\ki(\ell_n)\to \ki(\ell')$ and $\pi(\ell_n)\to \pi(\ell')$.
        \end{enumerate}
        \end{lemma}
    
    \begin{proof}
            (i) We first show that $\ki(\ell_n)\to \ki(\ell)$. The case $\ell_n\ge \ell$ follows from right-continuity. In the case that $\ell$ is a continuity point of $\xi$, observe that $\ell$ is also a continuity point of $\eta$.
            Since $\varrho\circ\kappa$ is continuous by Lemma \ref{l:kappa_A} (i), $\ell$ is also a continuity point of $\ki$.
            To prove now that $\pi(\ell_n)\to \pi(\ell)$, we just observe  that $D^*(\ell_n,\ell)\le D^u(\ell_n,\ell)$ by definition. 
           
            (ii)  For a jump point $\ell$ of $\xi$, by \eqref{eq:xi-X} the time $\kappa\circ \eta (\ell-)$ is the first visit time of the bottom point of some spindle denoted by $\cT_{(a,r)}$. By Lemma~\ref{l:charact K}, $X$ starts an excursion above $r$ at time $\kappa\circ \eta (\ell-)$, and finishes this excursion at time $\kappa\circ\eta(\ell')=\tau_a^r$. Therefore $\ki(\ell-)=\ki(\ell')$. The second statement follows from $D^*(\ell_n,\ell')\le D^u(\ell_n,\ell')$.
        \end{proof}

\begin{theorem}\label{t:homeomorphism}
    The map $\overline{\ki}:(\mathscr{P}_\alpha,D^*)\to (\rK,d_E)$ is a homeomorphism.
\end{theorem}
\begin{proof}
We first prove that for every $x\in\mathscr{P}_\alpha$, the equivalence class of $x$, as a subset of $\bR_+$, is bounded. Take any $\ell\ge 0$ with $\pi(\ell)=x$. Since $\xi$ is recurrent, there exists some jump time $\ell''>\ell$ with $\xi_{\ell''-}<\xi_\ell<\xi_{\ell''}$. By \eqref{eq:identification bis}, it implies that the equivalence class of $x$, equivalently of $\ki(\ell)$ by \eqref{pi-ki}, is a subset of $[0,\ell'']$ .

    We next prove the continuity of $\overline{\ki}$. 
Let $x_n=\pi(\ell_n)$ be a sequence in $\mathscr{P}_\alpha$ that converges to some $x\in \mathscr{P}_\alpha$. With the notation above, for all $n$ large enough, the equivalence class of $x_n$ is a subset of $[0,\ell'']$. 
Therefore, for any subsequence $(n_k)_{k\ge 1}$, there exists a further subsequence $(n_{k(m)})_{m\ge 1}$ such that either $\ell_{n_{k(m)}}\downarrow \ell_{\infty}$ or $\ell_{n_{k(m)}}\uparrow \ell_{\infty}$ in $\bR_+$. 
If we are in case (i) of Lemma~\ref{l:convergence s_n}, then we have $\ki(\ell_{n_{k(m)}})\to \ki(\ell_{\infty})$ and $\pi(\ell_{n_{k(m)}})\to \pi(\ell_{\infty})$; while in  case (ii) of Lemma~\ref{l:convergence s_n}, we have $\ki(\ell_{n_{k(m)}})\to \ki(\ell')$ and $\pi(\ell_{n_{k(m)}})\to \pi(\ell')$. 
As $\pi(\ell_n) = x_n \to x$ in both cases, we deduce that the subsequence $\bar\ki(x_{n_{k(m)}})$ converges to  $\bar\ki(x)$. So we conclude the convergence of the original sequence $\bar\ki(x_{n})$ to $\bar\ki(x)$. 
The continuity of the inverse of $\overline{\ki}$ is proved similarly.
\end{proof}

\appendix

\section{Appendix}

\subsection{Decomposition of BESQ flows}
\label{s:proof decomp}
\begin{proof}[Proof of Proposition~\ref{p:decomp}]
(i) Suppose first that $\delta_3\ge 0$. Then $\cS_{r,x}^+(a)=\cS_{r,x}(a+Y_r)-Y_x$ by the comparison principle. For $x<\zeta$,  $\cS_{r,x}(a+Y_r)-Y_x$ is solution of the SDE \eqref{eq:BESQflow} with $\delta=\delta_3$ and $\cW_Y^+$ in place of $\cW$. For $x\ge r> \zeta$ (hence $\delta_2<0$), $Y_x=0$ and $\cS_{r,x}^+(a)=\cS_{r,x}(a)$. The regularity conditions are also satisfied so we get (i) in this case. In the case $\delta_3<0$, $\cS_{r,x}^+(a)=\cS_{r,x}(a+Y_r)-Y_x$ until it hits $0$. The comparison principle says that $\cS^+_{r,x}(a)=0$ afterwards. Up to this time (or to $\zeta$ if $\cS_{r,\zeta}^+(a)>0$), $\cS_{r,x}^+(a)$ is solution of the SDE \eqref{eq:BESQflow} with $\delta=\delta_3$ and $\cW_Y^+$ in place of $\cW$. We need to check the regularity conditions. The regularity conditions (i) and (ii) of Definition \ref{def:BESQflow} are clear. The perfect flow property follows from the perfect flow property of $\cS$ and the absorption of $\cS^+$ at $0$. 

(ii)  In the setting of \cite{Walsh1986AnIT}, the noise $\cW$ driving $Y$ or $\cS_{r,x}(a)$ in the case $\delta_3\le 0$ can be replaced by $\ind{[0,Y_x]}\cdot\cW(\cdot,\dd x)=\cW_Y^-(\cdot,\dd x)$. In the case $\delta_3>0$, the same holds until $\cS_{r,x}(a)$ hits $Y$, after which $\cS_{r,x}^-(a)=Y_x$ by the comparison principle.

(iii) It is a consequence of \cite[Proposition 3.2]{AWYskew}.
\end{proof}

\medskip

\begin{proof}[Proof of Proposition~\ref{p:dual S+}]
    Statement (iv) is a consequence of Proposition \ref{p:decomp} (iii). Let us prove (ii).  Let $\cS^*$ be the dual flow of $\cS$, which,  by Proposition \ref{p:BESQdual vary},  is a killed ${\rm BESQ}(\delta_1^*+\delta_3\,|_z \, \delta_2^*+\delta_3)$ flow driven by $-\cW^*$ (observe that $\delta_1^*+\delta_3\ge 2$  hence flow lines may be killed only on $(z,\infty)$). Then, for  $-\zeta^*<r\le x\le -r_0$, the comparison principle implies that $\cS^*_{-x,-r}(Y_{-x}) \ge Y_{-r}$   hence $\cS_{r,x}(a') \le Y_{-x}$ for any $0\le a' < Y_{-r}$  by Proposition \ref{p:properties dual} (ii) with $y=x$ and $b=Y_{-x}$ there. We turn to (iii). The flow line $Y$ is measurable w.r.t. $\cW_Y^-$ by Proposition \ref{p:decomp} (ii). Similarly, let 
\[
\cS^{*,-}_{r,x}(a) := \min(\cS^*_{r,x}(a),Y_x),  \qquad r_0\le r\le x,\, 0\le a < Y_r.
\]

\noindent Then $\cS^{*,-}$ is measurable w.r.t. $\cW_Y^{*,-}$ by another use of Proposition \ref{p:decomp} (ii). Let $0\le a<Y_{-r}$. Observe that 
\[
\{c \in [0,Y_{-x})\,:\, \cS^*_{-x,-r}(c)>a\}=\{c \in [0,Y_{-x})\,:\, \cS^{*,-}_{-x,-r}(c)>a\}.
\]

\noindent If the set above is empty, then $\cS_{r,x}(a)> c$ for all $c<Y_{-x}$  by Proposition \ref{p:properties dual} (i), hence $\cS_{r,x}(a)=Y_{-x}$ by (ii). Otherwise, since $\cS$ is the dual of $\cS^*$, $\cS_{r,x}(a)=\inf\{c\ge0\,:\, \cS^*_{-x,-r}(c)>a\}$ which is $\cS_{r,x}(a)=\inf\{c\in [0,Y_{-x})\,:\, \cS^{*,-}_{-x,-r}(c)>a\}$. We deduce that $\cS^{-}$ is measurable w.r.t. $\cW_Y^{*,-}$, hence also w.r.t. $\cW_Y^-$.
 It remains to prove (i). We let $\cS^{*,+}$  be defined as
\[
\cS^{*,+}_{r,x}(a) := \cS^*_{r,x}(a+ Y_{r})- Y_{x},  \qquad a\ge 0,\, r_0\le r\le x.
\]

\noindent By the comparison principle, $\cS^{*,+}_{r,x}(a)\ge 0$.  On $(z,\zeta^*)$, the flow lines $\cS^*_{r,\cdot}(a+ Y_{r})$ and $Y$ don't touch $0$ hence by Proposition \ref{p:decomp} (i) and (iii), $\cS^{*,+}$ is a ${\rm BESQ}^{\delta_3}$ flow driven by $-\cW^{*,+}_Y$, independent of $\cW_Y^{*,-}$ on this interval.  Let us show that $\cS^{+}$ is the dual flow of $\cS^{*,+}$ on $(-\infty,-r_0)$. Let $r< x< -r_0$ and $a\ge 0$. We want to show 
\begin{equation}\label{eq:dual difference}
    \cS^{+}_{r,x}(a) =\inf\{c\ge 0\,:\, \cS^{*,+}_{-x,-r}(c)>a\}.
\end{equation}

\noindent We first prove 
\begin{equation}\label{upper bound dual difference}
    \cS^{+}_{r,x}(a) \le \inf\{c\ge 0\,:\, \cS^{*,+}_{-x,-r}(c)>a\}.
\end{equation}

\noindent Let $c\ge 0$ such that $\cS^{*,+}_{-x,-r}(c)>a$, i.e. $\cS^{*}_{-x,-r}(c+Y_{-x})>a+Y_{-r}$.   By Proposition \ref{p:properties dual} (ii) with $y=x$ there, it implies that $c+Y_{-x}\ge \cS_{r,x}(a+Y_{-r})$. Hence $c\ge \cS_{r,x}(a+Y_{-r})-Y_{-x}$ and $c\ge \cS^+_{r,x}(a)$ since $c\ge 0$. It proves \eqref{upper bound dual difference}.  We prove the reverse inequality. Let $c\ge 0$ such that $ \cS^{*,+}_{-x,-r}(c)\le a$ (if it exists). By definition, $\cS^*_{-x,-r}(c+Y_{-x})\le a+Y_{-r}$, hence $\cS_{r,x}(a+Y_{-r})\ge c+Y_{-x}$ by definition of the dual flow, so $\cS_{r,x}^+(a)=\cS_{r,x}(a+Y_{-r})-Y_{-x}\ge c$.  It completes the proof of \eqref{eq:dual difference}.

Notice that on $(-\infty,-\zeta^*)$, $\cS^+=\cS$ hence is a ${\rm BESQ}^{2-\delta_2^*-\delta_3}$ flow driven by $\cW$ which is $\cW_Y^+$ on this interval. On $(-\zeta^*,-r_0)$, $\cS^+$ is the dual of $\cS^{*,+}$ which is a ${\rm BESQ}^{\delta_3}$ flow driven by $-\cW^{*,+}_Y$. By Proposition \ref{p:BESQdual vary}, $\cS^+$ is a killed ${\rm BESQ}^{2-\delta_3}$ flow driven by $\cW^{+}_Y$ on this interval. Recall that $Y$, hence $\zeta^*$, is measurable with respect to $\cW_Y^{*,-}$, thus is independent of $\cW_{Y}^+$. We deduce (i) by the perfect flow property of $\cS^+$ at level $-\zeta^*$ (when it is finite). 
\end{proof}

\subsection{Technical details for the measurability}\label{sec:appendix}
In this appendix we give the proof of Proposition \ref{c:measurability}. 
To this end, we need the measurability of a flow line of $\cS^-$ of \eqref{eq:decomp-}  starting from a random point w.r.t. $\cW_Y^-$, for which we will  use Proposition~\ref{p:decomp}~(ii) combined with the following proposition.

\begin{proposition}\label{p:decomp_coalescence}
    Let $\cD$ be a countable dense subset of $\bR_+\times \bR$. In the setting of Proposition \ref{p:decomp},  almost surely: 
    
    For every $x> r\ge r_0$ and $0\le a< Y_r$ or $a=Y_r>0$,
    \begin{equation}\label{decomp:coalescence}
        \cS_{r,x}^-(a)=\limsup_{\substack{(a_n,r_n)\in \cD\to (a,r) \\ r_n\in (r,x),\ a_n < Y_{r_n}}} \cS_{r_n,x}^-(a_n),
    \end{equation}
    where the definition of $\cS_{r,x}^-(a)$ in \eqref{eq:decomp-} is also extended to $a=Y_r$.
    
    For every $x> r\ge r_0$ and $0<a\le Y_r$,
    \begin{equation}\label{decomp:coalescence-}
        \cS_{r,x}^-(a-)=\liminf_{\substack{(a_n,r_n)\in \cD \to (a,r) \\ r_n\in (r,x),\ a_n<Y_{r_n}}} \cS_{r_n,x}^-(a_n).
    \end{equation}
\end{proposition}
\begin{proof}
    Let us first prove \eqref{decomp:coalescence}. By \eqref{eq:coalescence},
    \[
    \cS^-_{r,x}(a)=\min(\cS_{r,x}(a),Y_x)=\limsup_{\substack{(a_n,r_n)\in \cD\to (a,r) \\ r_n\in (r,x)}} \min(\cS_{r_n,x}(a_n),Y_x).
    \]

    \noindent When $a<Y_r$, the condition $a_n<Y_{r_n}$ is asymptotically satisfied when $(a_n,r_n)\to (a,r)$, hence \eqref{decomp:coalescence} holds in this case.
It remains to prove \eqref{decomp:coalescence} when $a=Y_r$. Let from now on $a=Y_r>0$ with $r\ge r_0$. If $\delta_3\ge 0$, $\cS_{r,s}(a)\ge Y_s$ for all $s\ge r$ by the comparison principle. 
By the perfect flow property and Proposition~\ref{c: no middle bifurcation} applied to $Y$, $Y_s$ is on the left-continuous version of the ${\rm BESQ}(\delta_1\,|_z\, \delta_2)$ flow line starting from $(a=Y_r,r)$; then 
$\cS_{r,s}(a-)\ge Y_s$ by the comparison principle applied to the left-continuous versions of the flows. 
Noticing that 
\[
\liminf_{\substack{(a_n,r_n)\in \cD\to (a,r) \\ r_n\in (r,x)}} \cS_{r_n,x}(a_n)
\le 
\limsup_{\substack{(a_n,r_n)\in \cD\to (a,r) \\ r_n\in (r,x) \\ a_n< \min(\cS_{r,r_n}(a-),Y_{r_n})}} \cS_{r_n,x}(a_n) \le  
\limsup_{\substack{(a_n,r_n)\in \cD\to (a,r) \\ r_n\in (r,x),\ a_n< \cS_{r,r_n}(a-)}} \cS_{r_n,x}(a_n), 
\]
we deduce by \eqref{eq:coalescence-} that
\[
\limsup_{\substack{(a_n,r_n)\in \cD\to (a,r) \\ r_n\in (r,x),\ a_n<Y_{r_n}}} \cS_{r_n,x}(a_n) = \lim_{\substack{(a_n,r_n)\in \cD\to (a,r) \\ r_n\in (r,x) \\ a_n< \min(\cS_{r,r_n}(a-),Y_{r_n})}} \cS_{r_n,x}(a_n)= \cS_{r,x}(a-).
\] 
Taking the minimum with $Y_x$ yields that
\[
\limsup_{\substack{(a_n,r_n)\in \cD\to (a,r) \\ r_n\in (r,x),\ a_n<Y_{r_n}}} \cS^-_{r_n,x}(a_n) =
Y_x=\cS_{r,x}^-(a),
\]
which is equation \eqref{decomp:coalescence}. It remains the case $\delta_3<0$. Note that 
\[
    \cS^-_{r,x}(a)=\limsup_{\substack{(a_n,r_n)\in \cD\to (a,r) \\ r_n\in (r,x)}} \min(\cS_{r_n,x}(a_n),Y_x)\ge \limsup_{\substack{(a_n,r_n)\in \cD\to (a,r) \\ r_n\in (r,x),\ a_n<Y_{r_n}}} \cS^-_{r_n,x}(a_n) \vspace{-0.2cm}
    \]

\noindent so we only have to prove the reverse inequality.
Since $\delta_3<0$, $\cS_{r,s}(a)\le Y_s$ for all $s\ge r$ by the comparison principle and $\cS_{r,s}(a)<Y_s$ for $s>r$ arbitrarily close to $r$. Observing that $\cS_{r_n,x}^-(a_n)=\cS_{r_n,x}(a_n)$ when $a_n<Y_{r_n}$, 
the right-hand side of \eqref{decomp:coalescence} is
\[
\limsup_{\substack{(a_n,r_n)\in \cD\to (a,r) \\ r_n\in (r,x),\ a_n<Y_{r_n}}} \cS_{r_n,x}(a_n) \ge \lim_{\substack{(a_n,r_n)\in \cD\to (a,r) \\ r_n\in (r,x) \\ a_n\in (\cS_{r,r_n}(a),Y_{r_n})}} \cS_{r_n,x}(a_n)= \cS_{r,x}(a)=\cS_{r,x}^-(a)
\]

\noindent by  \eqref{eq:coalescence}. The proof of \eqref{decomp:coalescence} is complete. We prove now \eqref{decomp:coalescence-}. Fix $a\in (0,Y_r]$. We have 
\begin{align*}
    \cS^-_{r,x}(a-)&=\min(\cS_{r,x}(a-),Y_x) =\liminf_{\substack{(a_n,r_n)\in \cD\to (a,r) \\ r_n\in (r,x)}} \min(\cS_{r_n,x}(a_n),Y_x)\\
    & \le \liminf_{\substack{(a_n,r_n)\in \cD \to (a,r) \\ r_n\in (r,x),\ a_n<Y_{r_n}}} \min(\cS_{r_n,x}(a_n), Y_x) \le
    \lim_{\substack{(a_n,r_n)\in \cD \to (a,r) \\ r_n\in (r,x) \\ a_n<\min(\cS_{r,r_n}(a-),Y_{r_n})}} \min(\cS_{r_n,x}(a_n),Y_x)\\
    &=\min(S_{r,x}(a-),Y_x)=\cS_{r,x}^-(a-),
\end{align*}
where the second and the penultimate equalities  come from \eqref{eq:coalescence-}. We obtain  \eqref{decomp:coalescence-} by recalling that $\min(\cS_{r_n,x}(a_n), Y_x)=\cS_{r_n,x}^-(a_n)$ when $a_n<Y_{r_n}$.
\end{proof}

\begin{proposition}\label{p:decompdual_coalescence}
    Let $\cD$ be a dense countable set of $\bR_+\times \bR$. In the setting of Proposition \ref{p:dual S+}, almost surely: for every $-\zeta^*<r < x\le -r_0$ and $0\le a< Y_{-r}$,
    \begin{equation}\label{decompdual:coalescence}
        \cS_{r,x}^-(a)=\limsup_{\substack{(a_n,r_n) \in \cD \to (a,r)\\ r_n\in (r,x),\ a_n<Y_{-r_n}}} \cS_{r_n,x}^-(a_n),
    \end{equation}
    and for every $-\zeta^*<r < x\le -r_0$ and $0< a\le Y_{-r}$, 
    \begin{equation}\label{decompdual:coalescence-}
        \cS_{r,x}^-(a-)=\liminf_{\substack{(a_n,r_n)\in \cD\to (a,r)\\ r_n\in (r,x),\ a_n<Y_{-r_n} }} \cS_{r_n,x}^-(a_n).
    \end{equation}
\end{proposition}
\begin{proof}
    Equation \eqref{decompdual:coalescence} is a consequence of \eqref{eq:coalescence}. Notice that we restricted to $a<Y_{-r}$.
    The proof of \eqref{decompdual:coalescence-} follows the lines of the proof of \eqref{decomp:coalescence-}. 
\end{proof}

\begin{proof}[Proof of Proposition~\ref{c:measurability}]
Since any spindle discovered before time $T=\tau_b^s$ contains  some rational point $(b',s')$ which satisfies $\tau_{b'}^{s'}\le \tau_b^s$, it suffices to fix some point $(b',s')\in (0,\infty)\times \bR$ and consider the spindle containing $(b',s')$, denoted by $\cT=\cT_{(a,r)}$ with bottom point $(a,r)$.   
Recall that $\rF$ denotes the natural filtration of the Brownian motion $B$. Since $\{\tau_{b'}^{s'}\le \tau_b^s\}\in \rF_{\tau_b^s}$,  Proposition \ref{p:markov S} (ii) entails that 
$\{\tau_{b'}^{s'}\le \tau_b^s\}\in\rG^{b,s}_s$.
We argue on the event $\{\tau_{b'}^{s'}\le \tau_b^s\}$ from now on. 
We can also restrict to the case $s'< s$; otherwise, it means that $\cT$ contains only points $(b',s')$ with $s'>s$, which implies by Proposition~\ref{prop:explore} that $\cT$ is fully explored by time $T$ and hence $\mathcal{F_T} = \sigma(W^{T,-})$-measurable.

Note that, since $\tau_{b'}^{s'}\le T=\tau_{b}^{s}$,  for all $x\le s'$, 
\begin{equation}\label{eq:meas_R}
     L(\tau_{b'}^{s'},x)=\cR^*_{-s',-x}(b') \le \cR^*_{-s,-x}(b)=L(T,x).
\end{equation} 
Let 
\begin{align}\label{eq:meas_x0}
s_0 &:=\sup\{x\le s'\colon \cB^*_{-s',-x}(b')= \cR^*_{-s,-x}(b)\}, \\  b_0 &:=\cB^*_{-s',-s_0}(b')= \cR^*_{-s,-s_0}(b)=L(T,s_0). \nonumber
\end{align}
Equation \eqref{eq:meas_R} applied to $x=s'$ yields that $\cR^*_{-s,-s'}(b)\ge b'$, and since $(b',s')$ and $(b,s)$ are fixed, almost surely
$\cR^*_{-s,-s'}(b)>b'$. Together with the fact that $\cR^*_{-s,-I_T}(b)=0$, it  implies that $s_0$ is indeed finite and $s_0\in (I_T,s')$.
We will show later that  $(b_0,s_0)$ almost surely lies in a spindle partially explored at time $T$, and we let $(a_0,r_0)$ be its bottom point. By Proposition~\ref{p:ancestor equiv}, 
 \[
     \cR^*_{-s_0,-r_0}(b_0)=\cB^*_{-s_0,-r_0}(b_0)=a_0
 \]
 and $\cB^*_{-s_0,-x}(b_0)>\cR^*_{-s_0,-x}(b_0)$ for all $x< r_0$. By the perfect flow property, $\cB^*_{-s',-x}(b')=\cB^*_{-s_0,-x}(b_0)$ and $\cR^*_{-s,-x}(b)=\cR^*_{-s_0,-x}(b_0)$. We proved that 
  \begin{equation}\label{eq:a0r0}
     \cR^*_{-s,-r_0}(b)=\cB^*_{-s',-r_0}(b')=a_0
 \end{equation}
 and $\cB^*_{-s',-x}(b')>\cR^*_{-s,-x}(b)$ for all $x<r_0$.
 Then \eqref{eq:meas_R} entails $\cB^*_{-s',-x}(b')>\cR^*_{-s',-x}(b')$ for all $x<r_0$.  As $(a,r)$ is the bottom point of the spindle containing $(b',s')$, another use of Proposition \ref{p:ancestor equiv} yields that
 \begin{equation}\label{proof_meas_r}
 r=\inf\{x\in [r_0,s']:\cR^*_{-s',-x}(b')=\cB^*_{-s',-x}(b')\}
 \end{equation}

\noindent and $a=\cR^*_{-s',-r}(b')=\cB^*_{-s',-r}(b')$. 
Note that $r\neq s_0$ a.s., since by independence one cannot have $\cB^*_{-s',\cdot}(b')-\cR^*_{-s',\cdot}(b')=0$ at the hitting time of $0$ by $\cR^*_{-s,\cdot}(b)-\cB^*_{-s',\cdot}(b')$ (apply Proposition~\ref{p:decomp} with $\cS=\cR^*$ and $Y=\cB^*_{-s',\cdot}(b')$).

If $r>s_0$ (hence $(a,r)\neq (a_0,r_0)$), one has $a=\cB^*_{-s',-r}(b')<\cR^*_{-s,-r}(b)=L(T,r)$ by definition of $s_0$, so $T\ge \tau_a^r$ and $\cT$ has been fully explored at time $T$. 

If $r<s_0$, then $(b_0,s_0)=(\cB^*_{-s',-s_0}(b'),s_0) \in \cT$ by Lemma \ref{lem:cT} (ii), hence $(a,r)=(a_0,r_0)$; as we admitted that $\cT_{(a_0,r_0)}$ is partially explored at time $T$,  then $\cT=\cT_{(a,r)}=\cT_{(a_0,r_0)}$ is partially explored. 
Let
\begin{align*}
    r_1:=\inf\{x>r_0\colon \cG^T_x= L(T,x)\},\quad
    a_1:=\cG^T_{r_1}=L(T,r_1).
\end{align*}

\noindent By Remark \ref{r:green line} (i), the spindle $\cT_{(a_0,r_0)}$ contains $(b,s)$ if and only if $r_1>s$. See Figure \ref{f:proof_measurability} for some illustrations of the different cases.

 \begin{figure}[htbp]
\centering
       \scalebox{0.4}{ 
        \def\svgwidth{0.8\columnwidth}
        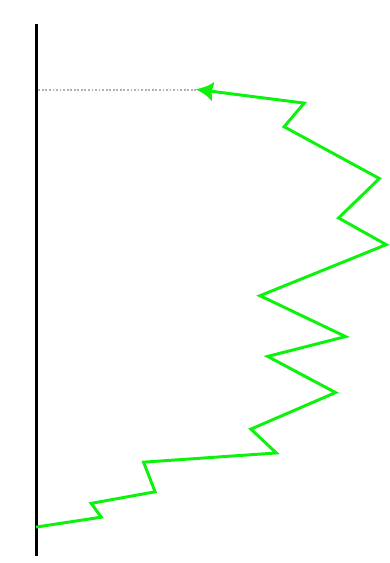
    }
    \hfill
     \scalebox{0.4}{ 
        \def\svgwidth{0.8\columnwidth}
        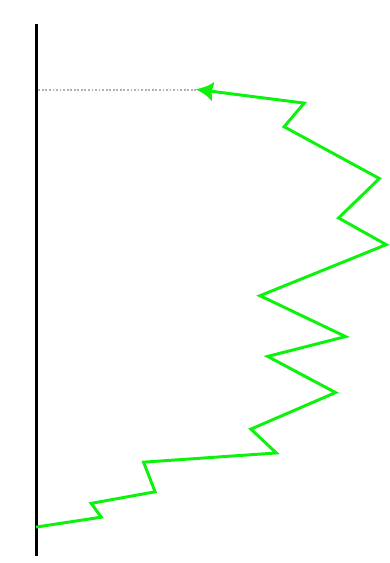
    }
    \scalebox{0.4}{ 
        \def\svgwidth{0.8\columnwidth}
        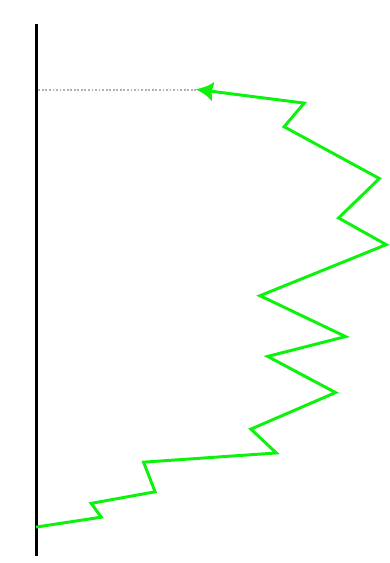
    }
\caption{The pictures illustrate the cases $(a_0,r_0)\neq (a,r)$ (left), $(a_0,r_0)=(a,r),\, r_1<s$ (middle) and $(a_0,r_0)= (a,r),\, r_1>s$ (right) corresponding respectively to whether the spindle $\cT$ is fully or only partially explored and whether $(b,s)\in \cT_{(a_0,r_0)}$. }
\label{f:proof_measurability}
\end{figure}

\bigskip

We prove below the following statements.
\begin{enumerate}[(I)]
    \item The bottom point $(a,r)$ is $\rG_s^{b,s}$-measurable.
    \item The entrance time $s_{\cT}=\sigma_a^r$ and the process $L(\sigma_a^r,x)=\cR_{r,x}(a-)$, $x\ge r$ are $\rG_s^{b,s}$-measurable.
    \item  If $(a,r)\neq (a_0,r_0)$, the process $(\cB_{r,x}(a),\,x\ge r)$ is $\rG_s^{b,s}$-measurable.
    \item If $(a,r)= (a_0,r_0)$, the process $(\cB_{r,x}(a),\,x\in [r,r_1])$ is $\rG^{b,s}$-measurable and even $\rG^{b,s}_s$-measurable if $r_1<s$.
    \item  If $(a,r)= (a_0,r_0)$ and $r_1<s$, the process $(\cB_{r,x}(a),\,x\ge r_1)$ is $\rG_s^{b,s}$-measurable.
    \item  If $(a,r)= (a_0,r_0)$ and $r_1>s$, the process $(\cB_{r,x}(a),\,x\ge r_1)$ is $\rG^{b,s}$-measurable.
    \item The first exit time $t_{\cT}$ is $\rG^{b,s}$-measurable.
\end{enumerate}

\noindent These statements imply the proposition. Indeed, the measurability of $(\cR_{r,x}(a-),\, x\ge r)$ ,  $(\cB_{r,x}(a),\, x\ge r)$, $s_{\cT}=\sigma_a^r$  and  $t_{\cT}$  w.r.t.\ $\mathscr{G}^{b,s}$ are respectively consequences of (II), (III+IV+V+VI), (II) and (VII). Statements  (i) and (ii) of the proposition are consequences respectively of (I) and  (II). The spindles in  (iii) are either spindles that were fully explored, or spindles that are partially explored except the one containing $(b,s)$. Their measurability is guaranteed respectively  by  (I+II+III) and (I+II+IV+V). We turn to the proofs of (I)--(VII). In the course of the proof of (I), we will prove the claim that $(b_0,s_0)$ is in a partially explored spindle $\cT_{(a_0,r_0)}$.

(I)   Recalling that $\tau_{b'}^{s'}\le T$, the processes $L(\tau_{b'}^{s'},\cdot)$ and $L(T,\cdot)$ are measurable w.r.t. $\rF_{T}=\sigma(W^{T,-})$. Then both $s_0$ and the process  $(\cB^*_{-s',-x}(b'),\, x\in [s_0,s'])$ are measurable w.r.t.\ $W^{T,-}$  by Proposition \ref{p:decomp} (ii) applied to $\cS=\cB^*$ and $Y=\cR^*_{-s,\cdot}(b)$ (the fact that $Y$ is killed does not matter). By Proposition \ref{p:green}, $\cY$ is a  ${\rm BESQ}(2-\delta\, |_{s}\, -\delta)$ flow line driven by $W^T$. 
Since $s_0$ is $W^{T,-}$-measurable, it is independent of $W^T$. We notice that $s_0<s$ and hence obtain $\cY_{s_0}>0$ a.s., i.e. $\cG^T_{s_0}>L(T,s_0)=b_0$. 
By definition of $\cG^T$, it implies that $(b_0,s_0)$ is in some spindle which is partially explored at time $T$ indeed. By the representation in Remark \ref{r:green line} (i), its bottom point $(a_0,r_0)$ satisfies
\begin{align*}\vspace{-0.1cm}
r_0 =\sup\{x \in(I_T,s_0):\cG^T_x=L(T,x)\}, \qquad 
a_0 = L(T,r_0).\vspace{-0.1cm}
\end{align*}
 By Proposition \ref{p:decomp} (ii) applied to $Y=\cY$, $r_0$ is measurable w.r.t.\ $(s_0,I_T)$  and $\widehat{W}^{T,-}_{|(I_T,s)}$. We deduce that $(a_0,r_0)$ is measurable w.r.t.\ $\rG^{b,s}_s$. 
If  $\cR^*_{-s',-r_0}(b')=\cR^*_{-s,-r_0}(b)$, then equations \eqref{eq:a0r0} and \eqref{proof_meas_r} show that $r=r_0$, hence we have $(a,r)=(a_0,r_0)$. Otherwise, we know that $r>s_0$ therefore \eqref{proof_meas_r} becomes\vspace{-0.1cm}
\[
 r=\inf\{x\in [s_0,s']:\cR^*_{-s',-x}(b')=\cB^*_{-s',-x}(b')\}.\vspace{-0.2cm}
 \]

\noindent It shows the measurability of $(a,r)$ w.r.t.\ $\rG_s^{b,s}$.

(II)  Since $(a,r)$ is $\rG^{b,s}_s$-measurable and $\sigma_a^r<T$, we have that $\sigma_a^r$ and $L(\sigma_a^r,x)=\cR_{r,x}(a-)$, $x\ge r$ are measurable w.r.t.\ $\rG^{b,s}_s$ by Proposition \ref{p:markov S} (ii).

(III)  In the case $(a,r)\neq (a_0,r_0)$, $r>s_0$ a.s.  hence $a=\cB^*_{-s',-r}(b')<\cR^*_{-s,-r}(b)$ by \eqref{eq:meas_x0}.  The measurability of $(\cB_{r,x}(a),\,x\in [r,s])$ (starting from a random point) w.r.t.\ $\rG^{b,s}_s$ follows from equation~\eqref{decompdual:coalescence} and Proposition~\ref{p:dual S+}~(iii) applied to $Y=\cR^*_{-s,\cdot}(b)$ and $\cS=\cB$. 
Recall that $\cB\le \cR$ by the comparison principle. Moreover $\cB_{r,s}(a)\le \cR^*_{-s,-s}(b)=b$ by Proposition~\ref{p:dual S+}~(ii). 
The measurability of  $\cB_{r,x}(a)=\cB_{s,x}(\cB_{r,s}(a))$ for $x\ge s$ follows from   Proposition~\ref{p:decomp}~(ii) and equation \eqref{decomp:coalescence} with $\cS=\cB$ and $Y=\cR_{s,\cdot}(b)$.

 (IV)   By definition of $\cG^T$ in Definition \ref{green}, we have $\cB_{r_0,x}(a_0)=\cG^T_x$ for $x\in [r_0,r_1]$. Since $L(T,\cdot)$ is measurable w.r.t.\ $W^{T,-}$ by Proposition~\ref{p:markov S}~(ii) and $\cY=\cG^T-L(T,\cdot)$ is driven by\ $\widehat{W}^{T,-}$ by Proposition~\ref{p:decomp}~(ii) applied to $Y=\cY$, we deduce that $(\cB_{r,x}(a),\, x\in [r,r_1])$ is measurable w.r.t.\ $\rG^{b,s}$  and $(\cB_{r,x}(a),\, x\in [r,\min(r_1,s)])$  is measurable w.r.t.\ $\rG^{b,s}_s$. Note that the event $\{r_1<s\}$ is also measurable w.r.t.\ $\rG^{b,s}_s$.

 (V) If $(a,r)= (a_0,r_0)$ and $r_1<s$,   $a_1=L(T,r_1)=\cR^*_{-s,-r_1}(b)>0$, so by Proposition \ref{c: no middle bifurcation}, $\cB_{r_1,x}(a_1-)=\cB_{r_1,x}(a_1)$. By Proposition~\ref{p:dual S+}~(iii) and equation \eqref{decompdual:coalescence-} applied to $\cS=\cB$ and $Y=\cR^*_{-s,\cdot}(b)$, we have that $(\cB_{r,x}(a),\,x\in [r_1,s])=(\cB_{r_1,x}(a_1),\, x\in [r_1,s])$  is measurable w.r.t.\ $\rG^{b,s}_s$. The measurability extends to $x\ge s$ as in (III). 

(VI)   Since $r_1>s$, $a_1=L(T,r_1)=\cR_{s,r_1}(b)$. We have either $a_1=\cB_{r,r_1}(a)=\cR_{s,r_1}(b)=0$ hence $\cB_{r,x}(a)=0$ for all $x\ge r_1$; or $a_1=\cB_{r,r_1}(a)=\cR_{s,r_1}(b)>0$ and we apply  Proposition \ref{p:decomp} (ii) and equation \eqref{decomp:coalescence} with $\cS=\cB$ and $Y=\cR_{s,\cdot}(b)$ to get the measurability of $(\cB_{r,x}(a),\, x\ge r_1)=(\cB_{r_1,x}(a_1),\, x\ge r_1)$ w.r.t.\ $\rG^{b,s}$.

(VII) Let $(c,z)$ be the top point  of $\cT$, which is measurable w.r.t. $\rG^{b,s}$ by the statements (I)-(VI). By Proposition \ref{prop:explore} (ii), $t_{\cT}=\inf\{t\ge \sigma_a^r\colon X_t=z\}$. Therefore $\min(t_{\cT},T)$ is  measurable w.r.t. $(\sigma_a^r,z)$ and $\rF_T$, hence w.r.t. $\rG^{b,s}$. It already deals with the case $t_{\cT}<T$. In the remaining case, it means that $(b,s)$ is in the spindle $\cT$ and the explorer has not reached the top of the spindle yet. For $y\in (s,z)$, $a_y:=L(T,y)<\cB_{r,y}(a)=\cG^T_y$ by definition of $\cG^T$. In the notation of Proposition \ref{p:markov S} (iv), we have $L(\tau_{a_y}^y,x)-L(T,x)=\cR^*_{-y,-x}(a_y)-L^*(T,-x)$ for $x\le y$.  By Proposition \ref{p:markov S} (iv), $\cR^*-L^*(T,\cdot)$ is on $(-\infty,-I_T)$ a ${\rm BESQ}(2\,|_{-s} 0)$ flow driven by $-W^{T,*}$. By Proposition \ref{p:dual S+} (iii) and equation \eqref{decomp:coalescence} applied to $Y=\cG^T-L(T,\cdot)$ and $\cS=\cR^*-L^*(T,\cdot)$, the process $L(\tau_{a_y}^y,x)-L(T,x)$, $x \in (I_T,y)$ is measurable w.r.t.\ $\rG^{b,s}$.    
For $x\ge y$, $L(\tau_{a_y}^y,x)=\cR_{y,x}(a_y)=\cR_{s,x}(b)=L(T,x)$ by the perfect flow property, therefore it is also measurable w.r.t. $ \rG^{b,s}$ by Proposition \ref{p:markov S} (ii). By the occupation times formula, $\tau_{a_y}^y=\int_{I_T}^{+\infty} L(\tau_{a_y}^y,x)\dd x$ is $\rG^{b,s}$-measurable. 
  Making $y\to z$, hence $(a_y,y)\to (c,z)$, yields that $t_{\cT}$ is also $\rG^{b,s}$-measurable. This completes the proof. 
\end{proof}

\bigskip
\begin{paragraph}{Acknowledgments} 
    The authors would like to thank William Da Silva and Matthias Winkel for fruitful discussions. 
\end{paragraph}
\begin{paragraph}{Funding} 
    This work was supported in part by the National Key R\&D Program of China (grant 2022YFA1006500) and National Natural Science Foundation of China (grants QXH1411004, 12288201 and 12301169). 
\end{paragraph}
\bibliographystyle{alpha}
\bibliography{CRP}

\end{document}